\documentclass{amsart}

\usepackage{tabularx}
\usepackage{graphicx}
\usepackage{multicol}
\usepackage{enumerate}
\usepackage[all]{xy}
\xyoption{poly}
\xyoption{arc}
\xyoption{curve}

\CompileMatrices

\begin{document}

\renewcommand{\theequation}{\thesection.\arabic{equation}}
\newcommand{\nc}{\newcommand}

\nc{\pr}{\noindent{\em Proof. }} \nc{\g}{\mathfrak g}
\nc{\n}{\mathfrak n} \nc{\opn}{\overline{\n}}\nc{\h}{\mathfrak h}
\renewcommand{\b}{\mathfrak b}
\nc{\Ug}{U(\g)} \nc{\Uh}{U(\h)} \nc{\Un}{U(\n)}
\nc{\Uopn}{U(\opn)}\nc{\Ub}{U(\b)} \nc{\p}{\mathfrak p}
\renewcommand{\l}{\mathfrak l}
\nc{\z}{\mathfrak z} \renewcommand{\h}{\mathfrak h}
\nc{\m}{\mathfrak m}
\renewcommand{\k}{\mathfrak k}
\nc{\opk}{\overline{\k}}
\nc{\opb}{\overline{\b}}
\nc{\e}{{\epsilon}}
\nc{\ke}{{\bf k}_\e}
\nc{\Hk}{{\rm Hk}^{\gr}(A,A_0,\e )}
\nc{\gr}{\bullet}
\nc{\ra}{\rightarrow}
\nc{\Alm}{A-{\rm mod}}
\nc{\DAl}{{D}^-(A)}
\nc{\HA}{{\rm Hom}_A}

\newtheorem{theorem}{Theorem}{}
\newtheorem{lemma}[theorem]{Lemma}{}
\newtheorem{corollary}[theorem]{Corollary}{}
\newtheorem{conjecture}[theorem]{Conjecture}{}
\newtheorem{proposition}[theorem]{Proposition}{}
\newtheorem{axiom}{Axiom}{}
\newtheorem{remark}{Remark}{}
\newtheorem{example}{Example}{}
\newtheorem{exercise}{Exercise}{}
\newtheorem{definition}{Definition}{}

\renewcommand{\thetheorem}{\thesection.\arabic{theorem}}

\renewcommand{\thelemma}{\thesection.\arabic{lemma}}

\renewcommand{\theproposition}{\thesection.\arabic{proposition}}

\renewcommand{\thecorollary}{\thesection.\arabic{corollary}}

\renewcommand{\theremark}{\thesection.\arabic{remark}}

\renewcommand{\thedefinition}{\thesection.\arabic{definition}}

\title{Strictly transversal slices to conjugacy classes in algebraic groups}

\author{A. Sevostyanov}

\address{ Institute of Pure and Applied Mathematics,
University of Aberdeen \\ Aberdeen AB24 3UE, United Kingdom \\e-mail: a.sevastyanov@abdn.ac.uk}

\begin{abstract}
We show that for every conjugacy class $\mathcal{O}$ in a connected semisimple algebraic group $G$ over an algebraically closed field of characteristic good for $G$ one can find a special transversal slice $\Sigma$ to the set of conjugacy classes in $G$ such that $\mathcal{O}$ intersects $\Sigma$ and ${\rm dim}~\mathcal{O}={\rm codim}~\Sigma$.
\end{abstract}


\maketitle

\hfill {\hsize = 8cm \vtop{\noindent \em To Victor Kac on the occasion of his 70th birthday.}}


\pagestyle{myheadings}
\markboth{A. SEVOSTYANOV}{STRICTLY TRANSVERSAL SLICES TO CONJUGACY CLASSES}


\setcounter{equation}{0}

\section{Introduction}

Let $G$ be a
connected finite--dimensional semisimple algebraic group over an algebraically closed field $\bf k$ of characteristic good for $G$, ${\frak g}$ its Lie algebra. It is well known that if ${\rm char}~{\bf k}=0$ or ${\rm char}~{\bf k}>4h-2$, where $h$ is the Coxeter number of $\g$, then for any nilpotent element $e\in \g$ one can construct a special transversal slice $S(e)$, called the Slodowy slice, such that $S(e)$ is transversal to the set of adjoint orbits in $\g$, the adjoint orbit $\mathcal{O}_e$ of $e$ intersects $S(e)$, and ${\rm dim}~\mathcal{O}_e={\rm codim}~S(e)$ (see \cite{SL}). In this paper transversal slices with similar properties are defined for arbitrary conjugacy classes in the algebraic group $G$.

A general scheme for constructing transversal slices to conjugacy classes in algebraic groups was suggested in \cite{S6}. The slices defined there are associated to conjugacy classes in the corresponding Weyl group $W$ (see Section \ref{trans} for details). The definition given in \cite{S6} is a deep generalization of the Steinberg cross--section to the set of regular conjugacy classes in $G$ (see \cite{St2}). In fact for each conjugacy class $\mathcal{O}_s\subset W$ of an element $s\in W$ the corresponding slice $\Sigma_s$ depends also on the ordering of the eigenvalues of $s$ which are not equal to $1$ and have non--negative imaginary parts, in the complex reflection representation (note that $s$ is of finite order, and hence all these eigenvalues are roots of unity). In this paper we show that for each element $g\in G$ one can find a conjugacy class $\mathcal{O}_s\subset W$ and an ordering of the eigenvalues of $s$ with non--negative imaginary parts such that the conjugacy class $\mathcal{O}_g$ of $g$ in $G$ intersects $\Sigma_s$ and ${\rm dim}~\mathcal{O}_g={\rm codim}~\Sigma_s$.

In order to solve this problem we use a partition of $G$ the strata of which $G_{\mathcal{O}_s}$, labeled by a special set $C(W)$ of conjugacy classes $\mathcal{O}_s$ in $W$, are unions of conjugacy classes of the same dimension,
$$
G=\bigcup_{\mathcal{O}_s\in C(W)}G_{\mathcal{O}_s}.
$$
This remarkable partition was recently introduced by Lusztig in \cite{L1}. We recall the definition of the Lusztig partition and the relevant combinatorics in Section \ref{luspart}.

It turns out that if the eigenvalues of $\epsilon_1,\ldots,\epsilon_k$ of  $s$, which are not equal to $1$ and have non--negative imaginary parts, are ordered so that
$$
0<\arg \epsilon_1\leq\arg \epsilon_2\leq\ldots\leq\arg \epsilon_{k},
$$
where $\arg$ stands for the principal argument of a complex number, then all conjugacy classes in the stratum $G_{\mathcal{O}_s}$ intersect the corresponding transversal slice $\Sigma_s$, and for each conjugacy class $\mathcal{O}\in G_{\mathcal{O}_s}$
\begin{equation}\label{sttr}
{\rm dim}~\mathcal{O}={\rm codim}~\Sigma_s.
\end{equation}
The first property is a consequence of inclusions of Bruhat cells in the Bruhat order. But condition (\ref{sttr}) is checked by straightforward calculation in case of all simple algebraic groups and in case of all strata (see Theorem \ref{mainth}). In case of root systems of exceptional types a computer program is used for that purpose (see Appendix 1).

Note that in case when $s$ is elliptic, i.e. $s$ acts without fixed points in the reflection representation of $W$, our results imply that $\Sigma_s=\dot{s}N_s$  and  ${\rm dim}~\Sigma_s=\underline{l}(s)$, where $\underline{l}(s)$ is the length of a representative $s\in \mathcal{O}_s$ of minimal possible length with respect to a system $\Gamma$ of positive simple roots, $N_s=\{ v \in N|\dot{s}v\dot{s}^{-1}\in \overline{N} \}$, $N$ is the unipotent radical of $G$ corresponding to $\Gamma$, $\overline{N}$ is the opposite unipotent radical and $\dot{s}$ is a representative of $s$ in $G$  (see the second part of the proof of Theorem \ref{mainth}). Therefore from Theorem 0.7 in \cite{L4} one immediately deduces that $\Sigma_s$ coincides with the cross--section associated to $\mathcal{O}_s$ in \cite{Lus}.

{\bf Acknowledgements}

The author is grateful to Dave Benson for his help with the MAGMA software and to Institut des Hautes \'{E}tudes Scientifiques where this work was completed.


\setcounter{equation}{0}
\setcounter{theorem}{0}

\section{Notation}\label{notation}

Fix the notation used throughout the text.
Let $G$ be a
connected finite--dimensional semisimple algebraic group over an algebraically closed field $\bf k$ of characteristic good for $G$, ${\frak g}$ the complex semisimple Lie algebra of the same type as $G$, $G_p$ a connected finite--dimensional semisimple algebraic group of the same type as $G$ over an algebraically closed field of characteristic exponent $p$. Fix a Cartan subalgebra ${\frak h}\subset {\frak
g}$ and let $\Delta$ be the set of roots of $\left( {\frak g},{\frak h}
\right)$.  Let $\alpha_i,~i=1,\ldots l,~~l={\rm rank}~{\frak g}$ be a system of
simple roots, $\Delta_+=\{ \beta_1, \ldots ,\beta_N \}$
the set of positive roots.
Let $H_1,\ldots ,H_l$ be the set of simple root generators of $\frak h$.

Let $a_{ij}$ be the corresponding Cartan matrix,
and let $d_1,\ldots , d_l$, $i=1,\ldots,l$ be coprime positive integers such that the matrix
$b_{ij}=d_ia_{ij}$ is symmetric. There exists a unique non--degenerate invariant
symmetric bilinear form $\left( ~,~\right)$ on ${\frak g}$ such that
$(H_i , H_j)=d_j^{-1}a_{ij}$. It induces an isomorphism of vector spaces
${\frak h}\simeq {\frak h}^*$ under which $\alpha_i \in {\frak h}^*$ corresponds
to $d_iH_i \in {\frak h}$. We denote by $\alpha^\vee$ the element of $\frak h$ that
corresponds to $\alpha \in {\frak h}^*$ under this isomorphism.
The induced bilinear form on ${\frak h}^*$ is given by
$(\alpha_i , \alpha_j)=b_{ij}$. We shall always identify $\h$ and $\h^*$ by means of the form $\left( ~,~\right)$.

Let $W$ be the Weyl group of the root system $\Delta$. $W$ is the subgroup of $GL({\frak h})$
generated by the fundamental reflections $s_1,\ldots ,s_l$,
$$
s_i(h)=h-\alpha_i(h)H_i,~~h\in{\frak h}.
$$
The action of $W$ preserves the bilinear form $\left( ~,~\right)$ on $\frak h$.
We denote a representative of $w\in W$ in $G$ by
$\dot{w}$. For $w\in W, g\in G$ we write $w(g)=\dot{w}g\dot{w}^{-1}$.
For any root $\alpha\in \Delta$ we also denote by $s_\alpha$ the corresponding reflection.
For every element $w\in W$ one can introduce the set $\Delta_w=\{\alpha \in \Delta_+: w(\alpha)\in -\Delta_+\}$, and the number of the elements in the set $\Delta_w$ is equal to the length $\underline{l}(w)$ of the element $w$ with respect to the system of simple roots in $\Delta_+$.

Let ${{\frak b}_+}$ be the Borel subalgebra associated to $\Delta_+$ and ${\frak b}_-$
the opposite Borel subalgebra; let ${\frak n}_+=[{{\frak b}_+},{{\frak b}_+}]$ and $%
{\frak n}_-=[{\frak b}_-,{\frak b}_-]$ be their
nilradicals. Let $B_+$ be the Borel subgroup in $G$ which corresponds to $\Delta_+$, $H$ the maximal torus in $B_+$, $B_-$ the opposite Borel subgroup, and $N_{\pm}$ the unipotent radicals of $B_\pm$. Thus $W$ is the Weyl group of the pair $(\g,\h)$ and of the pair $(G,H)$, $W=W(\g,\h)=W(G,H)$, and $\Delta$ is the root system of the pair $(G,H)$, $\Delta=\Delta(G,H)$. This notation will also be used when $G$ is reductive. $\Delta_+$ can also be defined as the system of positive roots associated to $\b_+$ or to $B_+$.

In this paper we denote by $\mathbb{N}$ the set of nonnegative integer numbers, $\mathbb{N}=\{0,1,\ldots \}$.


\section{Transversal slices to conjugacy classes in algebraic groups}\label{trans}

\setcounter{equation}{0}
\setcounter{theorem}{0}

In this section, following \cite{S6}, we recall the general definition of transversal slices to conjugacy classes in algebraic groups. The slices are introduced with the help of systems of positive roots associated to Weyl group elements.
We start with the definition of these systems of positive roots.

Let $s$ be an element of the Weyl group $W$ of the pair $(\g,\h)$ and $\h_{\mathbb{R}}$ the real form of $\h$, the real linear span of simple coroots in $\h$. The set of roots $\Delta$ is a subset of the dual space $\h_\mathbb{R}^*$. The restriction of the canonical bilinear form to $\h_{\mathbb{R}}$ induces a non--degenerate Weyl group invariant bilinear form on it. Let $\h'_{\mathbb{R}}$ be the orthogonal complement in $\h_{\mathbb{R}}$, with respect to the canonical bilinear form, to the subspace of $\h_{\mathbb{R}}$ fixed by the natural action of $s$ on $\h_{\mathbb{R}}$, $\h'\subset \h$ the complexification of $\h'_{\mathbb{R}}$. Let $\h'^*$ be the image of $\h'$ in $\h^*$ under the identification $\h^*\simeq \h$ induced by the canonical bilinear form on $\g$ and $P_{{\h'}^*}$ the orthogonal projection operator onto ${{\h'}^*}$ in $\h^*$, with respect to canonical bilinear form.
Now recall that by Theorem C in \cite{C} $s$ can be represented as a product of two involutions,
\begin{equation}\label{sinv}
s=s^1s^2,
\end{equation}
where $s^1=s_{\gamma_1}\ldots s_{\gamma_n}$, $s^2=s_{\gamma_{n+1}}\ldots s_{\gamma_{l'}}$, the roots in each of the sets $\gamma_1, \ldots \gamma_n$ and ${\gamma_{n+1}}\ldots {\gamma_{l'}}$ are positive and mutually orthogonal, and
the roots $\gamma_1, \ldots \gamma_{l'}$ form a linear basis of $\h'$.

The Weyl group element $s$ naturally acts on $\h_{\mathbb{R}}$ as an orthogonal transformation with respect to the scalar product induced by the canonical bilinear form on $\g$. Now we recall some results of \cite{Car}, Sect. 10.4 on the spectral decomposition for the action of $s$ on $\h_{\mathbb{R}}$.

Let $f_1,\ldots,f_{l'}$ be the vectors of unit length in the directions of $\gamma_1, \ldots \gamma_{l'}$, and $\widehat{f}_1,\ldots,\widehat{f}_{l'}$ the basis of $\h_{\mathbb{R}}'$ dual to $f_1,\ldots,f_{l'}$. Let $M$ be the $l'\times l'$ symmetric matrix with real entries $M_{ij}=(f_i,f_j)$. $I-M$ is also a symmetric real matrix, and hence it is diagonalizable and has real eigenvalues.

The following proposition gives a recipe for constructing a spectral decomposition for the action of the orthogonal transformation $s$ on $\h_{\mathbb{R}}$.

\begin{proposition}
Let $\lambda\neq 0,\pm 1$ be a (real) eigenvalue of the symmetric matrix $I-M$, and $u\in \mathbb{R}^{l'}$ a corresponding non--zero real eigenvector with components $u_i$, $i=1,\ldots ,l'$. Let $a_u,b_u\in \h_{\mathbb{R}}$ be defined by
\begin{equation}\label{ab}
a_u = \sum_{k=1}^n u_i\widehat{f}_i,~b_u = \sum_{k=n+1}^{l'} u_i\widehat{f}_i.
\end{equation}

Then the angle $\theta$ between $a_u$ and $b_u$ satisfies $\cos \theta=\lambda$.

The plane $\h_\lambda \subset \h_{\mathbb{R}}$ spanned by $a_u$ and $b_u$ is invariant with respect to the involutions $s^{1,2}$, $s^1$ acts on $\h_\lambda$ as the reflection in the line spanned by $b_u$, and $s^2$ acts on $\h_\lambda$ as the reflection in the line spanned by $a_u$. The orthogonal transformation $s=s^1s^2$ acts on $\h_\lambda$ as a rotation through the angle $2\theta$.

In particular, if $\lambda\neq 0,\pm 1$ is an eigenvalue of $I-M$ then $-\lambda$ is also an eigenvalue of $I-M$, and
if $\lambda\neq \mu$ are two positive eigenvalues of $I-M$, $\lambda,\mu \neq 1$ then the planes $\h_\lambda$ and $\h_\mu$ are mutually orthogonal.

Moreover, let $\lambda\neq 0,\pm 1$ be an eigenvalue of $I-M$ of multiplicity greater than $1$, and $u^k \in \mathbb{R}^{l'}$, $k=1,\ldots, ~{\rm mult}~\lambda$ a basis of the eigenspace corresponding to $\lambda$. If the basis $u^k$ is orthonormal with respect to the standard scalar product on $\mathbb{R}^{l'}$ then the corresponding planes $\h_\lambda^k$ defined with the help of $u^k$, $k=1,\ldots, {\rm mult}~\lambda$ are mutually orthogonal.

$\lambda=\pm 1$ are not eigenvalues of $I-M$ , and if $\lambda= 0$ is an eigenvalue of $I-M$, then there is a basis $u^k \in \mathbb{R}^{l'}$, $k=1,\ldots, ~{\rm mult}~0$ of the eigenspace corresponding to $0$ orthonormal with respect to the standard scalar product on $\mathbb{R}^{l'}$ and such that the corresponding non--zero elements $a_{u^k}$, $b_{u^k}$ are all mutually orthogonal. Moreover, $s^1a_{u^k}=-a_{u^k}$, $s^2a_{u^k}=a_{u^k}$, $s^1b_{u^k}=b_{u^k}$, $s^2b_{u^k}=-b_{u^k}$ for non--zero elements $a_{u^k}$, $b_{u^k}$. In particular, for  non--zero elements $a_{u^k}$, $b_{u^k}$ we have $sa_{u^k}=-a_{u^k}$, $sb_{u^k}=-b_{u^k}$, and non--zero elements $a_{u^k}$, $b_{u^k}$ is a basis of the subspace of $\h_{\mathbb{R}}$ on which $s$ acts by multiplication by $-1$.
\end{proposition}

\begin{proof}
All statements of this proposition, except for the last two parts, are proved by repeating the arguments given in the proofs of Lemma 10.4.2, Proposition 10.4.3 in \cite{Car} and using the spectral theory of orthogonal transformations.

For the last two statements one has to use some calculations from the proof of Lemma 10.4.3 in \cite{Car}.
More precisely, by definition the matrix $M$ can be written in a block form,
\begin{equation}\label{m}
M=\left(
  \begin{array}{cc}
    I_n & A \\
    A^\top & I_{l'-n} \\
  \end{array}
\right),
\end{equation}
where $A$ ia a $n\times (l'-n)$ matrix, $A^\top$ is the transpose to $A$, $I_n$ and $I_{l'-n}$ are the unit matrixes of sizes $n$ and $l'-n$. $M^{-1}$ is also symmetric and has a similar block form,
\begin{equation}\label{m-1}
M^{-1}=\left(
  \begin{array}{cc}
    B & C \\
    C^\top & D \\
  \end{array}
\right),~B=B^\top,~D=D^\top,
\end{equation}
with the entries $M^{-1}_{ij}=(\widehat{f}_i,\widehat{f}_j)$.

For any vector $u\in \mathbb{R}^{l'}$ we introduce its $\mathbb{R}^{n}$ and $\mathbb{R}^{l'-n}$ components $\widetilde{u}$ and $\widetilde{\widetilde{u}}$ in a similar way,
\begin{equation}\label{udec}
u=\left(
    \begin{array}{c}
      \widetilde{u} \\
      \widetilde{\widetilde{u}} \\
    \end{array}
  \right).
\end{equation}
We shall consider both $\widetilde{u}$ and $\widetilde{\widetilde{u}}$ as elements of $\mathbb{R}^{l'}$ using natural embeddings $\mathbb{R}^{n}, \mathbb{R}^{l'-n} \subset \mathbb{R}^{l'}$ associated to decomposition (\ref{udec}).

If $u$ is a non--zero eigenvector of $I-M$ corresponding to an eigenvalue $\lambda\neq 0,\pm 1$ then the equation $(I-M)u=\lambda u$ gives
\begin{equation}\label{ueq}
-A\widetilde{\widetilde{u}}=\lambda\widetilde{u},~-A^\top\widetilde{u}=\lambda\widetilde{\widetilde{u}}.
\end{equation}
Since $M^{-1}M=I$ one has
\begin{equation}\label{mmi}
BA+C=0,~C^\top+DA^\top=0.
\end{equation}
Multiplying the first and the second equations in (\ref{ueq}) from the left by $B$ and $D$, respectively, and using (\ref{mmi}) we obtain that
\begin{equation}\label{ident}
C\widetilde{\widetilde{u}}=\lambda B\widetilde{u},~C^\top\widetilde{u}=\lambda D\widetilde{\widetilde{u}}.
\end{equation}

Now if $u^{1,2}$ are two non--zero eigenvectors of $I-M$ corresponding to an eigenvalue $\lambda\neq 0,\pm 1$ then by (\ref{m-1}) we have
\begin{equation}\label{1}
(a_{u^1},a_{u^2})=\sum_{i,j=1}^nu_i^1u_j^2(\widehat{f}_i,\widehat{f}_j)=\sum_{i,j=1}^nu_i^1u_j^2B_{ij}=
\widetilde{u}^1\cdot B\widetilde{u}^2,
\end{equation}
where $\cdot$ stands for the standard scalar product in $\mathbb{R}^{l'}$.

From (\ref{ident}) and the last formula we also obtain that
\begin{equation}\label{2}
(a_{u^1},a_{u^2})=\widetilde{u}^1\cdot B\widetilde{u}^2=\frac{1}{\lambda}\widetilde{u}^1\cdot C\widetilde{\widetilde{u}}^2= \frac{1}{\lambda}C^\top\widetilde{u}^1\cdot \widetilde{\widetilde{u}}^2=D\widetilde{\widetilde{u}}^1\cdot \widetilde{\widetilde{u}}^2=(b_{u^1},b_{u^2}).
\end{equation}

Similarly,
\begin{equation}\label{3}
(a_{u^1},b_{u^2})=\lambda(a_{u^1},a_{u^2})=\widetilde{u}^1\cdot C\widetilde{\widetilde{u}}^2,~(b_{u^1},a_{u^2})=\lambda(a_{u^1},a_{u^2})=\widetilde{\widetilde{u}}^1\cdot C^\top{\widetilde{u}}^2
\end{equation}

Now (\ref{1}), (\ref{2}), (\ref{3}) and the identity $M^{-1}u^2=\frac{1}{1-\lambda}u^2$ yield
$$
(a_{u^1}+b_{u^1},a_{u^2}+b_{u^2})=2(a_{u^1},a_{u^2})(\lambda + 1)=u^1\cdot M^{-1}u^2=\frac{1}{1-\lambda}u^1 \cdot u^2.
$$
Thus if $u^{1,2}$ are mutually orthogonal $a_{u^1}$ and $a_{u^2}$ are also mutually orthogonal, and from (\ref{2}) and (\ref{3}) we obtain that $b_{u^1}$ and $b_{u^2}$, $a_{u^1}$ and  $b_{u^2}$, $a_{u^2}$ and $b_{u^1}$  are mutually orthogonal. Therefore the planes spanned by $a_{u^1},b_{u^1}$ and by $a_{u^2},b_{u^2}$ are mutually orthogonal.

$\lambda=1$ is not an eigenvalue of $I-M$ since the matrix $M$ is invertible. $\lambda=-1$ is not an eigenvalue of $I-M$ since otherwise the corresponding elements $a_u$, $b_u$ would span a non--trivial fixed point subspace for the action of $s$ in $\h'_{\mathbb{R}}$ which is impossible as $s$ acts on $\h'_{\mathbb{R}}$ without non--trivial fixed points.

If $\lambda= 0$ is an eigenvalue of $I-M$ then $\widetilde{u}$ and $\widetilde{\widetilde{u}}$ are the components of an eigenvector $u$ of $I-M$ with eigenvalue $0$ iff $A\widetilde{\widetilde{u}}=0$ and $A^\top\widetilde{u}=0$. Therefore using the usual orthogonalization procedure one can construct a basis $u^k \in \mathbb{R}^{l'}$, $k=1,\ldots, ~{\rm mult}~0$ of the eigenspace corresponding to $0$ orthonormal with respect to the standard scalar product on $\mathbb{R}^{l'}$ and such that the components $\widetilde{u}^k$ and $\widetilde{\widetilde{u}}^k$ $k=1,\ldots, ~{\rm mult}~0$ are all mutually orthogonal.

Now using the definition of the eigenvector we deduce that for the basis $u^k$ the following relations hold $B\widetilde{u}^k=\widetilde{u}^k$, $D\widetilde{\widetilde{u}}^k=\widetilde{\widetilde{u}}^k$, $C\widetilde{\widetilde{u}}^k=0$, $C^\top \widetilde{u}^k=0$. Recalling also formulas (\ref{1}), the first and the last identities in (\ref{2}), and (\ref{3}), and repeating the arguments given in the proofs of Lemma 10.4.2, Proposition 10.4.3 in \cite{Car} we immediately deduce the last statement of this proposition.

This completes the proof.

\end{proof}

Let $\h_{-1}$ be the subspace of $\h_{\mathbb{R}}$ on which $s$ acts by multiplication by $-1$.
According to the previous proposition one can choose one--dimensional $s^{1,2}$--invariant subspaces in $\h_{-1}$ such that $\h_{-1}$ is the orthogonal direct sum of those subspaces.


Using the previous proposition we can decompose $\h_{\mathbb{R}}$ into a direct orthogonal sum of $s$--invariant subspaces,
\begin{equation}\label{hdec}
\h_\mathbb{R}=\bigoplus_{i=0}^{K} \h_i,
\end{equation}
where $\h_0$ is the linear subspace of $\h_{\mathbb{R}}$ fixed by the action of $s$, each $\h_i$ is also invariant with respect to both involutions $s^{1,2}$ in the decomposition $s=s^1s^2$,
and each of the subspaces $\h_i\subset \h_\mathbb{R}$, $i=1,\ldots, K$, is either two--dimensional ($\h_i=\h_\lambda^k$ for an eigenvalue $0<\lambda<1$ of the matrix $I-M$, and $k=1,\ldots, {\rm mult}~\lambda$) and the Weyl group element $s$ acts on it as rotation with angle $\theta_i$, $0<\theta_i<\pi$ or $\h_i=\h_{\lambda}^k$, $\lambda=0$, $k=1,\ldots, {\rm mult}~\lambda$ has dimension $1$, and $s$ acts on it by multiplication by $-1$ .  Note that since $s$ has finite order $\theta_i=\frac{2\pi n_i}{m_i}$, $n_i,m_i\in \{1,2,\ldots \}$.

Since the number of roots in the root system $\Delta$ is finite one can always choose elements $h_i\in \h_i$, $i=0,\ldots, K$, such that $h_i(\alpha)\neq 0$ for any root $\alpha \in \Delta$ which is not orthogonal to the $s$--invariant subspace $\h_i$ with respect to the natural pairing between $\h_{\mathbb{R}}$ and $\h_{\mathbb{R}}^*$.

Now we consider certain $s$--invariant subsets of roots $\overline{\Delta}_i$, $i=0,\ldots, K$, defined as follows
\begin{equation}\label{di}
{\overline{\Delta}}_i=\{ \alpha\in \Delta: h_j(\alpha)=0, j>i,~h_i(\alpha)\neq 0 \},
\end{equation}
where we formally assume that $h_{K+1}=0$.
Note that for some indexes $i$ the subsets ${\overline{\Delta}}_i$ are empty, and that the definition of these subsets depends on the order of terms in direct sum (\ref{hdec}).

Now consider the nonempty $s$--invariant subsets of roots $\overline{\Delta}_{i_k}$, $k=0,\ldots, T$.
For convenience we assume that indexes $i_k$ are labeled in such a way that $i_j<i_k$ if and only if $j<k$.
According to this definition $\overline{\Delta}_{0}=\{\alpha \in \Delta: s\alpha=\alpha\}$ is the set of roots fixed by the action of $s$. Observe also that the root system $\Delta$ is the disjoint union of the subsets $\overline{\Delta}_{i_k}$,
$$
\Delta=\bigcup_{k=0}^{T}\overline{\Delta}_{i_k}.
$$

Now assume that
\begin{equation}\label{cond}
|h_{i_k}(\alpha)|>|\sum_{p\leq j<k}h_{i_j}(\alpha)|, ~{\rm for~any}~\alpha\in \overline{\Delta}_{i_k},~k=0,\ldots, T,~p<k.
\end{equation}
Condition (\ref{cond}) can be always fulfilled by suitable rescalings of the elements $h_{i_k}$.

Consider the element
\begin{equation}\label{hwb}
\bar{h}=\sum_{k=0}^{T}h_{i_k}\in \h_\mathbb{R}.
\end{equation}

From definition (\ref{di}) of the sets $\overline{\Delta}_i$ we obtain that for $\alpha \in \overline{\Delta}_{i_k}$
\begin{equation}\label{dech}
\bar{h}(\alpha)=\sum_{j\leq k}h_{i_j}(\alpha)=h_{i_k}(\alpha)+\sum_{j< k}h_{i_j}(\alpha)
\end{equation}
Now condition (\ref{cond}), the previous identity and the inequality $|x+y|\geq ||x|-|y||$ imply that for $\alpha \in \overline{\Delta}_{i_k}$ we have
$$
|\bar{h}(\alpha)|\geq ||h_{i_k}(\alpha)|-|\sum_{j< k}h_{i_j}(\alpha)||>0.
$$
Since $\Delta$ is the disjoint union of the subsets $\overline{\Delta}_{i_k}$, $\Delta=\bigcup_{k=0}^{M}\overline{\Delta}_{i_k}$, the last inequality ensures that  $\bar{h}$ belongs to a Weyl chamber of the root system $\Delta$, and one can define the subset of positive roots $\Delta_+$ and the set of simple positive roots $\Gamma$ with respect to that chamber. We call $\Delta_+$ introduced in this way a system of positive roots associated to (the conjugacy class of) the Weyl group element $s$. From condition (\ref{cond}) and formula (\ref{dech}) we also obtain that a root $\alpha \in \overline{\Delta}_{i_k}$ is positive if and only if
\begin{equation}\label{wc}
h_{i_k}(\alpha)>0.
\end{equation}
We denote by $(\overline{\Delta}_{i_k})_+$ the set of positive roots contained in $\overline{\Delta}_{i_k}$, $(\overline{\Delta}_{i_k})_+=\Delta_+\bigcap \overline{\Delta}_{i_k}$.

We shall also need the parabolic subgroup $P$ of $G$ associated to the subset $\Gamma_0=\Gamma\bigcap\overline{\Delta}_{0}\subset \Gamma$ of simple roots.
Let $N$ and $L$ be the unipotent radical and the Levi factor of $P$, respectively. Note that we have natural inclusions $P\supset B_+\supset N$, where $B_+$ is the Borel subgroup of $G$ corresponding to the system $\Gamma$ of simple roots, and $\overline{\Delta}_{0}$ is the root system of the reductive group $L$. We also denote by $\overline{N}$ the unipotent radical opposite to $N$.

Now we can define certain transversal slices to conjugacy classes in algebraic groups.
Let $Z$ be the subgroup of $G$  generated by the semisimple part $M$ of the Levi factor $L$ and by the centralizer of $\dot{s}$ in $H$.

\begin{proposition}\label{crosssect}
{\bf (\cite{S6}, Propositions 2.1 and 2.2)}
Let $N_s=\{ v \in N|\dot{s}v\dot{s}^{-1}\in \overline{N} \}$.
Then the conjugation map
\begin{equation}\label{cross}
N\times \dot{s}ZN_s\rightarrow N\dot{s}ZN
\end{equation}
is an isomorphism of varieties. Moreover, the variety $\Sigma_s=\dot{s}ZN_s$ is a transversal slice to the set of conjugacy classes in $G$.
\end{proposition}

Propositions 2.1 and 2.2 in \cite{S6} were proved when $\bf k$ is the field of complex numbers. But the proof can be straightforwardly generalized to the case of an arbitrary algebraically closed field $\bf k$.


\setcounter{equation}{0}
\setcounter{theorem}{0}

\section{The Lusztig partition}\label{luspart}

Recall that the purpose in this paper is show that for every conjugacy class $\mathcal{O}$ in $G$ one can find a transversal slice $\Sigma_s$ such that $\mathcal{O}$ intersects $\Sigma_s$ and ${\rm dim}~\mathcal{O}={\rm codim}~\Sigma_s$. It turns out that there is a remarkable partition of the group $G$ the strata of which are unions of conjugacy classes of the same dimension (see \cite{L1}). For each stratum of this partition there is a Weyl group element $s$ such that all conjugacy classes $\mathcal{O}$ from that stratum intersect a transversal slice $\Sigma_s$, and ${\rm dim}~\mathcal{O}={\rm codim}~\Sigma_s$. In this section following \cite{L1} we describe this partition which is called the Lusztig partition.

For any Weyl group $W$ let $\widehat{W}$ be the set of isomorphism classes of irreducible representations of $W$ over $\mathbb{Q}$. For any $E\in \widehat{W}$ let $b_E$ be the smallest nonnegative integer such that $E$ appears with non--zero multiplicity in the $b_E$-th symmetric power of the reflection representation of $W$. If this multiplicity is equal to $1$ then one says that $E$ is good. If $W'\subset W$ are two Weyl groups, and $E\in \widehat{W}'$ is good then there is a unique $\widetilde{E}\in \widehat{W}$ such that $\widetilde{E}$ appears in the decomposition of the induced representation ${\rm Ind}_{W'}^WE$, $b_{\widetilde{E}}=b_E$, and $\widetilde{E}$ is good. The representation $\widetilde{E}$ is called j-induced from $E$, $\widetilde{E}=j_{W'}^WE$.

Let $g\in G_p$, and $g=g_sg_u$ its decomposition as a product of the semisimple part $g_s$ and the unipotent part $g_u$. Let $C=Z_{G_p}(g_s)^0$ be the identity component of the centralizer of $g_s$ in $G_p$. $C$ is a reductive subgroup of $G_p$ of the same rank as $G_p$. Let $H_p$ be a maximal torus of $C$. $H_p$ is also a maximal torus in $G_p$, and hence one has a natural imbedding
$$
W'=N_C(H_p)/H_p \rightarrow N_{G_p}(H_p)/H_p=W,
$$
where $N_C(H_p), N_{G_p}(H_p)$ stand for the normalizers of $H_p$ in $C$ and in $G_p$, respectively, $W'$ is the Weyl group of $C$ and $W$ is the Weyl group of $G_p$.

Let $E$ be the irreducible representation of $W'$ associated with the help of the Springer correspondence to the conjugacy class of $g_u$ and the trivial local system on it. Then $E$ is good, and let $\widetilde{E}$ be the j-induced representation of $W$. This gives a well defined map $\phi_{G_p}:G_p\rightarrow \widehat{W}$. The fibers of this map are called the strata of $G_p$. By definition the map $\phi_{G_p}$ is constant on each conjugacy class in $G_p$. Therefore the strata are unions of conjugacy classes.

Moreover, by 1.3 in \cite{L1} we have for any $g\in G_p$
\begin{equation}\label{dimcon}
{\rm dim}~Z_{G_p}(g)={\rm rank}~G_p+2b_{\phi_{G_p}(g)},
\end{equation}
where ${\rm rank}~G_p$ is the rank of $G_p$.

It turns out that the image $\mathcal{R}(W)$ of $\phi_{G_p}$ only depends on $W$. It can be described as follows. Let $\mathcal{N}(G_p)$ be the unipotent variety of $G_p$ and $\underline{\mathcal{N}}(G_p)$ the set of unipotent classes in $G_p$. Let $\mathcal{X}^p(W)$ be the set of irreducible representations of $W$ associated by the Springer correspondence to unipotent classes in $\underline{\mathcal{N}}(G_p)$ and the trivial local systems on them. We shall identify $\mathcal{X}^p(W)$ and $\underline{\mathcal{N}}(G_p)$. Let $f_p:\underline{\mathcal{N}}(G_p) \rightarrow \mathcal{X}^p(W)$ be the corresponding bijective map.

\begin{proposition}{\bf (\cite{L1}, Proposition 1.4)}
We have
$$
\mathcal{R}(W)=\mathcal{X}^1(W)\bigcup_{r~{\rm prime}}\mathcal{X}^r(W).
$$
\end{proposition}

If $G_p$ is of type $A_n$, $(n\geq 1)$ or $E_6$ then $\mathcal{R}(W)=\mathcal{X}^1(W)$.

If $G_p$ is of type $B_n$ $(n\geq 2)$, $C_n$ $(n\geq 3)$, $D_n$ $(n\geq 4)$, $F_4$ or $E_7$ then $\mathcal{R}(W)=\mathcal{X}^2(W)$.

If $G_p$ is of type $G_2$ then $\mathcal{R}(W)=\mathcal{X}^3(W)$.

If $G_p$ is of type $E_8$ then $\mathcal{R}(W)=\mathcal{X}^2(W)\bigcup\mathcal{X}^3(W)$, and $\mathcal{X}^2(W)\bigcap\mathcal{X}^3(W)=\mathcal{X}^1(W)$.

The above description of the set $\mathcal{R}(W)$ and the bijections $\underline{\mathcal{N}}(G_p) \rightarrow \mathcal{X}^p(W)$ yield certain maps between sets $\underline{\mathcal{N}}(G_p)$ which preserve dimensions of conjugacy classes by (\ref{dimcon}). For instance, one always has an inclusion $\mathcal{X}^1(W) \subset \mathcal{X}^r(W)$ for any $r\geq 2$. The corresponding inclusion $\underline{\mathcal{N}}(G_1)\subset \underline{\mathcal{N}}(G_p)$ coincides with the Spaltenstein map $\pi^G_p: \underline{\mathcal{N}}(G_1)\rightarrow \underline{\mathcal{N}}(G_p)$ (see \cite{Spal}, Th\'{e}or\`{e}me III.5.2).

Fix a system of positive roots in $\Delta=\Delta(G_p,H_p)$, and let $B_p$ be the corresponding Borel subgroup in $G_p$, $H_p\subset B_p$ the maximal torus, and $\underline{l}$ the corresponding length function on $W$. Denote by $\underline{W}$ the set of conjugacy classes in $W$. For each $w\in W=N_{G_p}(H_p)/H_p$ one can pick up a representative $\dot{w}\in G_p$.
If $p$ is good for $G_p$, we simply write $G_p=G$, $B_p=B$, $\underline{\mathcal{N}}(G_p)=\underline{\mathcal{N}}(G)$.

Let $\mathcal{C}$ be a conjugacy class in $W$. Pick up a representative $w\in \mathcal{C}$ of minimal possible length with respect to $\underline{l}$. By Theorem 0.4 in \cite{L4} there is a unique conjugacy class $\mathcal{O}\in \underline{\mathcal{N}}(G)$ of minimal possible dimension which intersects the Bruhat cell $B\dot{w}B$ and does not depend on the choice of the minimal possible length representative $w$ in $\mathcal{C}$. We denote this class by $\Phi_1^G(\mathcal{C})$.

The representative $w$ is elliptic in a parabolic Weyl subgroup $W'\subset W$, i.e. $w$ acts without fixpoints in the reflection representation of $W'$. Let $P'\subset G$ be the parabolic subgroup which corresponds to $W'$, and $M'$ the semi-simple part of the Levi factor of $P'$, so that $W'$ is the Weyl group of $M'$. Let $\Phi_p^G(\mathcal{C})$ be the unipotent class in $G_p$ containing the class $\pi_p^{M'}\Phi_1^{M'}(\mathcal{C})$. This class only depends on the conjugacy class $\mathcal{C}$, and hence one has a map $\Phi_p^G: \underline{W}\rightarrow \underline{\mathcal{N}}(G_p)$ which is in fact surjective by 4.5(a) in \cite{L4}.

Let $\mathcal{C}\in \underline{W}$, and $m_\mathcal{C}$ the dimension of the fixpoint space for the action of any $w\in \mathcal{C}$ in the reflection representation.
Then by Theorem 0.2 in \cite{L3} for any $\gamma \in \underline{\mathcal{N}}(G_p)$ there is a unique $\mathcal{C}_0\in (\Phi_p^G)^{-1}(\gamma)$ such that the function $m_\mathcal{C}: (\Phi_p^G)^{-1}(\gamma) \rightarrow \mathbb{N}$ reaches its minimum at $\mathcal{C}_0$. We denote $\mathcal{C}_0$ by $\Psi_p^G(\gamma)$. Thus one obtains an injective map $\Psi_p^G: \underline{\mathcal{N}}(G_p)\rightarrow \underline{W}$.

Now recall that using identifications $f_p:\underline{\mathcal{N}}(G_p) \rightarrow \mathcal{X}^p(W)$ one can define a bijection
$$
F:\widehat{\underline{\mathcal{N}}}(G)=\underline{\mathcal{N}}(G_1)\bigcup_{r~{\rm prime}}\underline{\mathcal{N}}(G_r)\rightarrow \mathcal{X}^1(W)\bigcup_{r~{\rm prime}}\mathcal{X}^r(W)=\mathcal{R}(W).
$$

Using maps $\Phi_p^G$ one can also define a surjective map $\Phi^W: \underline{W}\rightarrow \widehat{\underline{\mathcal{N}}}(G)$ as follows. If $\Phi_r^G(\mathcal{C})\in \underline{\mathcal{N}}(G_1)$ for all $r>1$ then $\Phi_r^G(\mathcal{C})$ is independent of $r$, and one puts $\Phi^W(\mathcal{C})=\Phi_r^G(\mathcal{C})$ for any $r>1$. If $\Phi_r^G(\mathcal{C}) \not\in \underline{\mathcal{N}}(G_1)$ for some $r>1$ then $r$ is unique, one defines $\Phi^W(\mathcal{C})=\Phi_r^G(\mathcal{C})$.

By definition there is a right sided injective inverse $\Psi^W$ to $\Phi^W$ such that if $\gamma \in \underline{\mathcal{N}}(G_1)$ then $\Psi^W(\gamma)=\Psi_1^G(\gamma)$, and if $\gamma \not \in \underline{\mathcal{N}}(G_1)$, and $\gamma \in \underline{\mathcal{N}}(G_r)$ then $\Psi^W(\gamma)=\Psi_r^G(\gamma)$.

Denote by $C(W)$ the image of $\widehat{\underline{\mathcal{N}}}(G)$ in $\underline{W}$ under the map $\Psi^W$, $C(W)=\Psi^W(\widehat{\underline{\mathcal{N}}}(G))$. We shall identify $C(W)$, $\widehat{\underline{\mathcal{N}}}(G)$ and  $\mathcal{R}(W)$.

Now assume that $p$ is not a bad prime for $G_p$. In this case the strata of the Lusztig partition can be described geometrically as follows. Let $\mathcal{C} \in C(W)$. Pick up a representative $w\in \mathcal{C}$ of minimal possible length with respect to $\underline{l}$. Denote by $\underline{G}_p$ the set of conjugacy classes in $G_p$, and by $G_\mathcal{C}'$ the set of all conjugacy classes in $G_p$ which intersect the Bruhat cell $B_p\dot{w}B_p$. This definition does not depend on the choice of the the minimal possible length representative $w$. Let
$$
d_\mathcal{C}={\hbox{\raise-1.5mm\hbox{${\textstyle \rm min}\atop {\scriptstyle \gamma \in G_\mathcal{C}'}$}}}~{\rm dim}~\gamma.
$$
Then the stratum $G_\mathcal{C}=\phi_{G_p}^{-1}(F(\Phi^W(\mathcal{C})))$ can be described as follows (see Theorem 2.2, \cite{L1}),
$$
G_\mathcal{C}=\bigcup_{\gamma \in G_\mathcal{C}', ~{\rm dim}\gamma=d_\mathcal{C}}\gamma.
$$

Thus we have a disjoint union
$$
G_p=\bigcup_{\mathcal{C}\in C(W)}G_\mathcal{C}.
$$
Note that by definition a stratum $G_\mathcal{C}$ contains a unique unipotent class if and only if $\mathcal{C}\in {\rm Im}(\Psi_1^G)$.

For good $p$ the maps introduced above are summarized in the following diagram
\begin{equation}\label{diag}
\begin{array}{ccccccc}
   &  & \mathcal{X}^1(W) & \stackrel{f_1}{\longleftarrow} & \underline{\mathcal{N}}(G) &  & \\
   &  & \downarrow ~{\rm in} &  & \downarrow ~\pi^G &  & \\
  G & \stackrel{\phi_G}{\longrightarrow} & \mathcal{R}(W) & \stackrel{F}{\longleftarrow} & \widehat{\underline{\mathcal{N}}}(G) & {\hbox{\raise-1mm\hbox{${\textstyle \stackrel{\Phi^W}{\longleftarrow}}\atop {\hbox{${\textstyle \longrightarrow}\atop {\scriptstyle \Psi^W}$}}$}}} & \underline{W},
\end{array}
\end{equation}
where ``$\rm in$'' is an inclusion, bijections $f_1,F$ are induced by the Springer correspondence with the trivial local data, and the inclusion $\pi^G$ is induced by the Spaltenstein map.

For exceptional groups the maps $f_1,F$ can be described explicitly using tables in \cite{Spal1}, $\Phi^W,\Psi^W$ can be described using the tables in Section 2 in \cite{L3}, and ``${\rm in}$'', $\pi^G$ can be described explicitly using the tables of unipotent classes in \cite{Li} or \cite{Spal1} (note that the labelling for unipotent classes in bad characteristics in \cite{Li} differs from that in \cite{Spal1}). The dimensions of the conjugacy classes in the strata in $G$ can be obtained using dimension tables of centralizers of unipotent elements in case when a stratum contains a unipotent class (see \cite{Car1,Li}), the tables for dimensions of the centralizers of unipotent elements in bad characteristic when a stratum does not contain a unipotent class (see \cite{Li}) or formula (\ref{dimcon}) and the tables of the values of the $b$--invariant $b_E$ for representations of Weyl groups (see \cite{Car1,Gk}). Note that formula (\ref{dimcon}) implies that
if $\mathcal{O}$ is any conjugacy class in $G_\mathcal{C}$, $\mathcal{O}\in G_\mathcal{C}$ then
\begin{equation}\label{od}
{\rm dim}~\mathcal{O}={\rm dim}~\Phi^W(\mathcal{C}).
\end{equation}

In case of classical groups all those maps and dimensions are described in terms of partitions (see \cite{Car1,Gk1,Li,L2,L3,L4,Spal}). In case of classical matrix groups the strata can also be described explicitly (see \cite{L1}).
We recall this description below. By (\ref{od}) the dimensions of the conjugacy classes in every stratum of $G$ are equal to the dimension of the corresponding conjugacy class in $\widehat{\underline{\mathcal{N}}}(G)$. The dimensions of centralizers of unipotent elements in arbitrary characteristic can be found in \cite{Hess,Li}.

If $\lambda =(\lambda_1\geq \lambda_2\geq \ldots\geq \lambda_m)$ is a partition we denote by $\lambda^* =(\lambda_1^*\geq \lambda_2^*\geq \ldots\geq \lambda_m^*)$ the corresponding dual partition. It is defined by the property that $\lambda_1^*=m$ and $\lambda_i^*-\lambda_{i+1}^*=l_i(\lambda)$, where $l_i(\lambda)$ is the number of times $i$ appears in the partition $\lambda$. We also denote by $\tau(\lambda)$ the length of $\lambda$, $\tau(\lambda)=m$. If a partition $\mu$ is obtained from $\lambda$ by adding a number of zeroes, we shall identify $\lambda$ and $\mu$.

\subsection*{$\bf A_n$}

$G$ is of type ${\rm SL}(V)$ where $V$ is a vector space of dimension $n+1 \geq 1$ over an algebraically closed field ${\bf k}$ of characteristic exponent $p\geq 1$. $W$ is the group of permutations of $n+1$ elements.
All sets in (\ref{diag}), except for $G$, are isomorphic to the set of partitions of $n+1$, and all maps, except for $\phi_G$, are the identity maps.

To describe $\phi_G$ for $G={\rm SL}(V)$ we choose a sufficiently large $m \in \mathbb{N}$. Let $g \in G$. For any $x \in {\bf k}^*$ let $V_x$ be the
generalized $x$--eigenspace of $g:V \rightarrow V$ and let $\lambda^x_1
\geq \lambda^x_2
\geq \ldots \geq \lambda^x_m$ be the sequence
in $\mathbb{N}$ whose terms are the sizes of the Jordan blocks of $x^{-1}g: V_x \rightarrow V_x$. Then $\phi_G(g)$ is
the partition $\lambda(g)_1 \geq \lambda(g)_2 \geq \ldots \geq \lambda(g)_m$ given by $\lambda(g)_j = \sum_{x\in {\bf k}^*} \lambda^x_j$.

If $g$ is any element in the stratum $G_{\lambda}$ corresponding to a partition $\lambda=(\lambda_1
\geq \lambda_2\geq \ldots \geq \lambda_m)$, $\lambda_m\geq 1$, then
\begin{equation}\label{dimsl}
{\rm dim}~Z_G(g)=n+2\sum_{i=1}^m(i-1)\lambda_i.
\end{equation}

The element of $\underline{W}$ which corresponds to $\lambda$ is the Coxeter class in the Weyl subgroup of the type
\begin{equation}\label{san}
A_{\lambda_1-1}+A_{\lambda_2-1}+\ldots +A_{\lambda_m-1}.
\end{equation}
The summands in the diagram above are called blocks. Blocks of type $A_0$ are called trivial.


\subsection*{$\bf C_n$}

$G$ is of type ${\rm Sp}(V)$ where $V$ is a symplectic space of dimension $2n$, $n \geq 2$ over an algebraically closed field ${\bf k}$ of characteristic exponent $p\neq 2$.
$W$ is the group of permutations of the set $E=\{\varepsilon_1,\ldots, \varepsilon_n,-\varepsilon_1,\ldots,-\varepsilon_n\}$ which also commute with the involution  $\varepsilon_i\mapsto -\varepsilon_i$.

Elements of $\underline{W}$ are parameterized by pairs of partitions $(\lambda,\mu)$, where the parts of $\lambda$ are even (for any $w\in \mathcal{C}\in \underline{W}$ they are the numbers of elements in the negative orbits $X$, $X=-X$,  in $E$ for the action of the group $<w>$ generated by $w$), $\mu$ consists of pairs of equal parts (they are the numbers of elements in the positive $<w>$--orbits $X$ in $E$; these orbits appear in pairs $X,-X$, $X\neq -X$), and $\sum \lambda_i +\sum \mu_j=2n$. We denote this set of pairs of partitions by $\mathcal{A}^1_{2n}$.
An element of $\underline{W}$ which corresponds to a pair $(\lambda,\mu)$, $\lambda=(\lambda_1
\leq \lambda_2\leq \ldots \leq \lambda_m)$ and $\mu=(\mu_1=\mu_2\leq \ldots \leq \mu_{2k-1}=\mu_{2k})$ is the Coxeter class in the Weyl subgroup of the type
\begin{equation}\label{scn}
C_{\frac{\lambda_1}{2}}+C_{\frac{\lambda_2}{2}}+\ldots +C_{\frac{\lambda_m}{2}}+A_{\mu_1-1}+A_{\mu_3-1}+\ldots +A_{\mu_{2k-1}-1}.
\end{equation}

Elements of $\underline{\mathcal{N}}(G)$ are parameterized by partitions $\lambda$ of $2n$ for which $l_j(\lambda)$ is even for odd $j$. We denote this set of partitions by $\mathcal{T}_{2n}$. In case of $G={\rm Sp}(V)$ the parts of $\lambda$ are just the sizes of the Jordan blocks in $V$ of the unipotent elements from the conjugacy class corresponding to $\lambda$.

In this case $\widehat{\underline{\mathcal{N}}}(G)=\underline{\mathcal{N}}(G_2)$, and $G_2$ is of type ${\rm Sp}(V_2)$ where $V_2$ is a symplectic space of dimension $2n$ over an algebraically closed field of characteristic $2$. Elements of $\underline{\mathcal{N}}(G_2)$ are parameterized by pairs $(\lambda,\varepsilon)$, where $\lambda = (\lambda_1
\leq \lambda_2\leq \ldots \leq \lambda_m)\in \mathcal{T}_{2n}$, and $\varepsilon:\{\lambda_1, \lambda_2, \ldots , \lambda_m\}\rightarrow \{0,1,\omega\}$ is a function such that
\begin{equation}\label{eps}
\varepsilon(k)=\left\{
                 \begin{array}{ll}
                 \omega  & \hbox{if $k$ is odd;} \\
                   1  & \hbox{if $k=0$;} \\
                   1  & \hbox{if $k>0$ is even, $l_k(\lambda)$ is odd;} \\
                   0~{\rm or}~1  & \hbox{if $k>0$ is even, $l_k(\lambda)$ is even.}
                 \end{array}
               \right.
\end{equation}
We denote the set of such pairs $(\lambda,\varepsilon)$ by $\mathcal{T}_{2n}^2$.

Elements of $\widehat{W}$ are parameterized by pairs of partitions $(\alpha,\beta)$ written in non--decreasing order, $\alpha_1\leq \alpha_2\leq \ldots \leq \alpha_{\tau(\alpha)}$, $\beta_1\leq \beta_2 \leq \ldots \leq \beta_{\tau(\beta)}$, and such that $\sum \alpha_i+\sum \beta_i=n$. By adding zeroes we can assume that the length $\tau(\alpha)$ of $\alpha$ is related to the length of $\beta$ by $\tau(\alpha)=\tau(\beta)+1$. The set of such pairs is denoted by $X_{n,1}$.

The maps $f_1,F$ can be described as follows. Let $\lambda = (\lambda_1
\leq \lambda_2\leq \ldots \leq \lambda_{2m+1})\in \mathcal{T}_{2n}$, and assume that $\lambda_1=0$. If $f_1(\lambda)=((c'_1,c'_3,\ldots,c'_{2m+1}),(c'_2,c'_4,\ldots,c'_{2m}))$ then the parts $c'_i$ are defined by induction starting from $c'_1=0$,
$$
  \begin{array}{ll}
    c'_i=\frac{\lambda_i}{2} & \hbox{if $\lambda_i$ is even and $c'_{i-1}$ is already defined;} \\
    c'_i=\frac{\lambda_i+1}{2} & \hbox{if $\lambda_i=\lambda_{i+1}$ is odd and $c'_{i-1}$ is already defined;} \\
    c'_{i+1}=\frac{\lambda_i-1}{2} & \hbox{if $\lambda_i=\lambda_{i+1}$ is odd and $c'_{i}$ is already defined.}
  \end{array}
$$

The image of $f_1$ consists of all pairs $((c'_1,c'_3,\ldots,c'_{2m+1}),(c'_2,c'_4,\ldots,c'_{2m}))\in X_{n,1}$ such that $c'_i\leq c'_{i+1}+1$ for all $i$.

If $F(\lambda,\varepsilon)=((c_1,c_3,\ldots,c_{2m+1}),(c_2,c_4,\ldots,c_{2m}))$ then the parts $c_i$ are defined by induction starting from $c_1=0$,
$$
  \begin{array}{ll}
    c_i=\frac{\lambda_i}{2} & \hbox{if $\lambda_i$ is even, $\varepsilon(\lambda_i)=1$ and $c_{i-1}$ is already defined;} \\
    c_i=\frac{\lambda_i+1}{2} & \hbox{if $\lambda_i=\lambda_{i+1}$ is odd and $c_{i-1}$ is already defined;} \\
    c_{i+1}=\frac{\lambda_i-1}{2} & \hbox{if $\lambda_i=\lambda_{i+1}$ is odd and $c_{i}$ is already defined;} \\
    c_i=\frac{\lambda_i+2}{2} & \hbox{if $\lambda_i=\lambda_{i+1}$ is even, $\varepsilon(\lambda_i)=\varepsilon(\lambda_{i+1})=0$ and $c_{i-1}$ is already defined;} \\
    c_{i+1}=\frac{\lambda_i-2}{2} & \hbox{if $\lambda_i=\lambda_{i+1}$ is even, $\varepsilon(\lambda_i)=\varepsilon(\lambda_{i+1})=0$ and $c_{i}$ is already defined.}
  \end{array}
$$

The image $\mathcal{R}(W)$ of $F$ consists of all pairs $((c_1,c_3,\ldots,c_{2m+1}),(c_2,c_4,\ldots,c_{2m}))\in X_{n,1}$ such that $c_i\leq c_{i+1}+2$ for all $i$.

The map $\Phi^W$ is defined by $\Phi^W(\lambda,\mu)=(\nu,\varepsilon)$, where the set of parts of $\nu$ is just the union of the sets of parts of $\lambda$ and $\mu$, and
$$
\varepsilon(k)=\left\{
                 \begin{array}{ll}
                   1 & \hbox{if $k\in 2\mathbb{N}$ is a part of $\lambda$;} \\
                   0 & \hbox{if $k\in 2\mathbb{N}$ is not a part of $\lambda$;} \\
                   \omega & \hbox{if $k$ is odd.}
                 \end{array}
               \right.
$$

The map $\Psi^W$ associates to each pair $(\nu, \varepsilon)$ a unique point $(\lambda,\mu)$ in the preimage $(\Phi^W)^{-1}(\nu, \varepsilon)$ such that the number of parts of $\mu$ is minimal possible. This point is defined by the conditions
$$
l_k(\lambda)=\left\{
               \begin{array}{ll}
                 0 & \hbox{if $k$ is odd or $k$ is even, $l_k(\nu)\geq 2$ is even and $\varepsilon(k)=0$;} \\
                 l_k(\nu) & \hbox{otherwise,}
               \end{array}
             \right.
$$
$$
l_k(\mu)=\left\{
               \begin{array}{ll}
                 l_k(\nu) & \hbox{if $k$ is odd or $k$ is even, $l_k(\nu)\geq 2$ is even and $\varepsilon(k)=0$;} \\
                 0 & \hbox{otherwise.}
               \end{array}
             \right.
$$

The map $\pi^G$ is given by $\pi^G(\lambda)=(\lambda, \varepsilon')$, where
$$
\varepsilon'(k)=\left\{
                 \begin{array}{ll}
                 \omega  & \hbox{if $k$ is odd;} \\
                   1  & \hbox{if $k$ is even.}
                 \end{array}
               \right.
$$
The map ${{\pi}}^G$ is injective and its image consists of pairs $(\lambda, \varepsilon)\in {\mathcal{T}}_{2n}^2$, where $\varepsilon$ satisfies the conditions above.

To describe $\phi_G$ for $G={\rm Sp}(V)$ we choose a sufficiently large $m \in \mathbb{N}$. Let $g \in G$. For any $x \in {\bf k}^*$ let $V_x$ be the
generalized $x$--eigenspace of $g:V \rightarrow V$. For any $x \in {\bf k}^*$ such that $x^2\neq 1$ let $\lambda^x_1
\geq \lambda^x_2
\geq \ldots \geq \lambda^x_{2m+1}$ be the sequence
in $\mathbb{N}$ whose terms are the sizes of the Jordan blocks of $x^{-1}g: V_x \rightarrow V_x$.

For any $x \in {\bf k}^*$ with $x^2= 1$ let $\lambda^x_1
\geq \lambda^x_2\geq \ldots \geq \lambda^x_{2m+1}$ be the sequence in $\mathbb{N}$,
where $((\lambda^x_1
\geq \lambda^x_3\geq \ldots \geq \lambda^x_{2m+1}),(\lambda^x_2
\geq \lambda^x_4\geq \ldots \geq \lambda^x_{2m}))$ is the pair of partitions
such that the corresponding irreducible representation of the Weyl
group of type $B_{{\rm dim}~V_x/2}$ is the Springer representation attached to the unipotent
element $x^{-1}g \in {\rm Sp}(V_x)$ and to the trivial local data.

Let $\lambda(g)$ be
the partition $\lambda(g)_1 \geq \lambda(g)_2 \geq \ldots \geq \lambda(g)_{2m+1}$ given by $\lambda(g)_j = \sum_{x} \lambda^x_j$, where $x$ runs over a set of representatives for the orbits of the
involution $a\mapsto a^{-1}$ of ${\bf k}^*$.
Now $\phi_G(g)$ is the pair of partitions $((\lambda(g)_1
\geq \lambda(g)_3\geq \ldots \geq \lambda(g)_{2m+1}),(\lambda(g)_2
\geq \lambda(g)_4\geq \ldots \geq \lambda(g)_{2m}))$.

If $g$ is any element in the stratum $G_{(\lambda,\varepsilon)}$ corresponding to a pair $(\lambda,\varepsilon)\in \mathcal{T}_{2n}^2$, $\lambda=(\lambda_1
\geq \lambda_2\geq \ldots \geq \lambda_m)$ then
\begin{equation}\label{dimsp}
{\rm dim}~Z_G(g)=n+\sum_{i=1}^m(i-1)\lambda_i+\frac{1}{2}|\{i:\lambda_i ~\hbox{is odd}\}|+|\{i:\lambda_i ~\hbox{is even and}~\varepsilon(\lambda_i)=0\}|.
\end{equation}


\subsection*{$\bf B_n$}

$G$ is of type ${\rm SO}(V)$ where $V$ is a vector space of dimension $2n+1$, $n \geq 2$ over an algebraically closed field ${\bf k}$ of characteristic exponent $p\neq 2$ equipped with a non--degenerate symmetric bilinear form.
$W$ is the same as in case of $C_n$.

An element of $\underline{W}$ which corresponds to a pair $(\lambda,\mu)$, $\lambda=(\lambda_1
\geq \lambda_2\geq \ldots \geq \lambda_m)$ and $\mu=(\mu_1=\mu_2\geq \ldots \geq \mu_{2k-1}=\mu_{2k})$ is the class represented by the sum of the blocks in the following diagram (we use the notation of \cite{C}, Section 7)
\begin{eqnarray}
A_{\mu_1-1}+A_{\mu_3-1}+\ldots +A_{\mu_{2k-1}-1}+ \qquad \qquad \qquad \qquad \qquad \qquad \qquad \qquad \qquad \qquad \qquad \qquad \nonumber \\
+D_{\frac{\lambda_1+\lambda_2}{2}}(a_{\frac{\lambda_2}{2}-1})+
D_{\frac{\lambda_3+\lambda_4}{2}}(a_{\frac{\lambda_4}{2}-1})+\ldots +D_{\frac{\lambda_{m-2}+\lambda_{m-1}}{2}}(a_{\frac{\lambda_{m-1}}{2}-1})+B_{\frac{\lambda_m}{2}} ~(\mbox{$m$ is odd}),\label{sbn} \\
A_{\mu_1-1}+A_{\mu_3-1}+\ldots +A_{\mu_{2k-1}-1}+ \qquad \qquad \qquad \qquad \qquad \qquad \qquad \qquad \qquad \qquad \qquad \qquad \nonumber \\
+D_{\frac{\lambda_1+\lambda_2}{2}}(a_{\frac{\lambda_2}{2}-1})+D_{\frac{\lambda_3+\lambda_4}{2}}(a_{\frac{\lambda_4}{2}-1})+\ldots +D_{\frac{\lambda_{m-1}+\lambda_{m}}{2}}(a_{\frac{\lambda_{m}}{2}-1})~(\mbox{$m$ is even}),\nonumber
\end{eqnarray}
where it is assumed that $D_k(a_0)=D_k$.

The elements of $\underline{\mathcal{N}}(G)$ are parameterized by partitions $\lambda$ of $2n+1$ for which $l_j(\lambda)$ is even for even $j$. We denote this set of partitions by $\mathcal{Q}_{2n+1}$. In case of $G={\rm SO}(V)$ the parts of $\lambda$ are just the sizes of the Jordan blocks in $V$ of the unipotent elements from the conjugacy class corresponding to $\lambda$.

In this case $\widehat{\underline{\mathcal{N}}}(G)=\underline{\mathcal{N}}(G_2)$, and $G_2$ is of type ${\rm SO}(V_2)$ where $V_2$ is a vector space of dimension $2n+1$ over an algebraically closed field of characteristic $2$ equipped with a bilinear form $(\cdot,\cdot)$ and a non--zero quadratic form $Q$ such that
$$
(x,y)=Q(x+y)-Q(x)-Q(y),~x,y\in V_2,
$$
and the restriction of $Q$ to the null space $V_2^\perp=\{x\in V_2:(x,y)=0~\forall~y\in V_2\}$ of $(\cdot,\cdot)$ has zero kernel. In fact $G_2$ is isomorphic to a group of type ${\rm Sp}(V_2)$, ${\rm dim}~V_2=2n$, and hence
$\underline{\mathcal{N}}(G_2) \simeq \mathcal{T}_{2n}^2$.

We also have $\widehat{W}\simeq X_{n,1}$, and the map $F$ is the same as in case of $C_n$.

The map $f_1$ can be described as follows. Let $\lambda = (\lambda_1
\leq \lambda_2\leq \ldots \leq \lambda_{2m+1})\in \mathcal{Q}_{2n+1}$. If $$f_1(\lambda)=((c'_1,c'_3,\ldots,c'_{2m+1}),(c'_2,c'_4,\ldots,c'_{2m}))$$ then the parts $c'_i$ are defined by induction starting from $c'_1$,
$$
  \begin{array}{ll}
    c'_i=\frac{\lambda_i-1}{2}+i-1-2\left[\frac{i-1}{2}\right] & \hbox{if $\lambda_i$ is odd and $c'_{i-1}$ is already defined;} \\
    c'_i=\frac{\lambda_i}{2} & \hbox{if $\lambda_i=\lambda_{i+1}$ is even and $c'_{i-1}$ is already defined;} \\
    c'_{i+1}=\frac{\lambda_i}{2} & \hbox{if $\lambda_i=\lambda_{i+1}$ is even and $c'_{i}$ is already defined.}
  \end{array}
$$

The image of $f_1$ consists of all pairs $((c'_1,c'_3,\ldots,c'_{2m+1}),(c'_2,c'_4,\ldots,c'_{2m}))\in X_{n,1}$ such that $c'_i\leq c'_{i+1}$ for all odd $i$ and $c'_i\leq c'_{i+1}+2$ for all even $i$.

The image $\mathcal{R}(W)$ of $F$ consists of all pairs $((c_1,c_3,\ldots,c_{2m+1}),(c_2,c_4,\ldots,c_{2m}))\in X_{n,1}$ such that $c_i\leq c_{i+1}+2$ for all $i$.

The maps $\Phi^W$ and $\Psi^W$ are the same as in case of $C_n$.

The map $\pi^G$ is given by $\pi^G(\lambda)=(\nu, \varepsilon')$, $\lambda = (\lambda_1
\leq \lambda_2\leq \ldots \leq \lambda_{2m+1})\in \mathcal{Q}_{2n+1}$, where
$$
\nu_i=\left\{
        \begin{array}{ll}
          \lambda_i-1 & \hbox{if $\lambda_i$ and $i$ are odd and $\lambda_{i-1}<\lambda_i$;} \\
          \lambda_i+1 & \hbox{if $\lambda_i$ is odd, $i$ is even and $\lambda_{i}<\lambda_{i+1}$;}  \\
          \lambda_i & \hbox{otherwise},
        \end{array}
      \right.
$$
and
$$
\varepsilon'(k)=\left\{
                 \begin{array}{ll}
                 \omega  & \hbox{if $k$ is odd;} \\
                   0  & \hbox{if $k$ is even, there exists even $\lambda_i=k$ with even $i$ such that $\lambda_{i-1}<\lambda_i$;} \\
1 & \hbox{otherwise.}
                 \end{array}
               \right.
$$

The map ${\pi}^G$ is injective and its image consists of pairs $(\nu, \varepsilon)\in {\mathcal{T}}_{2n}^2$ such that $\varepsilon(k)\neq 0$ if $\nu_k^*$ is odd and for each even $i$ such that $\nu_i^*$ is even we have $\nu_{i-1}^*=\nu_i^*$, i.e. $i-1$ does not appear in the partition $\nu$. Here $\nu_1^*\geq\nu_2^*\geq \ldots \geq \nu_{m}^*$ is the partition dual to $\nu$.

To describe $\phi_G$ for $G={\rm SO}(V)$ we choose a sufficiently large $m \in \mathbb{N}$. Let $g \in G$. For any $x \in {\bf k}^*$ let $V_x$ be the
generalized $x$--eigenspace of $g:V \rightarrow V$. For any $x \in {\bf k}^*$ such that $x^2\neq 1$ let $\lambda^x_1
\geq \lambda^x_2
\geq \ldots \geq \lambda^x_{2m+1}$ be the sequence
in $\mathbb{N}$ whose terms are the sizes of the Jordan blocks of $x^{-1}g: V_x \rightarrow V_x$.

For any $x \in {\bf k}^*$ with $x^2= 1$ let $\lambda^x_1
\geq \lambda^x_2\geq \ldots \geq \lambda^x_{2m+1}$ be the sequence in $\mathbb{N}$, where
$((\lambda^x_1
\geq \lambda^x_3\geq \ldots \geq \lambda^x_{2m+1}),(\lambda^x_2
\geq \lambda^x_4\geq \ldots \geq \lambda^x_{2m}))$ is the pair of partitions
such that the corresponding irreducible representation of the Weyl
group of type $B_{({\rm dim}~V_x-1)/2}$ (if $x\neq -1$) or $D_{{\rm dim}~V_x/2}$ (if $x=-1$) is the Springer representation attached to the unipotent
element $x^{-1}g \in {\rm SO}(V_x)$ and to the trivial local data.

Let $\lambda(g)$ be
the partition $\lambda(g)_1 \geq \lambda(g)_2 \geq \ldots \geq \lambda(g)_{2m+1}$ given by $\lambda(g)_j = \sum_{x} \lambda^x_j$, where $x$ runs over a set of representatives for the orbits of the
involution $a\mapsto a^{-1}$ of ${\bf k}^*$.
Now $\phi_G(g)$ is the pair of partitions $((\lambda(g)_1
\geq \lambda(g)_3\geq \ldots \geq \lambda(g)_{2m+1}),(\lambda(g)_2
\geq \lambda(g)_4\geq \ldots \geq \lambda(g)_{2m}))$.

If $g$ is any element in the stratum $G_{(\lambda,\varepsilon)}$ corresponding to a pair $(\lambda,\varepsilon)\in \mathcal{T}_{2n}^2$, $\lambda=(\lambda_1
\geq \lambda_2\geq \ldots \geq \lambda_m)$ then the dimension of the centralizer of $g$ in $G$ is given by formula (\ref{dimsp}),
\begin{equation}\label{dimsoodd}
{\rm dim}~Z_G(g)=n+\sum_{i=1}^m(i-1)\lambda_i+\frac{1}{2}|\{i:\lambda_i ~\hbox{is odd}\}|+|\{i:\lambda_i ~\hbox{is even and}~\varepsilon(\lambda_i)=0\}|.
\end{equation}


\subsection*{$\bf D_n$}

$G$ is of type ${\rm SO}(V)$ where $V$ is a vector space of dimension $2n$, $n \geq 3$ over an algebraically closed field ${\bf k}$ of characteristic exponent $p\neq 2$ equipped with a non--degenerate symmetric bilinear form.
$W$ is the group of even permutations of the set $E=\{\varepsilon_1,\ldots, \varepsilon_n,-\varepsilon_1,\ldots,-\varepsilon_n\}$ which also commute with the involution  $\varepsilon_i\mapsto -\varepsilon_i$. $W$ can be regarded as a subgroup in the Weyl group $W'$ of type $C_n$.

Let $\underline{\widetilde{W}}$ be the set of $W'$--conjugacy classes in $W$.
Elements of $\underline{\widetilde{W}}$ are parameterized by pairs of partitions $(\lambda,\mu)$, where the parts of $\lambda$ are even (for any $w\in \mathcal{C}\in \underline{\widetilde{W}}$ they are the numbers of elements in the negative orbits $X$, $X=-X$,  in $E$ for the action of the group $<w>$ generated by $w$), the number of parts of $\lambda$ is even, $\mu$ consists of pairs of equal parts (they are the numbers of elements in the positive $<w>$--orbits $X$ in $E$; these orbits appear in pairs $X,-X$, $X\neq -X$), and $\sum \lambda_i +\sum \mu_j=2n$.
We denote this set of pairs of partitions by $\mathcal{A}^0_{2n}$. To each pair $(-,\mu)$, where all parts of $\mu$ are even, there correspond two conjugacy classes in $W$. To all other elements of $\mathcal{A}^0_{2n}$ there corresponds a unique conjugacy class in $W$.

An element of $\underline{\widetilde{W}}$ which corresponds to a pair $(\lambda,\mu)$, $\lambda=(\lambda_1
\geq \lambda_2\geq \ldots \geq \lambda_m)$ and $\mu=(\mu_1=\mu_2\geq \ldots \geq \mu_{2k-1}=\mu_{2k})$ is the class represented by the sum of the blocks in the following diagram (we use the notation of \cite{C}, Section 7)
\begin{equation}\label{sdn}
A_{\mu_1-1}+A_{\mu_3-1}+\ldots +A_{\mu_{2k-1}-1}+D_{\frac{\lambda_1+\lambda_2}{2}}(a_{\frac{\lambda_2}{2}-1})+
D_{\frac{\lambda_3+\lambda_4}{2}}(a_{\frac{\lambda_4}{2}-1})+\ldots +D_{\frac{\lambda_{m-1}+\lambda_{m}}{2}}(a_{\frac{\lambda_{m}}{2}-1}).
\end{equation}

Let $G'$ be the extension of $G$ by the Dynkin graph automorphism of order $2$. Then $G'$ is of type ${\rm O}(V)$. Denote by $\underline{\widetilde{\mathcal{N}}}(G)$ the set of unipotent classes of $G'$ contained in $G$.
The elements of $\underline{\widetilde{\mathcal{N}}}(G)$ are parameterized by partitions $\lambda$ of $2n$ for which $l_j(\lambda)$ is even for even $j$. Note that the number of parts of such partitions is even. We denote this set of partitions by $\mathcal{Q}_{2n}$. In case when $G={\rm SO}(V)$ the parts of $\lambda$ are just the sizes of the Jordan blocks in $V$ of the unipotent elements from the conjugacy class corresponding to $\lambda$. If $\lambda$ has only even parts then $\lambda$ corresponds to two unipotent classes in $G$ of the same dimension. In all other cases there is a unique unipotent class in $G$ which corresponds to $\lambda$.

One also has $\widehat{\underline{\mathcal{N}}}(G)=\underline{\mathcal{N}}(G_2)$, and $G_2$ is of type ${\rm SO}(V_2)$ where $V_2$ is a vector space of dimension $2n$ over an algebraically closed field of characteristic $2$ equipped with a non--degenerate bilinear form $(\cdot,\cdot)$ and a non--zero quadratic form $Q$ such that
$$
(x,y)=Q(x+y)-Q(x)-Q(y),~x,y\in V_2.
$$

Let $G'_2$ be the extension of $G_2$ by the Dynkin graph automorphism of order $2$. Then $G'_2$ is of type ${\rm O}(V_2)$. Denote by $\underline{\widetilde{\mathcal{N}}}(G_2)$ the set of unipotent classes of $G'_2$ contained in $G_2$. There is a natural injective homomorphism from ${\rm O}(V_2)$ to ${\rm Sp}(V_2)$, ${\rm dim}~V_2=2n$, and hence $\underline{\widetilde{\mathcal{N}}}(G_2) \simeq \widetilde{\mathcal{T}}_{2n}^2$, where $\widetilde{\mathcal{T}}_{2n}^2$ is the set of elements $(\lambda,\varepsilon)\in \mathcal{T}_{2n}^2$ such that $\lambda$ has an even number of parts.

Let $\widehat{\widetilde{W}}$ be the set of orbits of irreducible characters of $W$ under the action of $W'$.
Elements of $\widehat{\widetilde{W}}$ are parameterized by unordered pairs of partitions $(\alpha,\beta)$ written in non--decreasing order, $\alpha_1\leq \alpha_2\leq \ldots \leq \alpha_{\tau(\alpha)}$, $\beta_1\leq \beta_2 \leq \ldots \leq \beta_{\tau(\beta)}$, and such that $\sum \alpha_i+\sum \beta_i=n$. By adding zeroes we can assume that the length of $\alpha$ is equal to the length of $\beta$. The set of such pairs is denoted by $Y_{n,0}$.

Instead of the maps in (\ref{diag}) we shall describe the following maps
\begin{equation}\label{diagso2n}
\begin{array}{ccccccc}
   &  & \widetilde{\mathcal{X}}^1(W) & \stackrel{\widetilde{f}_1}{\longleftarrow} & \underline{\widetilde{\mathcal{N}}}(G) &  & \\
   &  & \downarrow ~{\rm in} &  & \downarrow ~\widetilde{\pi}^G &  & \\
  G & \stackrel{{\widetilde{\phi}_G}}{\longrightarrow} & {\widetilde{\mathcal{R}}(W)} & \stackrel{\widetilde{F}}{\longleftarrow} & {\underline{\widetilde{\mathcal{N}}}}(G_2) & {\hbox{\raise-1mm\hbox{${\textstyle \stackrel{\widetilde{\Phi}^W}{\longleftarrow}}\atop {\hbox{${\textstyle \longrightarrow}\atop {\scriptstyle \widetilde{\Psi}^W}$}}$}}} & \underline{\widetilde{W}},
\end{array}
\end{equation}
where $\widetilde{f}_1,\widetilde{F}$ are induced by the restrictions of the maps $f_1,F$ for $G'$, $G'_2$ to $\underline{\widetilde{\mathcal{N}}}(G)$, ${\underline{\widetilde{\mathcal{N}}}}(G_2)$, respectively, $\widetilde{\mathcal{X}}^1(W)$ and ${\widetilde{\mathcal{R}}(W)}$ are their images, ${\widetilde{\phi}_G},\widetilde{\Psi}^W,\widetilde{\Phi}^W, \widetilde{\pi}^G$ are also induced by the corresponding maps for $G'$, $G'_2$ and $W'$.

The map $\widetilde{f}_1$ is defined as in case of $B_n$. The image of $\widetilde{f}_1$ consists of all pairs $((c'_1,c'_3,\ldots,c'_{2m+1}),(c'_2,c'_4,\ldots,c'_{2m}))\in Y_{n,0}$ such that $c'_i\leq c'_{i+1}$ for all odd $i$ and $c'_i\leq c'_{i+1}+2$ for all even $i$.

If $(\lambda,\varepsilon)\in \widetilde{\mathcal{T}}_{2n}^2$, $\lambda = (\lambda_1
\leq \lambda_2\leq \ldots \leq \lambda_{2m})$ and $\widetilde{F}(\lambda,\varepsilon)=((c_1,c_3,\ldots,c_{2m-1}),(c_2,c_4,\ldots,c_{2m}))$ then the parts $c_i$ are defined by induction starting from $c_1$,
$$
  \begin{array}{ll}
    c_i=\frac{\lambda_i-2}{2}+2(i-1)-4\left[\frac{i-1}{2}\right] & \hbox{if $\lambda_i$ is even, $\varepsilon(\lambda_i)=1$  and $c_{i-1}$ is already defined;} \\
    c_i=\frac{\lambda_i-1}{2} +2(i-1)-4\left[\frac{i-1}{2}\right] & \hbox{if $\lambda_i=\lambda_{i+1}$ is odd and $c_{i-1}$ is already defined;} \\
    c_{i+1}=\frac{\lambda_i-3}{2}+2i-4\left[\frac{i}{2}\right]  & \hbox{if $\lambda_i=\lambda_{i+1}$ is odd and $c_{i}$ is already defined;}
\\
    c_i=\frac{\lambda_i}{2} +2(i-1)-4\left[\frac{i-1}{2}\right] & \hbox{if $\lambda_i=\lambda_{i+1}$ is even, $\varepsilon(\lambda_i)=0$ and $c_{i-1}$ is already defined;} \\
    c_{i+1}=\frac{\lambda_i}{2} +2(i-1)-4\left[\frac{i-1}{2}\right] & \hbox{if $\lambda_i=\lambda_{i+1}$ is even, $\varepsilon(\lambda_i)=0$ and $c_{i}$ is already defined.}
  \end{array}
$$

The image ${\widetilde{\mathcal{R}}(W)}$ of $\widetilde{F}$ consists of all pairs $((c_1,c_3,\ldots,c_{2m+1}),(c_2,c_4,\ldots,c_{2m}))\in Y_{n,0}$ such that $c_i\leq c_{i+1}$ for all odd $i$ and $c_i\leq c_{i+1}+4$ for all even $i$.

The maps $\widetilde{\Phi}^W$ and $\widetilde{\Psi}^W$ are defined by the same formulas as in case of $C_n$.

The map ${\widetilde{\pi}}^G$ is given by $\widetilde{\pi}^G(\lambda)=(\nu, \varepsilon')$, $\lambda = (\lambda_1
\leq \lambda_2\leq \ldots \leq \lambda_{2m})\in \mathcal{Q}_{2n}$, where
$$
\nu_i=\left\{
        \begin{array}{ll}
          \lambda_i-1 & \hbox{if $\lambda_i$ is odd, $i$ is even and $\lambda_{i-1}<\lambda_i$;} \\
          \lambda_i+1 & \hbox{if $\lambda_i$ and $i$ are odd, and $\lambda_{i}<\lambda_{i+1}$;}  \\
          \lambda_i & \hbox{otherwise},
        \end{array}
      \right.
$$
and
$$
\varepsilon'(k)=\left\{
                 \begin{array}{ll}
                 \omega  & \hbox{if $k$ is odd;} \\
                   0  & \hbox{if $k$ is even, there exists even $\lambda_i=k$ with odd $i$ such that $\lambda_{i-1}<\lambda_i$;} \\
1 & \hbox{otherwise.}
                 \end{array}
               \right.
$$

The map ${\widetilde{\pi}}^G$ is injective and its image consists of pairs $(\nu, \varepsilon)\in \widetilde{\mathcal{T}}_{2n}^2$ such that $\varepsilon(k)\neq 0$ if $\nu_k^*$ is odd and for each even $i$ such that $\nu_i^*$ is even we have $\nu_{i-1}^*=\nu_i^*$, i.e. $i-1$ does not appear in the partition $\nu$. Here $\nu_1^*\geq\nu_2^*\geq \ldots \geq \nu_{m}^*$ is the partition dual to $\nu$.

To describe ${\widetilde{\phi}_G}$ for $G={\rm SO}(V)$ we choose a sufficiently large $m \in \mathbb{N}$. Let $g \in G$. For any $x \in {\bf k}^*$ let $V_x$ be the
generalized $x$--eigenspace of $g:V \rightarrow V$. For any $x \in {\bf k}^*$ such that $x^2\neq 1$ let $\lambda^x_1
\geq \lambda^x_2
\geq \ldots \geq \lambda^x_{2m}$ be the sequence
in $\mathbb{N}$ whose terms are the sizes of the Jordan blocks of $x^{-1}g: V_x \rightarrow V_x$.

For any $x \in {\bf k}^*$ with $x^2= 1$ let $\lambda^x_1
\geq \lambda^x_2\geq \ldots \geq \lambda^x_{2m}$ be the sequence in $\mathbb{N}$, where $((\lambda^x_1
\geq \lambda^x_3\geq \ldots \geq \lambda^x_{2m-1}),(\lambda^x_2
\geq \lambda^x_4\geq \ldots \geq \lambda^x_{2m}))$ is the pair of partitions
such that the corresponding irreducible representation of the Weyl
group of type $D_{{\rm dim}~V_x/2}$ is the Springer representation attached to the unipotent
element $x^{-1}g \in {\rm SO}(V_x)$ and to the trivial local data.

Let $\lambda(g)$ be
the partition $\lambda(g)_1 \geq \lambda(g)_2 \geq \ldots \geq \lambda(g)_{2m+1}$ given by $\lambda(g)_j = \sum_{x} \lambda^x_j$, where $x$ runs over a set of representatives for the orbits of the
involution $a\mapsto a^{-1}$ of ${\bf k}^*$.
Now $\widetilde{\phi}_G(g)$ is the pair of partitions $((\lambda(g)_1
\geq \lambda(g)_3\geq \ldots \geq \lambda(g)_{2m-1}),(\lambda(g)_2
\geq \lambda(g)_4\geq \ldots \geq \lambda(g)_{2m}))$.

The preimage $\widetilde{\phi}_G^{-1}(\lambda,\mu)$ is a stratum in $G$ in all cases except for the one when the pair $(\lambda,\mu)$ is of the form $((\lambda_1
\geq \lambda_3\geq \ldots \geq \lambda_{2m-1}),(\lambda_1
\geq \lambda_3\geq \ldots \geq \lambda_{2m-1}))$. In that case $\widetilde{\phi}_G^{-1}(\lambda,\mu)$ is a union of two strata, and the conjugacy classes in each of them have the same dimension.

If $g$ is any element in the stratum $G_{(\lambda,\varepsilon)}$ corresponding to a pair $(\lambda,\varepsilon)\in \widetilde{\mathcal{T}}_{2n}^2$, $\lambda=(\lambda_1
\geq \lambda_2\geq \ldots \geq \lambda_m)$ then the dimension of the centralizer of $g$ in $G$ is given by the following formula
\begin{equation}\label{dimsoev}
{\rm dim}~Z_G(g)=n+\sum_{i=1}^m(i-1)\lambda_i-\frac{1}{2}|\{i:\lambda_i ~\hbox{is odd}\}|-|\{i:\lambda_i ~\hbox{is even and}~\varepsilon(\lambda_i)=1\}|.
\end{equation}


\setcounter{equation}{0}
\setcounter{theorem}{0}

\section{The strict transversality condition}\label{stt}

In this section for every conjugacy class $\mathcal{C} \in C(W)$ we define a transversal slice $\Sigma_s$, $s\in \mathcal{C}$ such that every conjugacy class $\mathcal{O}\in G_\mathcal{C}$ intersects $\Sigma_s$ and
\begin{equation}\label{str1}
{\rm dim}~\mathcal{O}={\rm codim}~\Sigma_s.
\end{equation}
Recall that the definition of $\Sigma_s$ is given in terms of a system of positive roots $\Delta_+$ associated to $s$, and $\Delta_+$ depends on the choice of ordering of terms in decomposition (\ref{hdec}). It turns out that in order to fulfill condition (\ref{str1}) the subspaces $\h_i$ in (\ref{hdec}) should be ordered in such a way that if $\h_i=\h_\lambda^k$, $\h_j=\h_\mu^l$ and $0\leq\lambda <\mu< 1$ then $i<j$, where $\lambda$ and $\mu$ are eigenvalues of the corresponding matrix $I-M$ for $s$. In case of exceptional root systems this is verified using a computer program, and in case of classical root systems this is confirmed by explicit computation based on a technical lemma. In order to formulate this lemma we recall realizations of classical irreducible root systems.

Let $V$ be a real Euclidean $n$--dimensional vector space with an orthonormal basis $\varepsilon_1,\ldots ,\varepsilon_n$. The root systems of types $A_{n-1},B_n,C_n$ and $D_n$ can be realized in $V$ as follows.

\subsection*{$\bf A_n$} The roots are $\varepsilon_i-\varepsilon_j$, $1\leq i,j\leq n$, $i\neq j$, $\h_\mathbb{R}$ is the hyperplane in $V$ consisting of the points the sum of whose coordinates is zero.

\subsection*{$\bf B_n$} The roots are $\pm\varepsilon_i\pm\varepsilon_j$, $1\leq i<j\leq n$, $\pm \varepsilon_i$, $1\leq i\leq n$, $\h_\mathbb{R}=V$.

\subsection*{$\bf C_n$} The roots are $\pm\varepsilon_i\pm\varepsilon_j$, $1\leq i<j\leq n$, $\pm 2\varepsilon_i$, $1\leq i\leq n$, $\h_\mathbb{R}=V$.

\subsection*{$\bf D_n$} The roots are $\pm\varepsilon_i\pm\varepsilon_j$, $1\leq i<j\leq n$, $\h_\mathbb{R}=V$.

In all cases listed above the corresponding Weyl group $W$ is a subgroup of the Weyl group of type $C_n$ acting on the elements of the basis $\varepsilon_1,\ldots ,\varepsilon_n$ by permuting the basis vectors and changing the sign of an arbitrary subset of them. Each element $s\in W$ can be expressed as a product of disjoint cycles of the form
$$
\varepsilon_{k_1}\rightarrow \pm \varepsilon_{k_2}\rightarrow \pm \varepsilon_{k_3}\rightarrow \ldots \rightarrow  \pm \varepsilon_{k_r}\rightarrow \pm \varepsilon_{k_1}.
$$
The cycle above is of length $r$; it is called positive if $s^r(\varepsilon_{k_1})=\varepsilon_{k_1}$ and negative if $s^r(\varepsilon_{k_1})=-\varepsilon_{k_1}$. The lengths of the cycles together with their signs give a set of positive or negative integers called the signed cycle-type of $s$. To each positive cycle of $s$ of length $r$ there corresponds a pair of positive orbits $X,-X$, $|X|=r$, for the action of the group $<s>$ generated by $s$ on the set $E=\{\varepsilon_1,\ldots, \varepsilon_n,-\varepsilon_1,\ldots,-\varepsilon_n\}$, and to each negative cycle of $s$ of length $r$ there corresponds a negative orbit $X$, $|X|=2r$, for the action of $<s>$ on $E$. A positive cycle of length $1$ is called trivial. It corresponds to a pair of fixed points for the action of $<s>$ on $E$.

Now we formulate the main lemma.
\begin{lemma}\label{mainl}
Let $s$ be an element of the Weyl group of type $C_n$ operating on the set $E=\{\varepsilon_1,\ldots, \varepsilon_n,-\varepsilon_1,\ldots,-\varepsilon_n\}$ as indicated in Section \ref{luspart}, where $\varepsilon_1,\ldots ,\varepsilon_n$ is the basis of $V$ introduced above. Assume that $s$ has either only one nontrivial cycle of length $k/2$ ($k$ is even), which is negative, or only one nontrivial cycle of length $k$, which is positive, $1<k\leq n$. Let $\Delta$ be a root system of type $A_{n-1},B_n,C_n$ or $D_n$ realized in $V$ as above.

If $s$ has only one nontrivial cycle of length $k/2$, which is negative, then $k$ is even, the spectrum of $s$ in the complexification $V_\mathbb{C}$ of $V$ is $\epsilon_r=\exp(\frac{2\pi i (k-2r+1)}{k})$, $r=1,\ldots,k/2$, and possibly $\epsilon_0=1$, all eigenvalues are simple except for possibly $1$.

If $s$ only has one nontrivial cycle of length $k$, which is positive, then the spectrum of $s$ in the complexification of $V$ is $\epsilon_r=\exp(\frac{2\pi i (k-r)}{k})$, $r=1,\ldots,k-1$, and $\epsilon_0=1$, all eigenvalues are simple except for possibly $1$.

In both cases we denote by $V_r$ the invariant subspace in $V$ which corresponds to $\epsilon_r=\exp(\frac{2\pi i ([k/2]+1-r)}{k})$, $r=1,\ldots,\left[\frac{k}{2} \right]$ or $\epsilon_0=1$ in case of a positive nontrivial cycle and to $\epsilon_r=\exp(\frac{2\pi i (2\left[\frac{k/2+1}{2} \right]+1-2r)}{k})$, $r=1,\ldots,\left[\frac{k/2+1}{2} \right]$ or $\epsilon_0=1$ in case of a negative cycle. For $r\neq 0$ the space $V_r$ is spanned by the real and the imaginary parts of a nonzero eigenvector of $s$ in $V_\mathbb{C}$ corresponding to $\epsilon_r$, and $V_0$ is the subspace of fixpoints of $s$ in $V$.

$V_r$ is two--dimensional if $\epsilon_r\neq \pm 1$, one--dimensional if $\epsilon_r=-1$ or may have arbitrary dimension if $\epsilon_r=1$.

Let $\Delta_+$ be a system of positive roots associated to $s$ and defined as in Section \ref{trans}, where we use the decomposition
\begin{equation}\label{decv}
V=\bigoplus_{i} V_i
\end{equation}
as (\ref{hdec}) in the definition of $\Delta_+$. Denote by ${\overline{\Delta}}_{i}\subset \Delta$ the corresponding subsets of roots defined as in (\ref{di}).

Let $\Delta_0^s$ be the root subsystem fixed by the action of $s$ and $l(s)$ the number of positive roots which become negative under the action of $s$. Then if $s$ has only one nontrivial cycle of length $k$, which is positive, we have
\begin{enumerate}
\item if $\Delta=A_{n-1}$ then $\Delta_0^s=A_{n-k-1}$, $l(s)=2n-k-1$;

\item if $\Delta=B_{n}(C_n)$ then $\Delta_0^s=B_{n-k}(C_{n-k})$, $l(s)=4n-2k$ for odd $k$ and $l(s)=4n-2k+1$ for even $k$;

\item if $\Delta=D_{n}$ then $\Delta_0^s=D_{n-k}$, $l(s)=4n-2k-2$ for odd $k$ and $l(s)=4n-2k-1$ for even $k$.
\end{enumerate}

If $s$ has only one nontrivial cycle of length $\frac{k}{2}$,  which is negative, we have

\begin{enumerate}
\item if $\Delta=B_{n}(C_n)$ then $\Delta_0^s=B_{n-k/2}(C_{n-k/2})$, $l(s)=2n-k/2$;

\item if $\Delta=D_{n}$ then $\Delta_0^s=D_{n-k/2}$, $l(s)=2n-k/2-1$.
\end{enumerate}

If $s$ has only one nontrivial cycle of length $k$, which is positive, $\Delta$ is of type $B_n,C_n$ or $D_n$, and $k$ is even then
$\Delta={\overline{\Delta}}_{k/2}\bigcup {\overline{\Delta}}_{k/2-1}\bigcup \Delta_0^s$ (disjoint union),
and all roots in ${\overline{\Delta}}_{k/2-1}$ are orthogonal to the fixed point subspace for the action of $s$ on $V$.

In all other cases $\Delta={\overline{\Delta}}_{i_{\rm max}}\bigcup \Delta_0^s$ (disjoint union), where $i_{\rm max}$ is the maximal possible index $i$ which appears in decomposition (\ref{decv}).

\end{lemma}

\begin{proof}
The proof is similar in all cases. We only give details in the most complicated case when $s$ has only one nontrivial cycle, which is positive, $\Delta$ is of type $B_n(C_n)$, and $k$ is even. Without loss of generality one can assume that $s$ corresponds to the cycle of the form
$$
\varepsilon_1\rightarrow\varepsilon_2\rightarrow\varepsilon_4\rightarrow\varepsilon_6\rightarrow \cdots \rightarrow\varepsilon_{k-2}\rightarrow\varepsilon_k\rightarrow\varepsilon_{k-1}\rightarrow\varepsilon_{k-3}\rightarrow \cdots \rightarrow\varepsilon_3\rightarrow\varepsilon_1~(k>2),~ \varepsilon_1\rightarrow\varepsilon_2\rightarrow\varepsilon_1~(k=2).
$$

From this definition one easily sees that $\Delta_0^s=B_{n-k}(C_{n-k})=\Delta\bigcap V'$, where $V'\subset V$ is the subspace generated by $\varepsilon_{k+1},\ldots,\varepsilon_n$. Computing the eigenvalues of $s$ in $V_\mathbb{C}$ is a standard exercise in linear algebra. The eigenvalues are expressed in terms of the exponents of the root system of type $A_{k-1}$ (see \cite{Car}, Ch. 10).

The invariant subspace $V_r$ is spanned by the real and the imaginary parts of a nonzero eigenvector of $s$ in $V_\mathbb{C}$ corresponding to the eigenvalue $\epsilon_r$.
If $\epsilon_r\neq \pm 1$ then $V_r$ is two--dimensional, and for $\epsilon_r=-1$ $V_r$ is one--dimensional. In the former case $V_r$ will be regarded as the real form of a complex plane with the orthonormal basis $1,i$. Under this convention the orthogonal projection operator onto $V_r$ acts on the basic vectors $\varepsilon_j$ as follows
\begin{equation}\label{pr}
\varepsilon_{2j+1}\mapsto c\epsilon_r^j,~j=0,\ldots,\frac{k}{2}-1,\varepsilon_{2j}\mapsto c\epsilon_r^{-j},~j=1,\ldots,\frac{k}{2},
\end{equation}
where $c=\sqrt{\frac{2}{k}}$.
Consider the case when $k>2$; the case $k=2$ can be analyzed in a similar way.

To compute $l(s)$ using the definition of $\Delta_+$ given in Section \ref{trans} one should first look at all roots which have nonzero projections onto $V_{k/2}$ on which $s$ acts by rotation with the angle $\frac{2\pi }{k}$.

From (\ref{pr}) we deduce that the roots which are not fixed by $s$ and have zero orthogonal projections onto $V_{k/2}$ are $\pm(\varepsilon_j+\varepsilon_{k-j+1})$, $j=1,\ldots \frac{k}{2}$. The number of those roots is equal to $k$, and they all have nonzero orthogonal projections onto $V_{k/2-1}$. From (\ref{pr}) we also obtain that all the other roots which are not fixed by $s$ have nonzero orthogonal projections onto $V_{k/2}$, hence $|{\overline{\Delta}}_{k/2-1}|=k$. The number of roots fixed by $s$ is $2(n-k)^2$ since it is equal to the number of roots in $\overline{\Delta}_0=\Delta_0^s=B_{n-k}(C_{n-k})$. Hence $\Delta={\overline{\Delta}}_{k/2}\bigcup {\overline{\Delta}}_{k/2-1}\bigcup {\overline{\Delta}}_0$ (disjoint union), the number of roots in ${\overline{\Delta}}_{k/2}$ is   $|\Delta|-|{\overline{\Delta}}_0|-|{\overline{\Delta}}_{k/2-1}|=2n^2-2(n-k)^2-k=4nk-2k^2-k$, $|{\overline{\Delta}}_{k/2}|=4nk-2k^2-k$.

Now using the symmetry of the root system $\Delta$ as a subset of $V$ and the fact that $s$ acts as rotation by the angles $\frac{2\pi }{k}$ and $\frac{4\pi }{k}$ in $V_{k/2}$ and $V_{k/2-1}$, respectively, we deduce that the number of positive toots in ${\overline{\Delta}}_{k/2}$ (${\overline{\Delta}}_{k/2-1}$) which become negative under the action of $s$ is equal to the number of roots in ${\overline{\Delta}}_{k/2}$ (${\overline{\Delta}}_{k/2-1}$) divided by the order of $s$ in $V_{k/2}(V_{k/2-1})$. Therefore
$$
l(s)=\frac{|{\overline{\Delta}}_{k/2}|}{k}+\frac{|{\overline{\Delta}}_{k/2-1}|}{k/2}=\frac{4nk-2k^2-k}{k}+\frac{k}{k/2}=
4n-2k+1.
$$
This completes the proof in the considered case.
\end{proof}

Now we are in a position to prove the main statement of this paper.
\begin{theorem}\label{mainth}
Let $G$ be a connected semisimple algebraic group over an algebraically closed filed $\bf k$ of characteristic good for $G$, and $\mathcal{O}\in \widehat{\underline{\mathcal{N}}}(G)$. Let $H$ be a maximal torus of $G$, $W$ the Weyl group of the pair $(G,H)$, and $s\in W$ an element from the conjugacy class $\Psi^W(\mathcal{O})$. Let $\Delta$ be the root system of the pair $(G,H)$ and $\Delta_+$ the system of positive roots in $\Delta$ associated to $\Psi^W(\mathcal{O})$ and defined in Section \ref{trans} with the help of decomposition (\ref{hdec}), where the subspaces $\h_i$ are ordered in such a way that if $\h_i=\h_\lambda^k$, $\h_j=\h_\mu^l$ and $0\leq\lambda <\mu< 1$ then $i<j$. Then all conjugacy classes in the stratum $G_\mathcal{O}=\phi_G^{-1}(F(\mathcal{O}))$ intersect the corresponding transversal slice $\Sigma_s$ at some points of the subvariety $\dot{s}H_0N_s$, where $H_0\subset H$ is the centralizer of $\dot{s}$  in $H$, and for any $g\in G_\mathcal{O}$
\begin{equation}\label{dimid}
{\rm dim}~Z_G(g)={\rm dim}~\Sigma_s={\rm codim}_{G_p}~\mathcal{O},~\mathcal{O}\in \underline{\mathcal{N}}(G_p)\subset \widehat{\underline{\mathcal{N}}}(G).
\end{equation}
\end{theorem}

\begin{proof}
First we compute the dimension of the slice $\Sigma_s$, $s\in \Psi^W(\mathcal{O})$ and justify that for any $g\in G_\mathcal{O}$ equality (\ref{dimid}) holds.

Observe that by the definition of the slice $\Sigma_s$
$$
{\rm dim}~\Sigma_s=\underline{l}(s)+|\overline{\Delta}_0|+{\rm dim}~\h_0.
$$
Hence to compute ${\rm dim}~\Sigma_s$ we have to find all numbers in the right hand side of the last equality.

Consider the case of classical Lie algebras when each Weyl group element is a product of cycles in a permutation group. In this case the identity (\ref{dimid}) is proved by a straightforward calculation using Lemma \ref{mainl}.

If $G$ is of type $A_n$ let $s$ be a representative in the conjugacy class of the Weyl group which corresponds to a partition $\lambda=(\lambda_1\geq \lambda_2\geq \ldots\geq \lambda_m)$. The particular ordering of the invariant subspaces $\h_i$ in the formulation of this theorem implies that the length $\underline{l}(s)$ should be computed by successive application of Lemma \ref{mainl} to the cycles $s_i$ of $s$, which correspond to $\lambda_i$ placed in a non--increasing order.

We should first apply Lemma \ref{mainl} to the cycle $s_1$ of $s$ which corresponds to the maximal part $\lambda_1$. In this case $l(s_1)=2n-\lambda_1+1$ and $\Delta_0^{s_1}=A_{n-\lambda_1}=\Delta\setminus \overline{\Delta}_{i_T}$ in the notation of Section \ref{trans}. The remaining cycles $s_2,\ldots ,s_m$ of $s$ corresponding to $\lambda_2\geq \ldots\geq \lambda_m$ act on $\Delta_0^{s_1}$, and we can apply Lemma \ref{mainl} to $s_2$ acting on $\Delta_0^{s_1}$ to get $l(s_2)=2(n-\lambda_1)-\lambda_2+1$ and $\Delta_0^{s_2}=A_{n-\lambda_1-\lambda_2}=\Delta\setminus \left(\overline{\Delta}_{i_T}\bigcup \overline{\Delta}_{i_{T-1}}\right)$. Iterating this procedure we obtain that
\begin{equation}\label{ls}
\underline{l}(s)= \sum l(s_k),~l(s_k)=2(n-\sum_{i=1}^{k-1}\lambda_i)-\lambda_k+1,
\end{equation}
where $\underline{l}(s)$ is the length of $s$ and the first sum in (\ref{ls}) is taken over $k$ for which $\lambda_k>1$.

The number of roots fixed by $s$ can be represented in a similar form,
\begin{equation}\label{fx}
 |\overline{\Delta}_0|= \sum l(s_k),~l(s_k)=2(n-\sum_{i=1}^{k-1}\lambda_i)-\lambda_k+1,
\end{equation}
where the sum in (\ref{fx}) is taken over $k$ for which $\lambda_k=1$.

Finally the dimension of the fixed point space $\h_0$ of $s$ in $\h$ is $m-1$, ${\rm dim}~\h_0=m-1$.

Observe now that
\begin{equation}\label{dimsect}
 {\rm dim}~\Sigma_s=\underline{l}(s)+|\overline{\Delta}_0|+{\rm dim}~\h_0,
\end{equation}
and hence
$$
{\rm dim}~\Sigma_s=\sum_{k=1}^m l(s_k)+m-1=\sum_{k=1}^m\left(2(n-\sum_{i=1}^{k-1}\lambda_i)-\lambda_k+1\right)+m-1.
$$

Exchanging the order of summation and simplifying this expression we obtain that
$$
{\rm dim}~\Sigma_s=n+2\sum_{i=1}^m(i-1)\lambda_i
$$
which coincides with (\ref{dimsl}).

The computations of ${\rm dim}~\Sigma_s$ in case of $B_n$ and of $C_n$ are the same. If $(\nu, \varepsilon)\in \mathcal{T}_{2n}^2$, $\nu=(\nu_1\geq \nu_2\geq \ldots \geq \nu_m)$, corresponds to $\mathcal{O}\in \widehat{\underline{\mathcal{N}}}(G)=\underline{\mathcal{N}}(G_2)$ then $\Psi^W(\nu, \varepsilon)=(\lambda,\mu)\in \mathcal{A}^1_{2n}\simeq\underline{W}$ is defined in Section \ref{luspart}, part $\bf C_n$. $\lambda$ consists of even parts $\nu_i$ of $\nu$ for which $\varepsilon(\nu_i)=1$, and $\mu$ consists of all odd parts of $\nu$ and of even parts $\nu_i$ of $\nu$ for which $\varepsilon(\nu_i)=0$, the last two types of parts appear in pairs of equal parts. Let $s$ be a representative in the conjugacy class $\Psi^W(\nu, \varepsilon)$. Then  each part $\lambda_i$ corresponds to a negative cycle of $s$ of length $\frac{\lambda_i}{2}$, and each pair $\mu_i=\mu_{i+1}$ of equal parts of $\mu$ corresponds to a positive cycle of $s$ of length $\mu_i$. We order the cycles $s_k$ of $s$ associated to the (pairs of equal) parts of the partition $\nu$ in a way compatible with a non--increasing ordering of the parts of the partition $\nu=(\nu_1\geq \nu_2\geq \ldots \geq \nu_m)$, i.e. if we denote by $s_k$ the cycle that corresponds to an even part $\nu_k$ of $\nu$ for which $\varepsilon(\nu_k)=1$ or to a pair $\nu_k=\nu_{k+1}$ of odd parts of $\nu$ or of even parts of $\nu$ for which $\varepsilon(\nu_k)=0$ then $s_k\geq s_l$ if $\nu_k\geq \nu_l$.

Similarly to the case of $A_n$, by the definition of $\Delta_+$ and by Lemma \ref{mainl} applied iteratively to the cycles $s_k$ in the order defined above, the length $\underline{l}(s)$ of $s$ is the sum of the following terms $l(s_k)$.

To each even part $\nu_k$ of $\nu$ for which $\varepsilon(\nu_k)=1$ we associate the term $$l(s_k)=2(n-\sum_{i=1}^{k-1}\frac{\nu_i}{2})-\frac{\nu_k}{2};$$

to each pair of odd parts $\nu_k=\nu_{k+1}>1$ we associate the term
$$l(s_k)=4(n-\sum_{i=1}^{k-1}\frac{\nu_i}{2})-2\nu_k=
\left(2(n-\sum_{i=1}^{k-1}\frac{\nu_i}{2})-\frac{\nu_k}{2}\right)+\left(2(n-\sum_{i=1}^{k}\frac{\nu_i}{2})-
\frac{\nu_{k+1}}{2}\right);$$
note that the sum of these terms over all pairs $\nu_k=\nu_{k+1}=1$ gives the number $|\overline{\Delta}_0|$ of the roots fixed by $s$;

to each pair of even parts $\nu_k=\nu_{k+1}$ for which $\varepsilon(\nu_k)=0$ we associate the term
$$l(s_k)=4(n-\sum_{i=1}^{k-1}\frac{\nu_i}{2})-2\nu_k+1=
\left(2(n-\sum_{i=1}^{k-1}\frac{\nu_i}{2})-\frac{\nu_k}{2}+\frac{1}{2}\right)+\left(2(n-\sum_{i=1}^{k}\frac{\nu_i}{2})-
\frac{\nu_{k+1}}{2}+\frac{1}{2}\right).$$

The dimension of the fixed point space $\h_0$ of $s$ in $\h$ is equal to a half of the sum of the number of all even parts $\nu_k$ for which $\varepsilon(\nu_k)=0$ and of the number of all odd parts $\nu_k$,

\begin{equation}\label{hofix}
{\rm dim}~\h_0= \frac{1}{2}|\{i:\nu_i ~\hbox{is odd}\}|+\frac{1}{2}|\{i:\nu_i ~\hbox{is even and}~\varepsilon(\nu_i)=0\}|.
\end{equation}

Finally substituting all the computed contributions into formula (\ref{dimsect}) we obtain
\begin{eqnarray*}
{\rm dim}~\Sigma_s = \sum_{k=1}^m\left(2(n-\sum_{i=1}^{k-1}\frac{\nu_i}{2})-\frac{\nu_k}{2}\right)+\frac{1}{2}|\{i:\nu_i ~\hbox{is even and}~\varepsilon(\nu_i)=0\}|+ \\
    +\frac{1}{2}|\{i:\nu_i ~\hbox{is odd}\}|+\frac{1}{2}|\{i:\nu_i ~\hbox{is even and}~\varepsilon(\nu_i)=0\}|.
\end{eqnarray*}

Exchanging the order of summation and simplifying this expression we obtain that
\begin{equation}\label{dimsp1}
{\rm dim}~\Sigma_s= n+\sum_{i=1}^m(i-1)\nu_i+\frac{1}{2}|\{i:\nu_i ~\hbox{is odd}\}|+|\{i:\nu_i ~\hbox{is even and}~\varepsilon(\nu_i)=0\}|
\end{equation}
which coincides with (\ref{dimsp}) or (\ref{dimsoodd}).

In case of $D_n$ the number ${\rm dim}~\Sigma_s$ can be easily obtained if we observe that the map $\widetilde{\Psi}^W$ is defined by the same formula as $\Psi^W$ in case of $C_n$. In case when $\widetilde{\Psi}^W(\nu,\varepsilon)=(-,\mu)$, where all parts of $\mu$ are even, there are two conjugacy classes in $W$ which correspond to $\widetilde{\Psi}^W(\nu,\varepsilon)$. However, the numbers $\underline{l}(s)$, $|\overline{\Delta}_0|$ and ${\rm dim}~\h_0$ are the same in both cases. They only depend on $\widetilde{\Psi}^W(\nu,\varepsilon)$ in all cases. Let $s\in W$ be a representative from the conjugacy class $\widetilde{\Psi}^W(\nu,\varepsilon)$, $\nu=(\nu_1\geq \nu_2\geq \ldots \geq \nu_m)$.

From Lemma \ref{mainl} we deduce that in case of $D_n$ the contributions of the cycles $s_k$ of $s$ to the formula for ${\rm dim}~\Sigma_s$ can be obtained from the corresponding contributions in case of $C_n$ in the following way:
for each pair of odd parts $\nu_k=\nu_{k+1}$ and for each pair of even parts $\nu_k=\nu_{k+1}$ with $\varepsilon(\nu_k)=0$ the corresponding contribution $l(s_k)$ to $\underline{l}(s)$ should be reduced by 2 and for each even part $\nu_k$ of $\nu$ with $\varepsilon(\nu_k)=1$ the corresponding contribution $l(s_k)$ to $\underline{l}(s)$ should be reduced by 1. This observation and formula (\ref{dimsp1}) yield
\begin{eqnarray*}
{\rm dim}~\Sigma_s= n+\sum_{i=1}^m(i-1)\nu_i+\frac{1}{2}|\{i:\nu_i ~\hbox{is odd}\}|+|\{i:\nu_i ~\hbox{is even and}~\varepsilon(\nu_i)=0\}|- \\
-|\{i:\nu_i ~\hbox{is odd}\}|-|\{i:\nu_i ~\hbox{is even}\}|= \\
= n+\sum_{i=1}^m(i-1)\nu_i-\frac{1}{2}|\{i:\nu_i ~\hbox{is odd}\}|-|\{i:\nu_i ~\hbox{is even and}~\varepsilon(\nu_i)=1\}|
\end{eqnarray*}
which coincides with (\ref{dimsoev}).

In case of root systems of exceptional types ${\rm dim}~\Sigma_s$ can be found in the tables in Appendix 1. According to those tables equality (\ref{dimid}) holds in all cases.

Now we show that all conjugacy classes in the stratum $G_\mathcal{O}=\phi_G^{-1}(F(\mathcal{O}))$ intersect the corresponding transversal slice $\Sigma_s$, $s\in \Psi^W(\mathcal{O})$. Fix an arbitrary system of positive roots $\Delta_+^1$ in $\Delta=\Delta(G,H)$. Let $B_1$ be the Borel subgroup in $G$ corresponding to $\Delta_+^1$. By \cite{Gk}, 3.1.14 one can find a representative $w\in \Psi^W(\mathcal{O})$ of minimal length with respect to the system of simple roots in $\Delta_+^1$ such that there is a parabolic subgroup $P_1\supset B_1$ of $G$ for which $w$ is elliptic in the Weyl group $W(L_1,H)$ of the pair $(L_1,H)$, where $H\subset B_1$ is a maximal torus contained in $B_1$ and $L_1$ is the Levi factor of $P_1$. One can always find a representative $\dot{w}\in L_1$. The stratum $G_\mathcal{O}$ consists of all conjugacy classes of minimal possible dimension which intersect the Bruhat cell $B_1\dot{w}B_1$. Denote by $U_1$ the unipotent radical of $P_1$. Then by the definition of parabolic subgroups one can always find a one parameter subgroup $\rho:{\bf k}^*\rightarrow Z_G(L_1)$ such that
\begin{equation}\label{lm}
\lim_{t\rightarrow 0}\rho(t)n\rho(t^{-1})=1
\end{equation}
for any $n\in U_1$.

Now denote by $B_2=B_1\bigcap L_1$ the Borel subgroup in $L_1$, and let $\gamma \in G_\mathcal{O}$ be a conjugacy class which intersects $B_1\dot{w}B_1$ at point $b\dot{w}b'$, $b,b'\in B_1$ such that $b\dot{w}b'\not \in B_2\dot{w}B_2$. Since by definitions of $B_2$ and $U_1$ we have $B_1=B_2U_1$ there are unique factorizations $b=un$, $b'=u'n'$, $u,u'\in B_2$, $n,n'\in U_1$.
By (\ref{lm}) we have
$$
\lim_{t\rightarrow 0}\rho(t)b\dot{w}b'\rho(t^{-1})=\lim_{t\rightarrow 0}u\rho(t)n\rho(t^{-1})\dot{w}u'\rho(t)n'\rho(t^{-1})=u\dot{w}u'\in B_2\dot{w}B_2,
$$
and hence the closure of $\gamma$ contains a conjugacy class $\gamma'$ which intersects $B_1\dot{w}B_1$ at some point of $B_2\dot{w}B_2\subset B_1\dot{w}B_1$. In particular, ${\rm dim}~\gamma > {\rm dim}~\gamma'$. This is impossible by the definition of $G_\mathcal{O}$, and hence $\gamma$ intersects $B_1\dot{w}B_1$ at some point of $B_2\dot{w}B_2\subset B_1\dot{w}B_1$. Without loss of generality we can assume that $w$ is of minimal length in its conjugacy class in $W(L_1,H)$, with respect to the system of simple positive roots associated to $B_2$ (see \cite{Gk}, 3.1.14).

Let $M_1$ be the semisimple part of $L_1$ and $\mathcal{O}_n=\Phi^{W_1}(\mathcal{O}_w)\subset M_1$, where $\mathcal{O}_w$ is the conjugacy class of $w$ in the Weyl group $W_1=W(L_1,H)$. We claim that $\mathcal{O}_n\in {\underline{{\mathcal{N}}}}(M_1)$. This is a consequence of the fact that $w$ is elliptic. Indeed, in case when $M_1$ is of type $A_n$ this is obvious since ${\underline{\widehat{\mathcal{N}}}}(M_1)$ contains only unipotent classes. In fact in this case $\mathcal{O}_w$ is the Coxeter class, and $\mathcal{O}_n$ is the class of regular unipotent elements. If $M_1$ is of type $B_n$, $C_n$ or $D_n$, formula (\ref{hofix}) implies that if $\mathcal{O}_n$ corresponds to $(\nu,\varepsilon)\in \mathcal{T}^2_{2n}(\widetilde{\mathcal{T}}^2_{2n})$ then $\nu$ has no odd parts and no even parts $\nu_i$ with $\varepsilon(\nu_i)=0$. According to the description given in the previous section the map ${{\pi}}^{M_1}({\widetilde{\pi}}^{M_1})$ is injective and its image consists of pairs $(\nu, \varepsilon)\in {\mathcal{T}}_{2n}^2(\widetilde{\mathcal{T}}_{2n}^2)$ such that $\varepsilon(k)\neq 0$ if $\nu_k^*$ is odd and for each even $i$ such that $\nu_i^*$ is even we have $\nu_{i-1}^*=\nu_i^*$, i.e. $i-1$ does not appear in the partition $\nu$. We deduce that $\mathcal{O}_n$ is contained in the image of ${{\pi}}^{M_1}({\widetilde{\pi}}^{M_1})$, i.e. $\mathcal{O}_n\in {\underline{{\mathcal{N}}}}(M_1)$ is a unipotent class in $M_1$. In case when $M_1$ is of exceptional type this can be checked by examining the tables in Appendix 1.

Therefore $\mathcal{O}_n$ is the unipotent class of minimal possible dimension which intersects $B_2\dot{w}B_2$.
By Theorem 0.7 in \cite{L4} the codimension of $\mathcal{O}_n$ in $M_1$ is equal to $\underline{l}_1(w)$, where $\underline{l}_1$ is the length function in $W_1$ with respect to the system of simple roots in $\Delta(L_1,H)_+=\Delta_+^1\bigcap\Delta(L_1,H)$,
\begin{equation}\label{cdm}
{\rm codim}_{M_1}~\mathcal{O}_n=\underline{l}_1(w).
\end{equation}

Now let $\Delta_+$ be the positive root system associated to $s\in \Psi^W(\mathcal{O})$ as in the formulation of this theorem.
Let $L_2\subset G$ be the subgroup of Levi type such that $s$ is elliptic in the Weyl group of the pair $(L_2,H)$, $H\subset L_2$. $L_2$ exists since $w$ and $s$ are in the same Weyl group conjugacy class.
We claim that $s$ has minimal length in the Weyl group $W_2=W(L_2,H)$ with respect to the system of positive simple roots associated to $B_3$, where $B_3=B\bigcap L_2$, and $B$ is the Borel subgroup associated to $\Delta_+$.

Indeed, let $M_2$ be the semisimple part of $L_2$, and $\Sigma_s'=\Sigma_s\bigcap M_2$. Note that by the definition of $s$ and $w$ we have the following isomorphisms $W_1\simeq W_2$, $M_1\simeq M_2$. Let $\mathcal{O}_n'\subset M_2$ be the image in $M_2$ of the unipotent class $\mathcal{O}_n$ under the latter isomorphism.

By the definition we have $\mathcal{O}_n'=\Phi^{W_2}(\mathcal{O}_s)$, where $\mathcal{O}_s$ is the conjugacy class of $s$ in the Weyl group $W_2=W(L_2,H)$.
Therefore formula (\ref{dimid}) confirmed in the first part of this proof, formula (\ref{dimsect}) and the fact that $s$ is elliptic in $W_2$ imply that
$$
{\rm codim}_{M_1}~\mathcal{O}_n={\rm codim}_{M_2}~\mathcal{O}_n'={\rm dim}~\Sigma_s'=\underline{l}_2(s),
$$
where $\underline{l}_2$ is the length function in $W_2$ with respect to the system of simple roots in $\Delta(L_2,H)_+=\Delta_+\bigcap\Delta(L_2,H)$.

The last formula and (\ref{cdm}) yield $\underline{l}_1(w)=\underline{l}_2(s)$, and hence $s$ has minimal possible length in its conjugacy class in $W_2$ with respect to the system of simple roots in $\Delta(L_2,H)_+$. Therefore without loos of generality (see \cite{Gk}, 3.1.14) we can assume that $w$ is chosen in such a way that $vwv^{-1}=s$, $\dot{v}M_1\dot{v}^{-1}=M_2$ and $v(\Delta(L_1,H)_+)=\Delta(L_2,H)_+$ for some $v\in W$. In particular, $\dot{v}B_2\dot{v}^{-1}=B_3$, and  $\dot{v}B_2\dot{w}B_2\dot{v}^{-1}=B_3\dot{s}B_3$

Now by the property of $w$ proved above any conjugacy class $\gamma \in G_\mathcal{O}$ intersects $B_2\dot{w}B_2$, and hence it also intersects $B_3\dot{s}B_3\subset B\dot{s}B$. But by the definition of $L_2$ $s$ acts on the root system of the pair $(L_2,H)$ without fixed points. Since $B_3=B\bigcap L_2$ and $s$ fixes all the roots of the pair $(ZH,H)$ we have an inclusion $B_3\subset HN$, where $N$ is the unipotent radical of the parabolic subgroup $P$ attached to $s$. Hence  $B_3\dot{s}B_3\subset N\dot{s}HN$.

Let $H_0$ be the centralizer of $\dot{s}$ in $H$.
Let $\h_{\bf k}$ be the Lie algebra of $H$ and $\h_{\bf k}^0$ the Lie algebra of $H_0$. Let $\h'_{\bf k}$ be the $s$--invariant complementary subspace to $\h_{\bf k}^0$ in $\h_{\bf k}$. Since $\h_{\bf k}$ is abelian $\h'_{\bf k}\subset \h_{\bf k}$ is a Lie algebra. Let $H'\subset H$ be the subgroup which corresponds to $\h'_{\bf k}$ in $H$.
If $h,h'\in H'$, $h=e^x,h'=e^y$, $x,y\in \h'_{\bf k}$ then
$$
h'\dot{s}h(h')^{-1}=\dot{s}e^{(s^{-1}-1)x+y},
$$
and for any $y$ one can find a unique $x=\frac{1}{1-s^{-1}}y$ such that $h'\dot{s}h(h')^{-1}=\dot{s}$. Note also that $H$ normalizes $N$. Therefore the factorization $H=H_0H'$ implies that any element of $N\dot{s}HN$ can be conjugated by an element of $H'$ to an element from $N\dot{s}H_0N$.

Finally observe that $N\dot{s}H_0N\subset N\dot{s}ZN$, and hence any conjugacy class $\gamma \in G_\mathcal{O}$ intersects $N\dot{s}H_0N\subset N\dot{s}ZN$. By Proposition \ref{crosssect} $\gamma$ also intersects $\Sigma_s$. The proof of Proposition \ref{crosssect} in \cite{S6} implies that the $Z$--component of any element from $N\dot{s}ZN$ is equal to the $Z$--component in $\Sigma_s=\dot{s}ZN_s$ of its image under the isomorphism $N\dot{s}ZN\simeq N\times \dot{s}ZN_s$. Therefore  any conjugacy class $\gamma \in G_\mathcal{O}$ intersects  $\Sigma_s$ at some point of $\dot{s}H_0N_s$.
This completes the proof.

\end{proof}


\setcounter{equation}{0}
\setcounter{theorem}{0}

\section{Some axillary results}\label{aux}

In this section we obtain two axillary numeric results which are required for the study of representations of quantum groups at roots of unity.

\begin{proposition}\label{cQ}
Let $\Delta$ be an irreducible root system, $\Delta_+$ the system of positive roots associated to a conjugacy class of a Weyl group element $s\in W$ in Theorem \ref{mainth}, $s=s_{\gamma_1}\ldots s_{\gamma_{l'}}$ representation (\ref{sinv}) for $s$.
Let $q$ be a positive integer defined by the tables in Appendix 1 for some irreducible exceptional root systems. Assume that $q$ is not defined for the other exceptional root systems and for irreducible root systems of classical types. Let $m$ be an odd positive integer which is not divisible by $q$ if $q$ is defined. Then for any integers $p_j\in \mathbb{Z}$, $j=1,\ldots ,l$ the system of equations $Y_j(\sum_{i=1}^{l'}m_i\gamma_i)= mp_j$, $j=1,\ldots ,l$ for unknowns $m_i\in \{-m+1,\ldots ,m-1\}$, $i=1,\ldots ,l'$ has no nontrivial solutions; here $Y_j$ are defined by $Y_j(\alpha_i)=d_i\delta_{ij}$, and $\alpha_i$ are the simple roots of $\Delta_+$.
\end{proposition}

\begin{proof}
If $\Delta$ is of exceptional type the statement of the theorem is verified with the help of a computer program (see Appendix 1).

If $\Delta$ is of classical type, we rewrite the system of equations in question in the form
\begin{equation}\label{gsys}
\sum_{i=1}^{l'}~x_i\gamma_i = \sum_{j=1}^{l}\frac{p_j}{d_j}\alpha_j,~x_i=\frac{m_i}{m}.
\end{equation}
We have to show that for any $p_j\in \mathbb{Z}$ this system has no nontrivial rational solutions $-1< x_i<1$, $i=1,\ldots , l'$ such that $m$ is divisible by the lowest common multiple of the denominators of $x_i$, $i=1,\ldots , l'$.

Let $\Delta_+'$ be a system of positive roots in $\Delta$, $\alpha_1',\ldots\alpha_l'$ the system of simple roots in $\Delta'_+$. Fix $\Delta_+'$ in such a way that $s$ is elliptic in a parabolic Weyl subgroup $W'\subset W$ generated by the simple reflections corresponding to roots from a subset of $\alpha_1',\ldots\alpha_l'$ (see the proof of Theorem \ref{mainth}). We claim that the set $\gamma_1,\ldots ,\gamma_{l'}$ can be completed to a basis $\Upsilon$ of $\h^*$ by elements from the set $\alpha_1',\ldots\alpha_l'$ so that every element of the set $\alpha_1',\ldots\alpha_l'$ can be expressed as a linear combination of elements from $\Upsilon$ with half integer coefficients.

Assuming that it is true and $\Upsilon=\{\gamma_1,\ldots ,\gamma_{l'},\alpha_{i_{l'+1}}',\ldots\alpha_{i_l}'\}$ every root $\alpha_j$ can be expressed as a linear combination of elements of $\Upsilon$ with half integer coefficients, $\alpha_j=\sum_{k=1}^{l'}c_j^k\gamma_k+\sum_{k=l'+1}^{l}c_j^k\alpha_{i_k}'$, and system (\ref{gsys}) takes the form
\begin{equation}\label{gsys1}
\sum_{i=1}^{l'}~x_i\gamma_i = \sum_{j=1}^{l}\frac{p_j}{d_j}\left(\sum_{k=1}^{l'}c_j^k\gamma_k+\sum_{k=l'+1}^{l}c_j^k\alpha_{i_k}'\right).
\end{equation}

Equations (\ref{gsys1}) are equivalent to
\begin{equation}\label{sol1}
x_i=\sum_{j=1}^{l}\frac{p_j}{d_j}c_j^i,~i=1,\ldots,l',
\end{equation}
and we also have compatibility conditions $\sum_{j=1}^{l}c_j^i\frac{p_j}{d_j}=0$, $i=l'+1,\ldots, l$. Since the numbers $d_j$ are equal to $1$ or $2$ and the coefficients $c_j^i$ are half integer, nontrivial solutions (\ref{sol1}), $-1< x_i<1$ may only have denominators equal to $2$ or $4$, and $m$ is not divisible by $2$ and $4$.

Thus to complete the proof we have to construct the basis $\Upsilon$.
First observe that the conjugacy class of every element $s$ corresponds to the sum of a number of blocks of type $A_{n}$, $C_n$, $B_n$ or $D_{i+j}(a_{j-1})$ as in (\ref{san}), (\ref{sbn}), (\ref{scn}) or (\ref{sdn}) (see \cite{C}, Section 7). The root system $\Delta_0$ fixed by the cycles of $s$ corresponding to a block is of type $A_{n}$, $C_n$, $B_n$ or $D_{n}$, and the remaining blocks are contained in it. Therefore it suffices to find a basis $\Upsilon$ containing the set $\gamma_1,\ldots ,\gamma_{l'}$ and simple roots of $\Delta_0$ in case when the conjugacy class of $s$ corresponds to a diagram with a single nontrivial block different from blocks of type $A_0$ (which correspond to trivial cycles of $s$).

If the nontrivial block is of type $A_{l'}$ then one can choose $\Delta_+'$ in such a way that $\{\gamma_1,\ldots ,\gamma_{l'}\}=\{\alpha_1',\ldots\alpha_{l'}'\}$, so $\Upsilon=\{\alpha_1',\ldots\alpha_l'\}$.

If the nontrivial block is of type $C_{l'}$ then $\Delta$ is of type $C_l$, and one can choose $\Delta_+'$ in such a way that $\Upsilon=\{\alpha_1',\ldots,\alpha_l'\}$, $\{\gamma_1,\ldots ,\gamma_{l'}\}=\{\alpha_{l-l'+1}',\ldots,\alpha_{l}'\}$, the basis of simple roots of $\Delta_0$ is $\alpha_{1}',\ldots,\alpha_{l-l'-1}'$.

If the nontrivial block is of type $B_{l'}$ then $\Delta$ is of type $B_l$. One can choose $\Delta_+'$ in such a way that $\Upsilon=\{\alpha_1',\ldots,\alpha_l'\}$, $\{\gamma_1,\ldots ,\gamma_{l'}\}=\{\alpha_{l-l'+1}',\ldots,\alpha_{l}'\}$, the basis of simple roots of $\Delta_0$ is $\alpha_{1}',\ldots,\alpha_{l-l'-1}'$.

If the nontrivial block is of type $D_{i+j}(a_{j-1})$ and $\Delta$ is of type $B_l$
then one can choose $\Delta_+'$ in such a way that $\{\gamma_1,\ldots ,\gamma_{i+j-1}\}=\{\alpha_1',\ldots\alpha_{i+j-1}'\}$, $\gamma_{i+j}=\gamma_i+2\Sigma_{k=i+1}^{l}\alpha_k'$, the basis of simple roots of $\Delta_0$ is $\alpha_{i+j+1}',\ldots\alpha_{l}'$. So $\Upsilon=\{\gamma_1,\ldots\gamma_{i+j}, \alpha_{i+j+1}',\ldots\alpha_{l}'\}$, and $\alpha_{i+j}'=\frac{\gamma_{i+j}}{2}-\frac{\gamma_{i}}{2}-\Sigma_{k=i+1}^{i+j-1}\gamma_k-\Sigma_{k=i+j+1}^{l}\alpha_k'$.

If the nontrivial block is of type $D_{i+j}(a_{j-1})$ and $\Delta$ is of type $D_l$
then one can choose $\Delta_+'$ in such a way that $\{\gamma_1,\ldots ,\gamma_{i+j-1}\}=\{\alpha_1',\ldots\alpha_{i+j-1}'\}$, $\gamma_{i+j}=\gamma_i+2\Sigma_{k=i+1}^{l-2}\alpha_k'+\alpha_{l-1}'+\alpha_{l}'$ $(i<l-1)$ or $\gamma_{i+j}=\alpha_{l}'$ $(i=l-1,j=1)$, the basis of simple roots of $\Delta_0$ is $\alpha_{i+j+1}',\ldots\alpha_{l}'$. So $\Upsilon=\{\gamma_1,\ldots\gamma_{i+j}, \alpha_{i+j+1}',\ldots\alpha_{l}'\}$, and $\alpha_{i+j}'=\frac{\gamma_{i+j}}{2}-\frac{\gamma_{i}}{2}-\Sigma_{k=i+1}^{i+j-1}\gamma_k-\Sigma_{k=i+j+1}^{l-2}\alpha_k'-
\frac{\alpha_{l-1}'}{2}-\frac{\alpha_{l}'}{2}$ $(i<l-1)$.

Thus in all cases every element of the set $\alpha_1',\ldots\alpha_l'$ can be expressed as a linear combination of elements from $\Upsilon$ with half integer coefficients. This completes the proof.

\end{proof}

\begin{proposition}
Let $\Delta$ be an irreducible root system, $\Delta_+$ the system of positive roots associated to a conjugacy class of a Weyl group element $s\in W$ in Theorem \ref{mainth}, $s=s_{\gamma_1}\ldots s_{\gamma_{l'}}$ representation (\ref{sinv}) for $s$, $\alpha_1,\ldots,\alpha_l$ the system of simple roots in $\Delta_+$.
Then if $\Delta$ is of exceptional type the lowest common multiple $d$ of the denominators of the numbers $\frac{1}{d_j}\left( {1+s \over 1-s }P_{{\h'}^*}\alpha_i,\alpha_j\right)$, where $i,j=1,\ldots,l$ and $i<j$ (or $i>j$) is given in the tables in Appendix 1; in the non simply--laced cases two numbers are given: the first one is for $i<j$ and the second one is for $i>j$.

If $\Delta$ is of classical type then the conjugacy class of $s$ corresponds to the sum of a number of blocks as in (\ref{san}), (\ref{scn}), (\ref{sbn}) or (\ref{sdn}). To each block of type $X$ we associate an integer $d_{ij}(X)$, $i,j=1,\ldots , l$ as follows:

if $\Delta$ is not of type $A_l,D_l$, an orbit with the smallest number of elements for the action of the group $<s>$ on $E$ corresponds to a block of type $A_n$ and $s$ does not fix any root from $\Delta$ then

for $\Delta=B_l$
\begin{equation}\label{abn}
d_{ij}(A_n)=\left\{
         \begin{array}{ll}
           2p+1 & \hbox{if $n=2p$ is even;} \\
           p+1 & \hbox{if $n=2p+1$, $n\neq 4p-1$ is odd;} \\
           p  &  \hbox{if $n=4p-1$ is odd and $i<j$;} \\
           2p  &  \hbox{if $n=4p-1$ is odd and $i>j$;}
         \end{array}
       \right.
\end{equation}

for $\Delta=C_l$
\begin{equation}\label{acn}
d_{ij}(A_n)=\left\{
         \begin{array}{ll}
           2p+1 & \hbox{if $n=2p$ is even;} \\
           p+1 & \hbox{if $n=2p+1$, $n\neq 4p-1$ is odd;} \\
           2p  &  \hbox{if $n=4p-1$ is odd and $i<j$;} \\
           p  &  \hbox{if $n=4p-1$ is odd and $i>j$;}
         \end{array}
       \right.
\end{equation}

for $\Delta=D_l$ if $A_{l-1}\subset D_l$ is the only nontrivial block of the conjugacy class of $s$ then
\begin{equation}\label{adn}
d_{ij}(A_{l-1})=\left\{
         \begin{array}{ll}
           2p+1 & \hbox{if $l=2p+1$ is odd;} \\
           p+1 & \hbox{if $l=2p+2$, $l\neq 4p$ is even;} \\
           p  &  \hbox{if $l=4p$ is even;}
         \end{array}
       \right.
\end{equation}

for $\Delta=A_l$ if $s$ is a representative in the Coxeter conjugacy class, i.e. the conjugacy class of $s$ corresponds to the block of type $A_l$, then
\begin{equation}\label{acox}
d_{ij}(A_l)=1;
\end{equation}

in all other cases

\begin{equation}\label{asimplylaced}
d_{ij}(A_k)=\left\{
         \begin{array}{ll}
           k+1 & \hbox{if $k$ is even;} \\
           \frac{k-1}{2}+1 & \hbox{if $k$ is odd;}
         \end{array}
       \right.
\end{equation}

in all cases
$$
d_{ij}(C_n)=d_{ij}(B_n)=d_{ij}(D_{v+w}(a_{w-1}))=1.
$$

Then a common multiple $d$ of the denominators of the numbers $\frac{1}{d_j}\left( {1+s \over 1-s }P_{{\h'}^*}\alpha_i,\alpha_j\right)$, where $i,j=1,\ldots,l$ and $i<j$ ($i>j$) is the lowest common multiple of the numbers $d_{ij}(X)$ for all blocks $X$ of the conjugacy class of $s$ and $i<j$ ($i>j$, respectively).

If $\alpha_1',\ldots \alpha_l'$ is another system of simple roots then a common multiple of the denominators of the numbers $\frac{1}{d_j}\left( {1+s \over 1-s }P_{{\h'}^*}\alpha_i,\alpha_j\right)$ will be also a common multiple of the denominators of the numbers $\frac{1}{d_j}\left( {1+s \over 1-s }P_{{\h'}^*}\alpha_i',\alpha_j'\right)$ and vice versa.

\end{proposition}

\begin{proof}

First observe that
if $\Delta'_+$ is another system of positive roots with the simple roots $\alpha_1',\ldots ,\alpha_l'$ then $\alpha_i=\sum_{k=1}^lc_i^k\alpha_k'$, $\alpha_j^\vee=\sum_{k=1}^lb_i^k\alpha_k'^\vee$, where $c_i^k,b_i^k$ are integer coefficients. Hence
$$\frac{1}{d_j}\left( {1+s \over 1-s }P_{{\h'}^*}\alpha_i,\alpha_j\right)=\left( {1+s \over 1-s }P_{{\h'}^*}\alpha_i,\alpha_j^\vee\right)=\sum_{k,p=1}^lc_i^kb_j^p\left( {1+s \over 1-s }P_{{\h'}^*}\alpha_k',\alpha_p'^\vee\right)=\sum_{k,p=1}^lc_i^kb_j^p\frac{1}{d_p}\left( {1+s \over 1-s }P_{{\h'}^*}\alpha_k',\alpha_p'\right),$$
and a common multiple of the denominators of the numbers $\frac{1}{d_j}\left( {1+s \over 1-s }P_{{\h'}^*}\alpha_i',\alpha_j'\right)$ will be also a common multiple of the denominators of the numbers $\frac{1}{d_j}\left( {1+s \over 1-s }P_{{\h'}^*}\alpha_i,\alpha_j\right)$ and vice versa.

In case of classical irreducible root systems we shall compute a common multiple $d$ of the denominators of the numbers $\frac{1}{d_j}\left( {1+s \over 1-s }P_{{\h'}^*}\alpha_i',\alpha_j'\right)$, where $\Delta'_+$ is chosen in such a way that $s$ is elliptic in a parabolic Weyl subgroup $W'\subset W$ generated by the simple reflections corresponding to roots from a subset of $\alpha_1',\ldots\alpha_l'$ (see the proof of Theorem \ref{mainth}).

As in the proof of the previous proposition it suffices to consider the case when the conjugacy class of $s$ corresponds to a diagram with a single nontrivial block.
We shall compute $d$ in case when this block is of type $A_{k}$, $k>1$. Other cases can be considered in a similar way.
Assume that the root system $\Delta$ is realized as in Section \ref{stt}, where $V$ is a real Euclidean $n$--dimensional vector space equipped with the standard scalar product, with an orthonormal basis $\varepsilon_1,\ldots ,\varepsilon_n$. In that case simple roots are

\subsection*{$\bf A_n$} $\alpha'_i=\varepsilon_i-\varepsilon_{i+1}$, $1\leq i\leq n$;

\subsection*{$\bf B_n$} $\alpha'_i=\varepsilon_i-\varepsilon_{i+1}$, $1\leq i< n$, $\alpha'_n=\varepsilon_n$;

\subsection*{$\bf C_n$} $\alpha'_i=\varepsilon_i-\varepsilon_{i+1}$, $1\leq i< n$, $\alpha'_n=2\varepsilon_n$;

\subsection*{$\bf D_n$} $\alpha'_i=\varepsilon_i-\varepsilon_{i+1}$, $1\leq i< n$, $\alpha'_n=\varepsilon_{n-1}+\varepsilon_n$;

Then $s$ is of the form
$$
s=s^1s^2,~s^1=s_{\alpha'_{p+1}}s_{\alpha'_{p+3}}\ldots,
~s^2=s_{\alpha'_{p+2}}s_{\alpha'_{p+4}}\ldots,
$$
where in the formulas for $s^{1,2}$ the products are taken over mutually orthogonal simple roots labeled by indexes of the same parity; the last simple root which appears in those products is $\alpha'_{p+k}=\varepsilon_{p+k}-\varepsilon_{p+k+1}$, so $\gamma_1,\ldots,\gamma_k=\alpha'_{p+1},\alpha'_{p+3},\ldots,\alpha'_{p+2},\alpha'_{p+4},\ldots$.

We have to compute the numbers $\left( {1+s \over 1-s }P_{{\h'}^*}\alpha_i',\alpha_j'^\vee\right)$. We consider the case when $i<j$. The case when $i>j$ can be obtained from it by observing that
\begin{equation}\label{ijji}
\left( {1+s \over 1-s }P_{{\h'}^*}\alpha_i',\alpha_j'^\vee\right)=-\left( {1+s \over 1-s }P_{{\h'}^*}\alpha_j',\alpha_i'^\vee\right)\frac{(\alpha_i',\alpha_i')}{(\alpha_j',\alpha_j')}.
\end{equation}

First observe that
by Lemma 6.2 in \cite{S9}
\begin{equation}\label{matrel}
\left( {1+s \over 1-s }P_{{\h'}^*}\gamma_i , \gamma_j \right)=
\varepsilon_{ij}(\gamma_i,\gamma_j),
\end{equation}
where
$$
\varepsilon_{ij} =\left\{ \begin{array}{ll}
-1 & i <j \\
0 & i=j \\
1 & i >j
\end{array}
\right. .
$$

Let $\omega_t'$ be the fundamental weights of the root subsystem $A_{k}\subset \Delta$ with respect to the basis of simple roots $\alpha'_i$, $i=p+1,\ldots,p+k$,
$$
\omega_t'=\varepsilon_{p+1}+\ldots+\varepsilon_{p+t}-\frac{t}{k+1}\sum_{j=1}^{k+1}\varepsilon_{p+j}.
$$
Then
$$
\left( {1+s \over 1-s }P_{{\h'}^*}\alpha_i',\alpha_j'^\vee\right)=\sum_{t,u=1}^{k}(\omega_t'^\vee,\alpha_i')\left( {1+s \over 1-s }P_{{\h'}^*}\alpha_{p+t}',\alpha_{p+u}'^\vee\right)(\omega_u',\alpha_j'^\vee).
$$
Since the scalar product in $V$ is normalized in such a way that $\alpha_{p+u}'^\vee=\alpha_{p+u}'$, $u=1,\ldots,k$ we obtain using (\ref{matrel})
\begin{eqnarray}\label{saa}
\left( {1+s \over 1-s }P_{{\h'}^*}\alpha_i',\alpha_j'^\vee\right)=\sum_{t,u=1}^{k}(\omega_t'^\vee,\alpha_i')\left( {1+s \over 1-s }P_{{\h'}^*}\alpha_{p+t}',\alpha_{p+u}'\right)(\omega_u',\alpha_j'^\vee)= \\
\qquad \qquad \qquad \qquad \qquad \qquad =\sum_{t=1}^{k}(-1)^t(\omega_t'^\vee,\alpha_i')
(\omega_{t-1}'+\omega_{t+1}',\alpha_j'^\vee), \nonumber
\end{eqnarray}
where we assume that $\omega_0'=\omega_{k+1}'=0$.

Now one has to consider several cases.

If one of the roots $\alpha_i',\alpha_j'$ is orthogonal to ${\h'}^*$ then the left hand side of the last equality is zero.

If $\alpha_i',\alpha_j'\in \{\gamma_1,\ldots,\gamma_k\}$ then by (\ref{matrel}) the left hand side of (\ref{saa}) is equal to $\pm 1$.

If $\alpha_i'=\alpha_{p+t}'$, $1<t<k$, $\alpha_j'=\alpha_{p+k+1}'$ then
\begin{equation}\label{dl}
\left( {1+s \over 1-s }P_{{\h'}^*}\alpha_i',\alpha_j'^\vee\right)=(-1)^t(\omega_{t-1}'+\omega_{t+1}',\alpha_{p+k+1}'^\vee)=
(-1)^t(\vartheta-\delta\frac{2t}{k+1}),
\end{equation}
where $\delta=2$ if $\alpha_j'^\vee=2\varepsilon_{p+k+1}$ or $\alpha_j'^\vee=\varepsilon_{p+k}+\varepsilon_{p+k+1}$, $\vartheta=0$ in the former case, and $\vartheta=1$ in the latter case. In all other cases $\vartheta=0$ and $\delta=1$. Note that $\delta\neq 1$ only in case when $\Delta$ is of type $B_n$ or $D_n$; for arbitrary $s$ this situation can only be realized if an orbit with the smallest number of elements for the action of the group $<s>$ on $E$ corresponds to a block of type $A_k$ and $s$ does not fix any root from $\Delta$. The denominator $d$ of the number in the right hand side of (\ref{dl}) is given by
\begin{equation}\label{d1}
d=\left\{
         \begin{array}{ll}
           2p+1 & \hbox{if $k=2p$ is even;} \\
           p+1 & \hbox{if $k=2p+1$, $k\neq 4p-1$ is odd;} \\
           \frac{2p}{\delta}  &  \hbox{if $k=4p-1$ is odd.}
         \end{array}
       \right.
\end{equation}

If $\alpha_i'=\alpha_{p+1}'$, $\alpha_j'=\alpha_{p+k+1}'$ then
$$
\left( {1+s \over 1-s }P_{{\h'}^*}\alpha_i',\alpha_j'^\vee\right)=-(\omega_{2}',\alpha_{p+k+1}'^\vee)=
\delta\frac{2}{k+1}-\vartheta,
$$
where $\delta=2$ if $\alpha_j'^\vee=2\varepsilon_{p+k+1}$ or $\alpha_j'^\vee=\varepsilon_{p+k}+\varepsilon_{p+k+1}$, $\vartheta=0$ in the former case, and $\vartheta=1$ in the latter case if $k=2$. In all other cases $\vartheta=0$ and $\delta=1$. We again obtain (\ref{d1}).

If $\alpha_i'=\alpha_{p+k}'$, $\alpha_j'=\alpha_{p+k+1}'$ then
$$
\left( {1+s \over 1-s }P_{{\h'}^*}\alpha_i',\alpha_j'^\vee\right)=(-1)^{k}(\omega_{k-1}',\alpha_{p+k+1}'^\vee)=
-(-1)^k\delta\frac{k-1}{k+1},
$$
and we obtain (\ref{d1}).

If $\alpha_i'=\alpha_{p}'$, $\alpha_j'=\alpha_{p+k+1}'$ then
\begin{eqnarray*}
\left( {1+s \over 1-s }P_{{\h'}^*}\alpha_i',\alpha_j'^\vee\right)=\sum_{t=1}^{k}(-1)^t(\omega_t',\alpha_p'^\vee)
(\omega_{t-1}'+\omega_{t+1}',\alpha_{p+k+1}'^\vee)= \qquad \qquad \qquad \qquad \qquad \qquad \qquad\\ =-\sum_{t=1}^{k-1}(-1)^t\left(-1+\frac{t}{k+1}\right)
\frac{2t\delta}{k+1}-
(-1)^k\left(-1+\frac{k}{k+1}\right)\frac{k-1}{k+1}\delta+(-1)^{k-1}\vartheta\left(-1+\frac{k-1}{k+1}\right).
\end{eqnarray*}

Using the fact that
$$
\sum_{r=1}^n(-1)^{r+1}r^2=(-1)^{n+1}\frac{n(n+1)}{2}~{\rm and}~\sum_{r=1}^n(-1)^{r+1}r=
\left\{
\begin{array}{ll}
\frac{n+1}{2} & \hbox{ if $n$ is odd;} \\
-\frac{n}{2} & \hbox{if $n$ is even}
\end{array}
\right.
$$
we obtain
\begin{eqnarray*}
\left( {1+s \over 1-s }P_{{\h'}^*}\alpha_i',\alpha_j'^\vee\right)=\left\{
\begin{array}{ll}
-\frac{\delta}{k+1}+\vartheta\frac{2}{k+1} & \hbox{if $k$ is even;} \\
-\vartheta\frac{2}{k+1}  & \hbox{if $k$ is odd.}
\end{array}
\right.
\end{eqnarray*}

The denominator $d$ of the number in the right hand side of the last equality is given by
$$
d=\left\{
         \begin{array}{ll}
           2p+1 & \hbox{if $k=2p$ is even;} \\
           1 & \hbox{if $k=2p+1$, $n$ is odd and $\vartheta=0$;} \\
           p+1  &  \hbox{if $k=2p+1$ is odd and $\vartheta=1$.}
         \end{array}
       \right.
$$

Summarizing all cases considered above and adding the case $i>j$ (see (\ref{ijji})) we arrive at (\ref{abn}), (\ref{acn}), (\ref{adn}), (\ref{acox}) and (\ref{asimplylaced}).

Other cases can be treated in a similar way.

\end{proof}


\setcounter{equation}{0}
\setcounter{theorem}{0}

\section*{Appendix 1. Transversal slices for simple exceptional algebraic groups.}

In this appendix, for simple exceptional algebraic groups we present the data related to the transversal slices $\Sigma_s$ defined in Theorem \ref{mainth}.
Let $G$ be a connected simple algebraic group of an exceptional type over an algebraically closed filed of characteristic good for $G$, and $\mathcal{O}\in \widehat{\underline{\mathcal{N}}}(G)$. Let $H$ be a maximal torus of $G$, $W$ the Weyl group of the pair $(G,H)$, and $s\in W$ an element from the conjugacy class $\Psi^W(\mathcal{O})$. Let $\Delta$ be the root system of the pair $(G,H)$ and $\Delta_+$ the system of positive roots in $\Delta$ associated to $\Psi^W(\mathcal{O})$ and defined in Section \ref{trans} with the help of decomposition (\ref{hdec}), where the subspaces $\h_i$ are ordered in such a way that if $\h_i=\h_\lambda^k$, $\h_j=\h_\mu^l$ and $0\leq \lambda <\mu< 1$ then $i<j$. Let $\Sigma_s$ be the corresponding transversal slice defined in Proposition \ref{crosssect}. Then straightforward calculation shows that
$$
{\rm dim}~Z_{G_p}(n)={\rm dim}~\Sigma_s
$$
for any $n\in \mathcal{O}\in \underline{\mathcal{N}}(G_p)\subset \widehat{\underline{\mathcal{N}}}(G)$. The numbers ${\rm dim}~Z_{G_p}(n)$ can be found in \cite{Li}, Chapter 22 (note, however, that the notation in \cite{Li} for some classes is different from ours; we follow \cite{L3,Spal1}). The numbers ${\rm dim}~\Sigma_s$ are contained in the tables below. These two numbers coincide in all cases. The tables below contain also the following information for each $\mathcal{O}\in \widehat{\underline{\mathcal{N}}}(G)$:
\\
-- The Weyl group conjugacy class $\Psi^W(\mathcal{O})$ which can be found in \cite{L3};
\\
-- The two involutions $s^1$ and $s^2$ in the decomposition $s=s^1s^2\in \Psi^W(\mathcal{O})$; they are represented by sets of natural numbers which are the numbers of roots appearing in decompositions $s^1=s_{\gamma_1}\ldots s_{\gamma_n}$, $s^2=s_{\gamma_{n+1}}\ldots s_{\gamma_{l'}}$, where the system of positive roots $\Delta_+$ is chosen as in Theorem \ref{mainth}, and the numeration of positive roots is given in Appendix 2;
\\
-- The dimension of the fixed point space $\h_0$ for the action of $s$ on $\h$;
\\
-- The number $|\overline{\Delta}_0|$ of roots fixed by $s$;
\\
-- The type of the root system $\overline{\Delta}_0$ fixed by $s$;
\\
-- The Dynkin diagram $\Gamma_0$ of $\overline{\Delta}_0$, where the numbers at the vertices of $\Gamma_0$ are the numbers of simple roots in $\Delta_+$ which appear in $\Gamma_0$; the numeration of simple roots is given in Appendix 2;
\\
-- The length $\underline{l}(s)$ of $s$ with respect to the system of simple roots in $\Delta_+$;
\\
-- ${\rm dim}~\Sigma_s={\rm dim}~\h_0+|\overline{\Delta}_0|+\underline{l}(s)$;
\\
-- The lowest common multiple $d$ of the denominators of the numbers $\frac{1}{d_j}\left( {1+s \over 1-s }P_{{\h'}^*}\alpha_i,\alpha_j\right)$, where $i,j=1,\ldots,l$ and $i<j$ (or $i>j$); in the non simply--laced cases two numbers are given: the first one is for $i<j$ and the second one is for $i>j$;
\\
-- The number $q$; if an odd positive integer $m$ is not divisible by $q$ then for any integers $p_j\in \mathbb{Z}$, $j=1,\ldots ,l$ the system of equations $Y_j(\sum_{i=1}^{l'}m_i\gamma_i)= mp_j$, $j=1,\ldots ,l$ for unknowns $m_i\in \{-m+1,\ldots ,m-1\}$, $i=1,\ldots ,l'$ has no nontrivial solutions; here $Y_j$ are defined by $Y_j(\alpha_i)=d_i\delta_{ij}$.

The algorithm for computing all the data above is as follows. MAGMA software was used to realize the algorithm (see \cite{magma}).

\begin{enumerate}

\item
The input data are two sets of mutually orthogonal roots $\gamma_1,\ldots,\gamma_n$ and $\gamma_{n+1},\ldots,\gamma_{l'}$ which appear in decomposition (\ref{sinv}). These sets can be found in \cite{C} for a representative $s$ in each Weyl group conjugacy class. MAGMA associates to each positive root its position in the list of positive roots (see Appendix 2), and the roots $\gamma_1,\ldots,\gamma_n$ and $\gamma_{n+1},\ldots,\gamma_{l'}$ are given in terms of their positions with respect to a system of positive roots $\Delta_+^0$.

\item
The eigenvalues of $s$ are computed in the cyclotomic field $\mathbb{Q}(\epsilon)$, where $\epsilon^h=1$, and $h$ is the order of $s$, $s^h=1$.

\item
The matrix $M$ has coefficients in the field $\mathbb{Q}$, $\mathbb{Q}(\sqrt{2})$ or $\mathbb{Q}(\sqrt{3})$.
The eigenvalues of the corresponding matrix $I-M$ are computed in the splitting field $\mathbb{F}$ of its characteristic polynomial. Note that MAGMA does computations in cyclotomic fields and in algebraic number fields with infinite precision.

\item The different eigenvalues $0<\lambda_i<1$ of $I-M$ are ordered in the following way
$$
1>\lambda_K >\lambda_{K-1}>\ldots>\lambda_1>0,
$$
and the corresponding eigenvalues of $s$ are ordered in the same way.

\item
A cycle over $\lambda_j$ is run starting from $\lambda_K$. For $\lambda_j$ an orthonormal basis of the corresponding eigenspace of $I-M$ is constructed in the space $\mathbb{F}^{l'}$.

\item
A cycle is run over all elements of the basis constructed at the previous step. For each element $u$ of that basis the corresponding elements $a_u$ and $b_u$ are defined in the $\mathbb{F}$--form $\h_\mathbb{F}^*$ of the Cartan subalgebra. $\h_\mathbb{F}^*$ is spanned by simple roots over $\mathbb{F}$.

\item
Assuming that the set $\bigcup_{k\leq i}\overline{\Delta}_k$  had been defined at the previous step,  $\bigcup_{k\leq i-1}\overline{\Delta}_{k}$ is defined as the subset of $\bigcup_{k\leq i}\overline{\Delta}_k$ which consists of roots orthogonal in $\h_\mathbb{F}^*$ to $a_u$ and $b_u$ spanning $\h_i$, and $\overline{\Delta}_i$ is defined as the subset of $\bigcup_{k\leq i}\overline{\Delta}_k$ which consists of roots not orthogonal to $a_u$ and $b_u$ in $\h_\mathbb{F}^*$.

\item
If necessarily the signs of some of the roots in the set $\{\gamma_1,\ldots,\gamma_{l'}\}\bigcap\overline{\Delta}_i$ are changed so that the orthogonal projections of the roots from the sets $\{\gamma_1,\ldots,\gamma_{n}\}\bigcap\overline{\Delta}_i$ and $\{\gamma_{n+1},\ldots,\gamma_{l'}\}\bigcap\overline{\Delta}_i$ onto the plane $\h_{i}$ spanned by $a_u$ and $b_u$ have the same directions. This does not change $s$.

\item
If the angle between the rays given by the orthogonal projections of the roots from the sets $\{\gamma_1,\ldots,\gamma_{n}\}\bigcap\overline{\Delta}_i$ and $\{\gamma_{n+1},\ldots,\gamma_{l'}\}\bigcap\overline{\Delta}_i$ onto the plane $\h_{i}$ is less than $\frac{\pi}{2}$ then the signs of roots from the set $\{\gamma_{n+1},\ldots,\gamma_{l'}\}\bigcap\overline{\Delta}_i$ are changed. It does not change $s$. This ensures that the angle between the rays $v_1$ and $v_2$ given by the orthogonal projections of the roots from the sets $\{\gamma_1,\ldots,\gamma_{n}\}\bigcap\overline{\Delta}_i$ and $\{\gamma_{n+1},\ldots,\gamma_{l'}\}\bigcap\overline{\Delta}_i$ onto the plane $\h_{i}$ is greater than or equal to $\frac{\pi}{2}$ (see Fig. 1).

$$
\xy/r10pc/: ="A",-(1,0)="B", "A",+(1,0)="C","A",-(0,1)="D","A",+(0.1,0.57)*{x},"A", {\ar+(0,+0.5)},"A", {\ar@{-}+(0,+1)}, "A";"B"**@{-},"A";"C"**@{-},"A";"D"**@{-},"A", {\ar+(0.6,0.13)},"A",+(0.65,0.18)*{v^2},"A",+(0.87,0.18)*{\overline{\Delta}_{i}^{2}},"A", {\ar@{-}+(0.9,0.41)}, "A",+(0.32,0)="D", +(-0.038,0.134)="E","D";"E" **\crv{(1.33,0.07)},"A",+(0.45,0.15)*{\psi},"A",+(0.49,0.05)*{\psi},"A", {\ar+(-0.6,0.23)},"A",+(-0.65,0.28)*{v^1},"A",+(-0.87,0.28)*{\overline{\Delta}_{i}^{1}},"A", {\ar@{-}+(-0.85,0.72)}, "A",+(-0.32,0)="F", +(0.066,0.21)="G","F";"G" **\crv{(0.67,0.07)},"A",+(-0.45,0.27)*{\varphi},"A",+(-0.49,0.08)*{\varphi}
\endxy
$$
\begin{center}
 Fig.1
\end{center}

\item
A generic element $x$ in the real form of $\h_{i}$ is defined as a linear combination of $a_u$ and $b_u$ with coefficients transcendental to all elements of $\mathbb{F}$ (the coefficients contain numbers $e$ and $\pi$).  $x$ is rotated by powers $s^c$, $c=0,\ldots,f-1$, where $f$ is the order of $s$ in $\h_{i}$, until the angles formed by $x$, $v_1$ and $x$, $v_2$ become acute (see Fig. 1). The order $f$ is found as the order of the eigenvalue of $s$ in $\mathbb{Q}(\epsilon)$ which corresponds to $\lambda_j$.

\item
Positive roots $\alpha\in (\overline{\Delta}_{i})_+$ in $\overline{\Delta}_i$ are defined by the condition $(\alpha,x)>0$. Since $x$ is defined as a linear combination of $a_u$ and $b_u$ with coefficients transcendental to all elements of $\mathbb{F}$ the set $(\overline{\Delta}_{i})_+$ is defined with infinite precision.

\item
The cycle over the elements $u$ of the basis is terminated.

\item
The cycle over the elements $\lambda_j$ is terminated.

\item
An orthonormal basis of the eigenspace corresponding to the eigenvalue $-1$ of $s$ is constructed in the $\mathbb{G}$--form $\h_\mathbb{G}^*$ of $\h^*$, where $\mathbb{G}=\mathbb{Q}$ in the simply--laced cases and $\mathbb{G}=\mathbb{Q}(\sqrt{2})$ or $\mathbb{G}=\mathbb{Q}(\sqrt{3})$ in the non simply--laced cases. The basis is chosen in such a way that the one--dimensional subspaces spanned by the basic vectors are also invariant with respect to the involutions $s^1$ and $s^2$.

\item
A cycle over the elements $v$ of the basis is run.

\item
Assuming that the set and $\bigcup_{k\leq i}\overline{\Delta}_k$  had been defined at the previous step,  $\bigcup_{k\leq i-1}\overline{\Delta}_{k}$ is defined as the subset of $\bigcup_{k\leq i}\overline{\Delta}_k$ which consists of roots orthogonal to $v\in \h_i$ in $\h_\mathbb{G}^*$, and $\overline{\Delta}_i$ is defined as the subset of $\bigcup_{k\leq i}\overline{\Delta}_k$ which consists of roots not orthogonal to $v$ in $\h_\mathbb{G}^*$.

\item
If necessarily the signs of some of the roots in the set $\{\gamma_1,\ldots,\gamma_{l'}\}\bigcap\overline{\Delta}_i$
are changed so that the orthogonal projections of the roots from the set $\{\gamma_1,\ldots,\gamma_{l'}\}\bigcap\overline{\Delta}_i$ onto the line $\h_{i}$ spanned by $v$ have the same directions as $v$. This does not change $s$.

\item
Positive roots $\alpha\in (\overline{\Delta}_{i})_+$ in $\overline{\Delta}_i$ are defined by the condition $(\alpha,v)>0$.

\item
The cycle over the elements $v$ of the basis is terminated.

\item
After running the previous steps we obtain the set of roots $\overline{\Delta}_0$ fixed by $s$ and the set of positive roots
$$
\Delta_+=\bigcup_{i=0}^{I}(\overline{\Delta}_{i})_+,
$$
where
$(\overline{\Delta}_{0})_+=\overline{\Delta}_0\bigcap\Delta_+^0$.
The decomposition $s=s^1s^2$ becomes reduced since the roots from the sets $\Delta_{s^{1,2}}\bigcap\overline{\Delta}_{i}$ belong to the sectors labeled by $\overline{\Delta}_{i}^{1,2}$ at Fig. 1, and these sectors are disjoint.
By construction the roots $\gamma_1,\ldots ,\gamma_{l'}$ are positive with respect to $\Delta_+$.

\item
Simple roots in $\Delta_+$ are found using the property that $\alpha$ is simple iff for any $\beta\neq \alpha$, $\beta\in \Delta_+$ we have $s_\alpha(\beta)\in \Delta_+$.

\item
Denote by $\alpha_1^0,\ldots,\alpha_l^0$ the simple roots of $\Delta_+^0$ numbered as in Appendix 2 and by $\alpha_1,\ldots,\alpha_l$ the simple roots of $\Delta_+$. A permutation $\pi$ of the set $1,\ldots,l$ is found such that if $w\in W$ and $w(\Delta_+)=\Delta_+^0$ then $w(\alpha_i)=\alpha_{\pi(i)}^0$ for all $i$. This is done by comparing the Cartan matrices with respect to the bases $\alpha_1^0,\ldots,\alpha_l^0$ and $\alpha_1,\ldots,\alpha_l$.
Then the basis $\alpha_1,\ldots,\alpha_l$ is reordered in such a way that $\pi=id$.

\item
The coordinates of the roots $\gamma_1,\ldots ,\gamma_{l'}$ are computed with respect to the basis $\alpha_1,\ldots,\alpha_l$ and the root positions of the roots $\gamma_1,\ldots ,\gamma_{l'}$ are defined according to the MAGMA numeration of roots.

\item
Simple roots in $\overline{\Delta}_0$ are found as the intersection $\overline{\Delta}_0\bigcap\{\alpha_1,\ldots,\alpha_l\}$.

\item
The type of the root system $\overline{\Delta}_0$ and its Coxeter graph are obtained.

\item
The basis $\gamma_1^*,\ldots ,\gamma_{l'}^*$ of the rational form  $\h_\mathbb{Q}'^*$ of $\h'^*$, dual to $\gamma_1,\ldots ,\gamma_{l'}$, is obtained.

\item
The matrix
$$
p_{ij}=\frac{1}{d_j}\left( {1+s \over 1-s }P_{{\h'}^*}\alpha_i,\alpha_j\right)=\sum_{p,q=1}^{l'}\frac{1}{d_j}\left( {1+s \over 1-s }P_{{\h'}^*}\gamma_p , \gamma_q \right)(\alpha_i,\gamma_p^*)(\alpha_j,\gamma_q^*)=\sum_{p,q=1}^{l'}\frac{1}{d_j}\varepsilon_{pq}(\gamma_p , \gamma_q )(\alpha_i,\gamma_p^*)(\alpha_j,\gamma_q^*)
$$
is computed.

\item
The lowest common multiple of the denominators of the numbers $p_{ij}$, $i<j$ ($i>j$) is computed.

\item
The matrix of the system $Y_j(\sum_{i=1}^{l'}x_i\gamma_i)= p_j$, $j=1,\ldots ,l$ is found and the system is reduced to the standard echelon form using Gaussian reduction. All computations are done over $\mathbb{Q}$.

\item
The reduced system is of the form
\begin{equation}\label{sS}
\left(
  \begin{array}{cccc}
    1 & 0 & \ldots & 0 \\
    0 & 1 & \ldots & 0 \\
    \vdots & \vdots & \vdots & \vdots \\
    0 & 0 & \ldots & 1 \\
    0 & 0 & \ldots & 0 \\
    \vdots & \vdots & \vdots & \vdots \\
    0 & 0 & \ldots & 0 \\
  \end{array}
\right)\left(
  \begin{array}{c}
    x_1 \\
    x_2 \\
    \vdots \\
    x_{l'} \\
  \end{array}
\right)
=
\left(
  \begin{array}{c}
    P_1 \\
    P_2 \\
    \vdots \\
    P_{l'} \\
    P_{l'+1} \\
    \vdots \\
    P_l \\
  \end{array}
\right),
\end{equation}
where $P_i$ are polynomials with rational coefficients in $p_j$, $j=1,\ldots ,l$. We are interested in the rational nontrivial solutions $-1< x_i<1$ with odd denominators (compare with the proof of Proposition \ref{cQ}).

Since $Y_j(\sum_{i=1}^{l'}x_i\gamma_i)= p_j$, the possible values of $p_j$ are integers, and $-b_j<p_j< b_j$ $b_j=Y_j(\sum_{i=1}^{l'}\gamma_i)$.

Now it is verified if system (\ref{sS}) has nontrivial rational solutions $-1< x_i<1$ with odd denominators by substituting $p_j=-b_j,\ldots,b_j$, $j=1,\ldots ,l$ into the right hand side of (\ref{sS}). If yes, the lowest common multiple $q$ of the denominators of the numbers $x_i$, $i=1,\ldots,l'$ is computed.

\item
The dimension of $\h_0$ is computed, ${\rm dim}~\h_0=l-l'$.

\item
The number of roots in $\overline{\Delta}_0$ is computed.

\item
The length of $s$ with respect to the system $\alpha_1,\ldots,\alpha_l$ is computed.

\item
The dimension of the slice ${\rm dim}~\Sigma_s=\underline{l}(s)+{\rm dim}~\h_0+|\overline{\Delta}_0|$ is computed.

Note that despite of the fact that a computer program is used to obtain the results summarized in the tables below, all computations are done with infinite precision since MAGMA does computations in algebraic number fields and in cyclotomic fields with infinite precision.

\end{enumerate}

\vskip 0.4cm

$\bf G_2.$

\vskip 0.1cm

\begin{tabular*}{1.03\textwidth}{|@{\extracolsep{\fill} }c|c|c|c|c|c|c|c|c|c|c|c|}
  \hline
 \rule[-.3cm]{0cm}{1cm}{}$\mathcal{O}$ & $\Psi^W(\mathcal{O})$ & $s^1$ & $s^2$ & ${\rm dim}~\h_0$ & $|\overline{\Delta}_0|$ & $\overline{\Delta}_0$ & $\Gamma_0$ & $l(s)$ & ${\rm dim}~\Sigma_s$ & d & q  \\ \hline \hline
$A_1$ & $A_1$ & 6 & -- & 1 & 2 & $A_1$ & $\xymatrix@R=.1cm@C=.1cm{
{\scriptstyle 1} \\
{\bullet}
}$
& 5 & 8 & 1,1 & 3 \\
  \hline
$(\tilde{A}_1)_3$ \rule[-.3cm]{0cm}{1cm}{} & $\tilde{A}_1$ & 4 & -- & 1 & 2 & $A_1$ &
$\xymatrix@R=.1cm@C=.1cm{
{\scriptstyle 2} \\
{\bullet}
}$
& 5 & 8 & 1,1 & -- \\
  \hline
$\tilde{A}_1$ \rule[-.3cm]{0cm}{1cm}{} & $A_1+\tilde{A}_1$ & 6 & 1 & 0 & 0 & -- & -- & 6 & 6 & 1,1 & 3 \\
  \hline
$G_2(a_1) \rule[-.2cm]{0cm}{0.6cm}{} $ & $A_2$ & 5 & 2 & 0 & 0 & -- & -- & 4 & 4 & 3,1 & 3 \\
  \hline
$G_2$ \rule[-.2cm]{0cm}{0.6cm}{} & $G_2$ & 1 & 2 & 0 & 0 & -- & -- & 2 & 2 & 1,1 & 3 \\
  \hline
\end{tabular*}

\newpage

$\bf F_4.$

\vskip 0.1cm

\begin{tabular*}{1.03\textwidth}{|@{\extracolsep{\fill} }c|c|c|c|c|c|c|c|c|c|c|c|}
  \hline
 \rule[-.3cm]{0cm}{1cm}{} $\mathcal{O}$ & $\Psi^W(\mathcal{O})$ & $s^1$ & $s^2$ & ${\rm dim}~\h_0$ & $|{\overline{\Delta}}_0|$ & $\overline{\Delta}_0$ & $\Gamma_0$  & $l(s)$ & ${\rm dim}~\Sigma_s$ & d & q  \\ \hline \hline
 $A_1$ & $A_1$ & 24 & -- & 3 & 18 & $C_3$ &

$\xymatrix@R=.1cm@C=.1cm{
{\scriptstyle 4}&&{\scriptstyle 3}&&{\scriptstyle 2}\\
{\bullet}\ar@{-}[rr]&&{\bullet}\ar@2{-}[rr]&&{\bullet}
}$
 & 15 & 36 & 1,1 & -- \\
  \hline
$(\tilde{A}_1)_2$ \rule[-.3cm]{0cm}{1cm}{} & $\tilde{A}_1$ & 21 & -- & 3 & 18 & $B_3$ &

$\xymatrix@R=.1cm@C=.1cm{
{\scriptstyle 1}&&{\scriptstyle 2}&&{\scriptstyle 3}\\
{\bullet}\ar@{-}[rr]&&{\bullet}\ar@2{-}[rr]&&{\bullet}
}$
& 15 & 36 & 1,1 & -- \\
\hline
 $\tilde{A}_1$ \rule[-.3cm]{0cm}{1cm}{} & $2A_1$ & 24 & 16 & 2 & 8 & $B_2$ &
$\xymatrix@R=.1cm@C=.1cm{
{\scriptstyle 2}&&{\scriptstyle 3}\\
{\bullet}\ar@2{-}[rr]&&{\bullet}
}$
& 20 & 30 & 1,1 & -- \\
  \hline
$A_1+\tilde{A}_1$ & $4A_1$ & $\begin{array}{c}
     16 \\
     24
   \end{array}
$ & $\begin{array}{c}
      2 \\
      9
    \end{array}$
 & 0 & 0 & -- & -- & 24 & 24 & 1,1 & -- \\
  \hline
$A_2$ & $A_2$ & 23 & 1 & 2 & 6 & $A_2$ &
$\xymatrix@R=.1cm@C=.1cm{
{\scriptstyle 3}&&{\scriptstyle 4}\\
{\bullet}\ar@{-}[rr]&&{\bullet}
}$
& 14 & 22 & 3,3 & -- \\
  \hline
$\tilde{A}_2$ \rule[-.3cm]{0cm}{1cm}{} & $\tilde{A}_2$ & 19 & 4 & 2 & 6 & $A_2$ &
$\xymatrix@R=.1cm@C=.1cm{
{\scriptstyle 1}&&{\scriptstyle 2}\\
{\bullet}\ar@{-}[rr]&&{\bullet}
}$
& 14 & 22 & 3,3 & -- \\
\hline
 $(B_2)_2$ & $B_2$ & 16 & 8 & 2 & 8 & $B_2$ &
$\xymatrix@R=.1cm@C=.1cm{
{\scriptstyle 2}&&{\scriptstyle 3}\\
{\bullet}\ar@2{-}[rr]&&{\bullet}
}$
& 10 & 20 & 1,2 & -- \\
\hline
$A_2+\tilde{A}_1$ & $A_2+\tilde{A}_1$ & $\begin{array}{c}
    23  \\
     7
   \end{array}
$  & 1 & 1 & 0 & -- & -- & 17 & 18 & 3,3 & -- \\
\hline
$(\tilde{A}_2+A_1)_2$ & $\tilde{A}_2+A_1$ & $\begin{array}{c}
       19 \\
       5
     \end{array}
$ & 4 & 1 & 0 & -- & -- & 17 & 18 & 3,3 & -- \\
\hline
$\tilde{A}_2+A_1$ & $A_2+\tilde{A}_2$ & $\begin{array}{c}
    23  \\
     3
   \end{array}
$ & $\begin{array}{c}
    1  \\
     4
   \end{array}
$ & 0 & 0 & -- & -- & 16 & 16 & 3,3 & 3 \\
\hline
$B_2$ & $A_3$ & $\begin{array}{c}
    1  \\
     14
   \end{array}
$ & 16 & 1 & 2 & $A_1$ &
$\xymatrix@R=.1cm@C=.1cm{
{\scriptstyle 3}\\
{\bullet}
}$
& 13 & 16 & 1,2 & -- \\
\hline
$(C_3(a_1))_2$ & $B_2+A_1$ & 16 & $\begin{array}{c}
    8  \\
     9
   \end{array}
$ & 1 & 2 & $A_1$ &
$\xymatrix@R=.1cm@C=.1cm{
{\scriptstyle 2}\\
{\bullet}
}$
& 13 & 16 & 1,2 & -- \\
\hline
$C_3(a_1)$ & $A_3+\tilde{A}_1$ & $\begin{array}{c}
    1  \\
     14
   \end{array}
$ & $\begin{array}{c}
    3  \\
     16
   \end{array}
$ & 0 & 0 & -- & -- & 14 & 14 & 1,2 & -- \\
\hline
$F_4(a_3)$ & $D_4(a_1)$ & $\begin{array}{c}
    16  \\
     2
   \end{array}
$ & $\begin{array}{c}
    5  \\
     11
   \end{array}
$ & 0 & 0 & -- & -- & 12 & 12 & 1,1 & -- \\
\hline
$B_3$ & $D_4$ & $\begin{array}{c}
    16  \\
     9 \\
     2
   \end{array}
$ & 1 & 0 & 0 & -- & -- & 10 & 10 & 1,1 & -- \\
\hline
$C_3$ & $C_3+A_1$ & $\begin{array}{c}
    3  \\
     14
   \end{array}
$ & $\begin{array}{c}
    4  \\
     1
   \end{array}
$ & 0 & 0 & -- & -- & 10 & 10 & 1,1 & -- \\
\hline
$F_4(a_2)$ & $F_4(a_1)$ & $\begin{array}{c}
    1  \\
     3
   \end{array}
$ & $\begin{array}{c}
    9  \\
     10
   \end{array}
$ & 0 & 0 & -- & -- & 8 & 8 & 1,1 & -- \\
\hline
$F_4(a_1)$ & $B_4$ & $\begin{array}{c}
    9  \\
     2
   \end{array}
$ & $\begin{array}{c}
    1  \\
     4
   \end{array}
$ & 0 & 0 & -- & -- & 6 & 6 & 1,1 & -- \\
\hline
$F_4$ & $F_4$ & $\begin{array}{c}
    1  \\
     3
   \end{array}
$ & $\begin{array}{c}
    2  \\
     4
   \end{array}
$ & 0 & 0 & -- & -- & 4 & 4 & 1,1 & -- \\
\hline
\end{tabular*}

\newpage

$\bf E_6.$

\vskip 0.1cm

\begin{tabular*}{1.03\textwidth}{|@{\extracolsep{\fill} }c|c|c|c|c|c|c|c|c|c|c|c|}
  \hline
\rule[-.3cm]{0cm}{1cm}{}  $\mathcal{O}$ & $\Psi^W(\mathcal{O})$ & $s^1$ & $s^2$ & ${\rm dim}~\h_0$ & $|{\overline{\Delta}}_0|$ & $\overline{\Delta}_0$ & $\Gamma_0$ & $l(s)$ & ${\rm dim}~\Sigma_s$ & d & q  \\ \hline \hline
$A_1$ & $A_1$ & 36 & -- & 5 & 30 & $A_5$ &
$\xymatrix@R=.1cm@C=.1cm{
{\scriptstyle 1}&&{\scriptstyle 3}&&{\scriptstyle 4}&&{\scriptstyle 5}&&{\scriptstyle 6}\\
{\bullet}\ar@{-}[rr]&&{\bullet}\ar@{-}[rr]&&{\bullet}\ar@{-}[rr]&& {\bullet}\ar@{-}[rr] &&{\bullet}
}$
& 21 & 56 & 1 & -- \\
\hline
$2A_1$ & $2A_1$ & 36 & 23 & 4 & 12 & $A_3$ &

$\xymatrix@R=.1cm@C=.1cm{
{\scriptstyle 3}&&{\scriptstyle 4}&&{\scriptstyle 5}\\
{\bullet}\ar@{-}[rr]&&{\bullet}\ar@{-}[rr]&&{\bullet}
}$
& 30 & 46 & 1 & -- \\
\hline
$3A_1$ & $4A_1$ & $\begin{array}{c}
    36 \\ 23
   \end{array}
$ & $\begin{array}{c}
    4 \\ 15
   \end{array}
$ & 2 & 0 & -- & -- & 36 & 38 & 1 & -- \\
\hline
$A_2$ & $A_2$ & 35 & 2 & 4 & 12 & $2A_2$ &
$\xymatrix@R=.1cm@C=.1cm{
{\scriptstyle 1}&&{\scriptstyle 3}\\
{\bullet}\ar@{-}[rr]&&{\bullet}
}$

$\xymatrix@R=.1cm@C=.1cm{
{\scriptstyle 5}&&{\scriptstyle 6}\\
{\bullet}\ar@{-}[rr]&&{\bullet}
}$
& 20 & 36 & 3 & -- \\
\hline
$A_2+A_1$ & $A_2+A_1$ & $\begin{array}{c}
    35 \\ 11
   \end{array}
$ & 2 & 3 & 6 & $A_2$ &
$\xymatrix@R=.1cm@C=.1cm{
{\scriptstyle 1}&&{\scriptstyle 3}\\
{\bullet}\ar@{-}[rr]&&{\bullet}
}$
& 23 & 32 & 3 & -- \\
\hline
$2A_2$ & $2A_2$ & $\begin{array}{c}
    35 \\ 6
   \end{array}
$ & $\begin{array}{c}
    2 \\ 5
   \end{array}
$ & 2 & 6 & $A_2$ &
$\xymatrix@R=.1cm@C=.1cm{
{\scriptstyle 1}&&{\scriptstyle 3}\\
{\bullet}\ar@{-}[rr]&&{\bullet}
}$
& 22 & 30 & 3 & -- \\
\hline
$A_2+2A_1$ & $A_2+2A_1$ & $\begin{array}{c}
    35 \\ 11
   \end{array}
$ & $\begin{array}{c}
    2 \\ 7
   \end{array}
$ & 2 & 0 & -- & -- & 26 & 28 & 3 & -- \\
\hline
$A_3$ & $A_3$ & $\begin{array}{c}
    24 \\ 2
   \end{array}
$ & 23 & 3 & 4 & $2A_1$ &
$\xymatrix@R=.1cm@C=.1cm{
{\scriptstyle 3}&&{\scriptstyle 5}\\
{\bullet}&&{\bullet}
}$
& 19 & 26 & 2 & -- \\
\hline
$2A_2+A_1$ & $3A_2$ & $\begin{array}{c}
    35 \\ 1 \\ 5
   \end{array}
$ & $\begin{array}{c}
    2 \\ 3 \\ 6
   \end{array}
$ & 0 & 0 & -- & -- & 24 & 24 & 3 & 3 \\
\hline
$A_3+A_1$ & $A_3+2A_1$ & $\begin{array}{c}
    2 \\ 5 \\ 24
   \end{array}
$ & $\begin{array}{c}
    23 \\ 3
   \end{array}
$ & 1 & 0 & -- & -- & 21 & 22 & 2 & -- \\
\hline
$D_4(a_1)$ & $D_4(a_1)$ & $\begin{array}{c}
    23 \\ 4
   \end{array}
$ & $\begin{array}{c}
    8 \\ 19
   \end{array}
$ & 2 & 0 & -- & -- & 18 & 20 & 2 & -- \\
\hline
$A_4$ & $A_4$ & $\begin{array}{c}
    24 \\ 21
   \end{array}
$ & $\begin{array}{c}
    2 \\ 1
   \end{array}
$ & 2 & 2 & $A_1$ &
$\xymatrix@R=.1cm@C=.1cm{
{\scriptstyle 5}\\
{\bullet}
}$
& 14 & 18 & 5 & -- \\
\hline
$D_4$ & $D_4$ & $\begin{array}{c}
    23 \\ 4 \\ 15
   \end{array}
$ & 2 & 2 & 0 & -- & -- & 16 & 18 & 1 & --  \\
\hline
$A_4+A_1$ & $A_4+A_1$ & $\begin{array}{c}
    24 \\ 21 \\ 5
   \end{array}
$ & $\begin{array}{c}
    2 \\ 1
   \end{array}
$ & 1 & 0 & -- & -- & 15 & 16 & 5 & -- \\
\hline
$A_5$ & $A_5+A_1$ & $\begin{array}{c}
    14 \\ 13 \\ 15
   \end{array}
$ & $\begin{array}{c}
    1 \\ 6 \\ 4
   \end{array}
$ & 0 & 0 & -- & -- & 14 & 14 & 1 & -- \\
\hline
$D_5(a_1)$ & $D_5(a_1)$ & $\begin{array}{c}
    2 \\ 7
   \end{array}
$ & $\begin{array}{c}
    15 \\ 12 \\ 16
   \end{array}
$ & 1 & 0 & -- & -- & 13 & 14 & 2 & -- \\
\hline
$A_5+A_1$ & $E_6(a_2)$ & $\begin{array}{c}
    1 \\ 2 \\ 6
   \end{array}
$ & $\begin{array}{c}
    9 \\ 10 \\ 19
   \end{array}
$ & 0 & 0 & -- & -- & 12 & 12 & 1 & -- \\
\hline
$D_5$ & $D_5$ & $\begin{array}{c}
    6 \\ 1 \\ 2
   \end{array}
$ & $\begin{array}{c}
    4 \\ 15
   \end{array}
$ & 1 & 0 & -- & -- & 9 & 10 & 2 & -- \\
\hline
$E_6(a_1)$ & $E_6(a_1)$ & $\begin{array}{c}
    6 \\ 8 \\ 9
   \end{array}
$ & $\begin{array}{c}
    1 \\ 2 \\ 5
   \end{array}
$ & 0 & 0 & -- & -- & 8 & 8 & 1 & -- \\
\hline
$E_6$ & $E_6$ & $\begin{array}{c}
    1 \\ 4 \\ 6
   \end{array}
$ & $\begin{array}{c}
    2 \\ 3 \\ 5
   \end{array}
$ & 0 & 0 & -- & -- & 6 & 6 & 1 & -- \\
\hline
\end{tabular*}

\newpage

$\bf E_7.$

\vskip 0.1cm

\begin{tabular*}{1.03\textwidth}{|@{\extracolsep{\fill} }c|c|c|c|c|c|c|c|c|c|c|c|}
  \hline
\rule[-.3cm]{0cm}{1cm}{}  $\mathcal{O}$ & $\Psi^W(\mathcal{O})$ & $s^1$ & $s^2$ & ${\rm dim}~\h_0$ & $|{\overline{\Delta}}_0|$ & $\overline{\Delta}_0$ & $\Gamma_0$ & $l(s)$ & ${\rm dim}~\Sigma_s$ & d & q  \\ \hline \hline
$A_1$ & $A_1$ & 63 & -- & 6 & 60 & $D_6$ &
$\xymatrix@R=.01cm@C=.1cm{
&&&&&&&{\scriptstyle 3}\\
&&&&&&&{\bullet}\\
{\scriptstyle 7}&&{\scriptstyle 6}&&{\scriptstyle 5}&&{\scriptstyle 4}&\\
{\bullet}\ar@{-}[rr]&&{\bullet}\ar@{-}[rr]&&{\bullet}\ar@{-}[rr]&& {\bullet}\ar@{-}[uur]\ar@{-}[ddr]& \\
&&&&&&&\\
&&&&&&&{\bullet}\\
&&&&&&&{\scriptstyle 2}
}$
& 33 & 99 & 1 & -- \\
\hline
$2A_1$ & $2A_1$ & 63 & 49 & 5 & 26 & $A_1+D_4$ &
$\xymatrix@R=.01cm@C=.1cm{
&&&&&{\scriptstyle 3}\\
&&&&&{\bullet}\\
{\scriptstyle 7}&&{\scriptstyle 2}&&{\scriptstyle 4}&\\
{\bullet}&&{\bullet}\ar@{-}[rr]&& {\bullet}\ar@{-}[uur]\ar@{-}[ddr]& \\
&&&&&\\
&&&&&{\bullet}\\
&&&&&{\scriptstyle 5}
}$ & 50 & 81 & 1 & --  \\
\hline
$(3A_1)''$ & $(3A_1)'$ & 63 & $\begin{array}{c}
    7 \\ 49
   \end{array}
$ & 4 & 24 & $D_4$ & $\xymatrix@R=.01cm@C=.1cm{
&&&{\scriptstyle 3}\\
&&&{\bullet}\\
{\scriptstyle 2}&&{\scriptstyle 4}&\\
{\bullet}\ar@{-}[rr]&& {\bullet}\ar@{-}[uur]\ar@{-}[ddr]& \\
&&&\\
&&&{\bullet}\\
&&&{\scriptstyle 5}
}$ & 51 & 79 & 1 & --  \\
\hline
$(3A_1)'$ & $(4A_1)''$ & $\begin{array}{c}
    63 \\ 49
   \end{array}
$ & $\begin{array}{c}
    2 \\ 28
   \end{array}
$ & 3 & 6 & $3A_1$ & $\xymatrix@R=.1cm@C=.1cm{
{\scriptstyle 3}&&{\scriptstyle 5}&&{\scriptstyle 7}\\
{\bullet}&&{\bullet}&&{\bullet}
}$ & 60 & 69 & 1 & --  \\
\hline
$A_2$ & $A_2$ & 62 & 1 & 5 & 30 & $A_5$ &
$\xymatrix@R=.1cm@C=.1cm{
{\scriptstyle 2}&&{\scriptstyle 4}&&{\scriptstyle 5}&&{\scriptstyle 6}&&{\scriptstyle 7}\\
{\bullet}\ar@{-}[rr]&&{\bullet}\ar@{-}[rr]&&{\bullet}\ar@{-}[rr]&& {\bullet}\ar@{-}[rr] &&{\bullet}
}$
 & 32 & 67 & 3 & --  \\
\hline
$4A_1$ & $7A_1$ & $\begin{array}{c}
    41 \\ 63 \\ 40 \\ 19
   \end{array}
$ & $\begin{array}{c}
    2 \\ 3 \\ 6
   \end{array}
$ & 0 & 0 & -- & -- & 63 & 63 & 1 & --  \\
\hline
$A_2+A_1$ & $A_2+A_1$ & $\begin{array}{c}
    30 \\62
   \end{array}
$ & 1 & 4 & 12 & $A_3$ & $\xymatrix@R=.1cm@C=.1cm{
{\scriptstyle 4}&&{\scriptstyle 5}&&{\scriptstyle 6}\\
{\bullet}\ar@{-}[rr]&& {\bullet}\ar@{-}[rr] &&{\bullet}
}$ & 41 & 57 & 3 & --  \\
\hline
$A_2+2A_1$ & $A_2+2A_1$ & $\begin{array}{c}
    30 \\62
   \end{array}
$ & $\begin{array}{c}
    1 \\ 18
   \end{array}
$ & 3 & 2 & $A_1$ & $\xymatrix@R=.1cm@C=.1cm{
{\scriptstyle 5}\\
{\bullet}
}$ & 46 & 51 & 3 & --  \\
\hline
$A_2+3A_1$ & $A_2+3A_1$ & $\begin{array}{c}
    30 \\62
   \end{array}
$ & $\begin{array}{c}
    1 \\ 5 \\ 18
   \end{array}
$ & 2 & 0 & -- & -- & 47 & 49 & 3 & --  \\
\hline
$2A_2$ & $2A_2$ & $\begin{array}{c}
    23 \\62
   \end{array}
$ & $\begin{array}{c}
    1 \\ 7
   \end{array}
$ & 3 & 6 & $A_2$ & $\xymatrix@R=.1cm@C=.1cm{
{\scriptstyle 4}&&{\scriptstyle 5}\\
{\bullet}\ar@{-}[rr] &&{\bullet}
}$ & 40 & 49 & 3 & --   \\
\hline
$A_3$ & $A_3$ & $\begin{array}{c}
    1 \\ 37
   \end{array}
$ & 49 & 4 & 14 & $A_1+A_3$ & $\xymatrix@R=.1cm@C=.1cm{
{\scriptstyle 7}&&{\scriptstyle 2}&&{\scriptstyle 4}&&{\scriptstyle 5}\\
{\bullet}&&{\bullet}\ar@{-}[rr]&& {\bullet}\ar@{-}[rr] &&{\bullet}
}$ & 31 & 49 & 2 & --  \\
\hline
$(A_3+A_1)''$ & $(A_3+A_1)'$ & $\begin{array}{c}
    1 \\ 37
   \end{array}
$ & $\begin{array}{c}
    7 \\ 49
   \end{array}
$ & 3 & 12 & $A_3$ & $\xymatrix@R=.1cm@C=.1cm{
{\scriptstyle 2}&&{\scriptstyle 4}&&{\scriptstyle 5}\\
{\bullet}\ar@{-}[rr]&& {\bullet}\ar@{-}[rr] &&{\bullet}
}$ & 32 & 47 & 2 & --  \\
\hline
$2A_2+A_1$ & $3A_2$ & $\begin{array}{c}
    5 \\ 25 \\62
   \end{array}
$ & $\begin{array}{c}
    1 \\ 2 \\ 6
   \end{array}
$ & 1 & 0 & -- & -- & 42 & 43 & 3 & 3  \\
\hline
$(A_3+A_1)'$ & $(A_3+2A_1)''$ & $\begin{array}{c}
    1 \\ 16 \\ 37
   \end{array}
$ & $\begin{array}{c}
    7 \\ 49
   \end{array}
$ & 2 & 2 & $A_1$ & $\xymatrix@R=.1cm@C=.1cm{
{\scriptstyle 4}\\
{\bullet}
}$ & 37 & 41 & 2 & --  \\
\hline
$A_3+2A_1$ & $A_3+3A_1$ & $\begin{array}{c}
    14 \\ 26 \\ 28
   \end{array}
$ & $\begin{array}{c}
    3 \\ 7 \\ 49
   \end{array}
$ & 1 & 0 & -- & -- & 38 & 39 & 2 & --  \\
\hline
\end{tabular*}

\begin{tabular*}{1.03\textwidth}{|@{\extracolsep{\fill} }c|c|c|c|c|c|c|c|c|c|c|c|}
  \hline
\rule[-.3cm]{0cm}{1cm}{}  $\mathcal{O}$ & $\Psi^W(\mathcal{O})$ & $s^1$ & $s^2$ & ${\rm dim}~\h_0$ & $|{\overline{\Delta}}_0|$ & $\overline{\Delta}_0$ & $\Gamma_0$ & $l(s)$ & ${\rm dim}~\Sigma_s$ & d & q  \\ \hline \hline
$D_4(a_1)$ & $D_4(a_1)$ & $\begin{array}{c}
    3 \\ 49
   \end{array}
$ & $\begin{array}{c}
    8 \\ 32
   \end{array}
$ & 3 & 6 & $3A_1$ & $\xymatrix@R=.1cm@C=.1cm{
{\scriptstyle 2} && {\scriptstyle 5} && {\scriptstyle 7}\\
{\bullet} && {\bullet} && {\bullet}
}$ & 30 & 39 & 2 & --  \\
\hline
$D_4(a_1)+A_1$ & $D_4(a_1)+A_1$ & $\begin{array}{c}
    3 \\ 7 \\ 49
   \end{array}
$ & $\begin{array}{c}
    8 \\ 32
   \end{array}
$ & 2 & 4 & $2A_1$ & $\xymatrix@R=.1cm@C=.1cm{
{\scriptstyle 2} && {\scriptstyle 5} \\
{\bullet} && {\bullet}
}$ & 31 & 37 & 2 & --  \\
\hline
$D_4$ & $D_4$ & $\begin{array}{c}
    3 \\ 28 \\ 49
   \end{array}
$ & 1 & 3 & 6 & $3A_1$ & $\xymatrix@R=.1cm@C=.1cm{
{\scriptstyle 2} && {\scriptstyle 5} && {\scriptstyle 7}\\
{\bullet} && {\bullet} && {\bullet}
}$ & 28 & 37 & 1 & --  \\
\hline
$(A_3+A_2)_2$ & $A_3+A_2$ & $\begin{array}{c}
    20\\21\\22
   \end{array}
$ & $\begin{array}{c}
    4\\49
   \end{array}
$ & 2 & 2 & $A_1$ & $\xymatrix@R=.1cm@C=.1cm{
{\scriptstyle 7}\\
{\bullet}
}$ & 33 & 37 & 6 & --  \\
\hline
$A_3+A_2$ & $D_4(a_1)+2A_1$ & $\begin{array}{c}
    3 \\ 7 \\ 49
   \end{array}
$ & $\begin{array}{c}
    8 \\ 2 \\ 32
   \end{array}
$ & 1 & 2 & $A_1$ & $\xymatrix@R=.1cm@C=.1cm{
{\scriptstyle 5} \\
{\bullet}
}$ & 32 & 35 & 2 & --  \\
\hline
$A_3+A_2+A_1$ & $2A_3+A_1$ & $\begin{array}{c}
    4 \\ 7 \\ 16 \\ 49
   \end{array}
$ & $\begin{array}{c}
    3 \\ 20 \\ 21
   \end{array}
$ & 0 & 0 & -- & -- & 33 & 33 & 2 & -- \\
\hline
$A_4$ & $A_4$ & $\begin{array}{c}
    37 \\ 45
   \end{array}
$ & $\begin{array}{c}
    1 \\ 6
   \end{array}
$ & 3 & 6 & $A_2$ & $\xymatrix@R=.1cm@C=.1cm{
{\scriptstyle 2}&&{\scriptstyle 4}\\
{\bullet}\ar@{-}[rr]&& {\bullet}
}$ & 24 & 33 & 5 & --  \\
\hline
$A_5''$ & $A_5'$ & $\begin{array}{c}
    1 \\ 7 \\ 37
   \end{array}
$ & $\begin{array}{c}
    23 \\ 24
   \end{array}
$ & 2 & 6 & $A_2$ & $\xymatrix@R=.1cm@C=.1cm{
{\scriptstyle 4}&&{\scriptstyle 5}\\
{\bullet}\ar@{-}[rr]&& {\bullet}
}$ & 23 & 31 & 3 & --  \\
\hline
$D_4+A_1$ & $D_4+3A_1$ & $\begin{array}{c}
    28\\49\\3
   \end{array}
$ & $\begin{array}{c}
   1\\7\\2\\5
   \end{array}
$ & 0 & 0 & -- & -- & 31 & 31 & 1 & --  \\
\hline
$A_4+A_1$ & $A_4+A_1$ & $\begin{array}{c}
    9 \\ 37 \\ 45
   \end{array}
$ & $\begin{array}{c}
    1 \\ 6
   \end{array}
$ & 2 & 0 & -- & -- & 27 & 29 & 5 & --  \\
\hline
$A_4+A_2$ & $A_4+A_2$ & $\begin{array}{c}
    4 \\ 37 \\ 45
   \end{array}
$ & $\begin{array}{c}
    1 \\ 2 \\ 6
   \end{array}
$ & 1 & 0 & -- & -- & 26 & 27 & 15 & --  \\
\hline
$D_5(a_1)$ & $D_5(a_1)$ & $\begin{array}{c}
    3 \\ 28 \\ 49
   \end{array}
$ & $\begin{array}{c}
    1 \\ 30
   \end{array}
$ & 2 & 2 & $A_1$ & $\xymatrix@R=.1cm@C=.1cm{
{\scriptstyle 5}\\
{\bullet}
}$ & 23 & 27 & 2 & --  \\
\hline
$(A_5+A_1)''$ & $A_5+A_2$ & $\begin{array}{c}
    1 \\ 5 \\ 7 \\ 37
   \end{array}
$ & $\begin{array}{c}
    4 \\ 23 \\ 24
   \end{array}
$ & 0 & 0 & -- & -- & 25 & 25 & 3 & 3  \\
\hline
$A_5'$ & $(A_5+A_1)''$ & $\begin{array}{c}
    4 \\ 16 \\ 37
   \end{array}
$ & $\begin{array}{c}
    1 \\ 12 \\ 13
   \end{array}
$ & 1 & 0 & -- & -- & 24 & 25 & 3 & --  \\
\hline
$D_5(a_1)+A_1$ & $D_5(a_1)+A_1$ & $\begin{array}{c}
    3 \\ 28 \\ 49
   \end{array}
$ & $\begin{array}{c}
    1 \\ 5 \\ 30
   \end{array}
$ & 1 & 0 & -- & -- & 24 & 25 & 2 & --  \\
\hline
$D_6(a_2)$ & $D_6(a_2)+A_1$ & $\begin{array}{c}
   2 \\ 7 \\ 8 \\ 32
   \end{array}
$ & $\begin{array}{c}
    5 \\ 23 \\ 24
   \end{array}
$ & 0 & 0 & -- & -- & 23 & 23 & 1 & --  \\
\hline
$(A_5+A_1)'$ & $E_6(a_2)$ & $\begin{array}{c}
   1 \\ 4 \\ 16
   \end{array}
$ & $\begin{array}{c}
    28 \\ 29 \\ 31
   \end{array}
$ & 1 & 0 & -- & -- & 22 & 23 & 3 & --  \\
\hline
\end{tabular*}

\newpage

\begin{tabular*}{1.03\textwidth}{|@{\extracolsep{\fill} }c|c|c|c|c|c|c|c|c|c|c|c|}
  \hline
 \rule[-.3cm]{0cm}{1cm}{} $\mathcal{O}$ & $\Psi^W(\mathcal{O})$ & $s^1$ & $s^2$ & ${\rm dim}~\h_0$ & $|{\overline{\Delta}}_0|$ & $\overline{\Delta}_0$ & $\Gamma_0$ & $l(s)$ & ${\rm dim}~\Sigma_s$ & d & q  \\ \hline \hline
$D_5$ & $D_5$ & $\begin{array}{c}
    3 \\ 28
   \end{array}
$ & $\begin{array}{c}
    1 \\ 6 \\ 19
   \end{array}
$ & 2 & 2 & $A_1$ & $\xymatrix@R=.1cm@C=.1cm{
{\scriptstyle 2}\\
{\bullet}
}$ & 17 & 21 & 2 & --  \\
\hline
$D_6(a_2)+A_1$ & $E_7(a_4)$ & $\begin{array}{c}
   1 \\ 2 \\ 12 \\ 13
   \end{array}
$ & $\begin{array}{c}
    15 \\ 17 \\ 32
   \end{array}
$ & 0 & 0 & -- & -- & 21 & 21 & 1 & --  \\
\hline
$D_5+A_1$ & $D_5+A_1$ & $\begin{array}{c}
   2 \\ 3 \\ 28
   \end{array}
$ & $\begin{array}{c}
    1 \\ 6 \\ 19
   \end{array}
$ & 1 & 0 & -- & -- & 18 & 19 & 2 & --  \\
\hline
$A_6$ & $A_6$ & $\begin{array}{c}
   11 \\ 19 \\ 26
   \end{array}
$ & $\begin{array}{c}
    6 \\ 9 \\ 10
   \end{array}
$ & 1 & 0 & -- & -- & 18 & 19 & 7 & --  \\
\hline
$D_6(a_1)$ & $D_6(a_1)$ & $\begin{array}{c}
    3 \\ 5 \\ 28
   \end{array}
$ & $\begin{array}{c}
    1 \\ 12 \\ 13
   \end{array}
$ & 1 & 2 & $A_1$ & $\xymatrix@R=.1cm@C=.1cm{
{\scriptstyle 2}\\
{\bullet}
}$ & 16 & 19 & 1 & --  \\
\hline
$D_6(a_1)+A_1$ & $A_7$ & $\begin{array}{c}
  4 \\ 16 \\ 15 \\ 17
   \end{array}
$ & $\begin{array}{c}
    1 \\ 12 \\ 13
   \end{array}
$ & 0 & 0 & -- & -- & 17 & 17 & 1 & --  \\
\hline
$D_6$ & $D_6+A_1$ & $\begin{array}{c}
   1 \\ 6 \\ 9
   \end{array}
$ & $\begin{array}{c}
    7 \\ 10 \\ 11 \\ 22
   \end{array}
$ & 0 & 0 & -- & -- & 15 & 15 & 1 & --  \\
\hline
$E_6(a_1)$ & $E_6(a_1)$ & $\begin{array}{c}
   8 \\ 19 \\ 22
   \end{array}
$ & $\begin{array}{c}
    1 \\ 4 \\ 6
   \end{array}
$ & 1 & 0 & -- & -- & 14 & 15 & 3 & --  \\
\hline
$E_6$ & $E_6$ & $\begin{array}{c}
   3  \\ 6  \\ 19
   \end{array}
$ & $\begin{array}{c}
    1 \\ 9 \\ 11
   \end{array}
$ & 1 & 0 & -- & -- & 12 & 13 & 3 & --  \\
\hline
$D_6+A_1$ & $E_7(a_3)$ & $\begin{array}{c}
   7\\10\\11\\22
   \end{array}
$ & $\begin{array}{c}
    1\\2\\6
   \end{array}
$ & 0 & 0 & -- & -- & 13 & 13 & 1 & --  \\
\hline
$E_7(a_2)$ & $E_7(a_2)$ & $\begin{array}{c}
   2\\3\\12\\13
   \end{array}
$ & $\begin{array}{c}
    1\\9\\11
   \end{array}
$ & 0 & 0 & -- & -- & 11 & 11 & 1 & --  \\
\hline
$E_7(a_1)$ & $E_7(a_1)$ & $\begin{array}{c}
   6\\9\\10
   \end{array}
$ & $\begin{array}{c}
    1\\2\\5\\7
   \end{array}
$ & 0 & 0 & -- & -- & 9 & 9 & 1 & --  \\
\hline
$E_7$ & $E_7$ & $\begin{array}{c}
   1\\4\\6
   \end{array}
$ & $\begin{array}{c}
    2\\3\\5\\7
   \end{array}
$ & 0 & 0 & -- & -- & 7 & 7 & 1 & --  \\
\hline
\end{tabular*}

\newpage

$\bf E_8.$

\vskip 0.1cm

\begin{tabular*}{1.03\textwidth}{|@{\extracolsep{\fill} }c|c|c|c|c|c|c|c|c|c|c|c|}
  \hline
\rule[-.3cm]{0cm}{1cm}{}  $\mathcal{O}$ & $\Psi^W(\mathcal{O})$ & $s^1$ & $s^2$ & ${\rm dim}~\h_0$ & $|{\overline{\Delta}}_0|$ & $\overline{\Delta}_0$ & $\Gamma_0$ & $l(s)$ & ${\rm dim}~\Sigma_s$ & d & q  \\ \hline \hline
$A_1$ & $A_1$ & 120 & -- & 7 & 126 & $E_7$ &
$\xymatrix@R=.1cm@C=.1cm{
{\scriptstyle 1}&&{\scriptstyle 3}&&{\scriptstyle 4}&&{\scriptstyle 5}&&{\scriptstyle 6}&&{\scriptstyle 7}\\
{\bullet}\ar@{-}[rr]&&{\bullet}\ar@{-}[rr]&&{\bullet}\ar@{-}[dd]\ar@{-}[rr]&& {\bullet}\ar@{-}[rr] &&{\bullet}\ar@{-}[rr]&&{\bullet}\\
&&&&&&&&&&\\
&&&&{\bullet}&&&&&&\\
&&&&{\scriptstyle 2}&&&&&&}$
& 57 & 190 & 1 & -- \\
\hline
$2A_1$ & $2A_1$ & 120 & 97 & 6 & 60 & $D_6$ &
$\xymatrix@R=.01cm@C=.1cm{
&&&&&&&{\scriptstyle 2}\\
&&&&&&&{\bullet}\\
{\scriptstyle 7}&&{\scriptstyle 6}&&{\scriptstyle 5}&&{\scriptstyle 4}&\\
{\bullet}\ar@{-}[rr]&&{\bullet}\ar@{-}[rr]&&{\bullet}\ar@{-}[rr]&& {\bullet}\ar@{-}[uur]\ar@{-}[ddr]& \\
&&&&&&&\\
&&&&&&&{\bullet}\\
&&&&&&&{\scriptstyle 3}
}$ & 90 & 156 & 1 & --  \\
\hline
$3A_1$ & $(4A_1)'$ & $\begin{array}{c}
    97 \\ 120
   \end{array}
$ & $\begin{array}{c}
    7 \\ 61
   \end{array}
$ & 4 & 24 & $D_4$ & $\xymatrix@R=.01cm@C=.1cm{
&&&{\scriptstyle 3}\\
&&&{\bullet}\\
{\scriptstyle 2}&&{\scriptstyle 4}&\\
{\bullet}\ar@{-}[rr]&& {\bullet}\ar@{-}[uur]\ar@{-}[ddr]& \\
&&&\\
&&&{\bullet}\\
&&&{\scriptstyle 5}
}$ & 108 & 136 & 1 & --  \\
\hline
$A_2$ & $A_2$ & 119 & 8 & 6 & 72 & $E_6$ &
$\xymatrix@R=.1cm@C=.1cm{
{\scriptstyle 1}&&{\scriptstyle 3}&&{\scriptstyle 4}&&{\scriptstyle 5}&&{\scriptstyle 6}\\
{\bullet}\ar@{-}[rr]&&{\bullet}\ar@{-}[rr]&&{\bullet}\ar@{-}[dd]\ar@{-}[rr]&& {\bullet}\ar@{-}[rr] &&{\bullet}\\
&&&&&&&&\\
&&&&{\bullet}&&&&\\
&&&&{\scriptstyle 2}&&&&}$
 & 56 & 134 & 3 & --  \\
\hline
$4A_1$ & $8A_1$ & $\begin{array}{c}
    74 \\ 118 \\ 104 \\ 32
   \end{array}
$ & $\begin{array}{c}
    2 \\ 3 \\ 5 \\ 8
   \end{array}
$ & 0 & 0 & -- & -- & 120 & 120 & 1 & --  \\
\hline
$A_2+A_1$ & $A_2+A_1$ & $\begin{array}{c}
    69 \\119
   \end{array}
$ & 8 & 5 & 30 & $A_5$ & $\xymatrix@R=.1cm@C=.1cm{
{\scriptstyle 1}&&{\scriptstyle 3}&&{\scriptstyle 4}&&{\scriptstyle 5}&&{\scriptstyle 6}\\
{\bullet}\ar@{-}[rr]&& {\bullet}\ar@{-}[rr] &&{\bullet}\ar@{-}[rr]&& {\bullet}\ar@{-}[rr] &&{\bullet}
}$ & 77 & 112 & 3 & --  \\
\hline
$A_2+2A_1$ & $A_2+2A_1$ & $\begin{array}{c}
    69 \\119
   \end{array}
$ & $\begin{array}{c}
    8 \\ 31
   \end{array}
$ & 4 & 12 & $A_3$ & $\xymatrix@R=.1cm@C=.1cm{
{\scriptstyle 3}&&{\scriptstyle 4}&&{\scriptstyle 5}\\
{\bullet}\ar@{-}[rr] &&{\bullet}\ar@{-}[rr] &&{\bullet}
}$ & 86 & 102 & 3 & --  \\
\hline
$A_3$ & $A_3$ & $\begin{array}{c}
    8 \\ 74
   \end{array}
$ & 97 & 5 & 40 & $D_5$ &
$\xymatrix@R=.01cm@C=.1cm{
&&&&&{\scriptstyle 2}\\
&&&&&{\bullet}\\
{\scriptstyle 6}&&{\scriptstyle 5}&&{\scriptstyle 4}&\\
{\bullet}\ar@{-}[rr]&&{\bullet}\ar@{-}[rr]&& {\bullet}\ar@{-}[uur]\ar@{-}[ddr]& \\
&&&&&\\
&&&&&{\bullet}\\
&&&&&{\scriptstyle 3}
}$
& 55 & 100 & 2 & --  \\
\hline
$A_2+3A_1$ & $A_2+4A_1$ & $\begin{array}{c}
    31 \\69 \\119
   \end{array}
$ & $\begin{array}{c}
    4 \\ 8 \\ 19
   \end{array}
$ & 2 & 0 & -- & -- & 92 & 94 & 3 & --  \\
\hline
$2A_2$ & $2A_2$ & $\begin{array}{c}
    63 \\119
   \end{array}
$ & $\begin{array}{c}
    2 \\ 8
   \end{array}
$ & 4 & 12 & $2A_2$ & $\xymatrix@R=.1cm@C=.1cm{
{\scriptstyle 1}&&{\scriptstyle 3}&&{\scriptstyle 5}&&{\scriptstyle 6}\\
{\bullet}\ar@{-}[rr] &&{\bullet}&&{\bullet}\ar@{-}[rr] &&{\bullet}
}$ & 76 & 92 & 3 & --   \\
\hline
$2A_2+A_1$ & $3A_2$ & $\begin{array}{c}
    6 \\ 63 \\119
   \end{array}
$ & $\begin{array}{c}
    2 \\ 5 \\ 8
   \end{array}
$ & 2 & 6 & $A_2$ & $\xymatrix@R=.1cm@C=.1cm{
{\scriptstyle 1}&&{\scriptstyle 3}\\
{\bullet}\ar@{-}[rr] &&{\bullet}
}$ & 78 & 86 & 3 & 3  \\
\hline
$A_3+A_1$ & $(A_3+2A_1)'$ & $\begin{array}{c}
    8 \\ 48 \\ 74
   \end{array}
$ & $\begin{array}{c}
    6 \\ 97
   \end{array}
$ & 3 & 12 & $A_3$ & $\xymatrix@R=.1cm@C=.1cm{
{\scriptstyle 2}&&{\scriptstyle 4}&&{\scriptstyle 3}\\
{\bullet}\ar@{-}[rr]&& {\bullet}\ar@{-}[rr] &&{\bullet}
}$ & 69 & 84 & 2 & --  \\
\hline
\end{tabular*}

\newpage

\begin{tabular*}{1.03\textwidth}{|@{\extracolsep{\fill} }c|c|c|c|c|c|c|c|c|c|c|c|}
  \hline
 \rule[-.3cm]{0cm}{1cm}{} $\mathcal{O}$ & $\Psi^W(\mathcal{O})$ & $s^1$ & $s^2$ & ${\rm dim}~\h_0$ & $|{\overline{\Delta}}_0|$ & $\overline{\Delta}_0$ & $\Gamma_0$ & $l(s)$ & ${\rm dim}~\Sigma_s$ & d & q  \\ \hline \hline
$D_4(a_1)$ & $D_4(a_1)$ & $\begin{array}{c}
    7 \\ 97
   \end{array}
$ & $\begin{array}{c}
    15 \\ 68
   \end{array}
$ & 4 & 24 & $D_4$ & $\xymatrix@R=.01cm@C=.1cm{
&&&{\scriptstyle 3}\\
&&&{\bullet}\\
{\scriptstyle 2}&&{\scriptstyle 4}&\\
{\bullet}\ar@{-}[rr]&& {\bullet}\ar@{-}[uur]\ar@{-}[ddr]& \\
&&&\\
&&&{\bullet}\\
&&&{\scriptstyle 5}
}$
& 54 & 82 & 2 & --  \\
\hline
$D_4$ & $D_4$ & $\begin{array}{c}
    7 \\ 61 \\ 97
   \end{array}
$ & 8 & 4 & 24 & $D_4$ & $\xymatrix@R=.01cm@C=.1cm{
&&&{\scriptstyle 3}\\
&&&{\bullet}\\
{\scriptstyle 2}&&{\scriptstyle 4}&\\
{\bullet}\ar@{-}[rr]&& {\bullet}\ar@{-}[uur]\ar@{-}[ddr]& \\
&&&\\
&&&{\bullet}\\
&&&{\scriptstyle 5}
}$
& 52 & 80 & 1 & --  \\
\hline
$2A_2+2A_1$ & $4A_2$ & $\begin{array}{c}
  3 \\  5 \\ 63 \\119
   \end{array}
$ & $\begin{array}{c}
   1 \\ 2 \\ 6 \\ 8
   \end{array}
$ & 0 & 0 & -- & -- & 80 & 80 & 3 & 3  \\
\hline
$A_3+2A_1$ & $A_3+4A_1$ & $\begin{array}{c}
    8 \\ 74 \\ 48 \\ 17
   \end{array}
$ & $\begin{array}{c}
    4 \\ 6 \\ 97
   \end{array}
$ & 1 & 0 & -- & -- & 75 & 76 & 2 & --  \\
\hline
$D_4(a_1)+A_1$ & $D_4(a_1)+A_1$ & $\begin{array}{c}
    7 \\ 32 \\97
   \end{array}
$ & $\begin{array}{c}
    15 \\ 68
   \end{array}
$ & 3 & 6 & $3A_1$ & $\xymatrix@R=.1cm@C=.1cm{
{\scriptstyle 2} &&{\scriptstyle 3} && {\scriptstyle 5} \\
{\bullet} &&{\bullet} && {\bullet}
}$ & 63 & 72 & 2 & --  \\
\hline
$(A_3+A_2)_2$ & $A_3+A_2$ & $\begin{array}{c}
    29\\55\\56
   \end{array}
$ & $\begin{array}{c}
    6\\97
   \end{array}
$ & 3 & 4 & $2A_1$ & $\xymatrix@R=.1cm@C=.1cm{
{\scriptstyle 2}&&{\scriptstyle 3}\\
{\bullet}&&{\bullet}
}$ & 65 & 72 & 6 & --  \\
\hline
$A_3+A_2$ & $(2A_3)'$ & $\begin{array}{c}
    13 \\ 22 \\ 40 \\ 62
   \end{array}
$ & $\begin{array}{c}
    7 \\ 97
   \end{array}
$ & 2 & 4 & $2A_1$ & $\xymatrix@R=.1cm@C=.1cm{
{\scriptstyle 2} && {\scriptstyle 3}\\
{\bullet} && {\bullet}
}$ & 64 & 70 & 2 & --  \\
\hline
$A_4$ & $A_4$ & $\begin{array}{c}
    74 \\ 93
   \end{array}
$ & $\begin{array}{c}
    1 \\ 8
   \end{array}
$ & 4 & 20 & $A_4$ & $\xymatrix@R=.1cm@C=.1cm{
{\scriptstyle 2}&&{\scriptstyle 4}&&{\scriptstyle 5}&&{\scriptstyle 6}\\
{\bullet}\ar@{-}[rr]&&{\bullet}\ar@{-}[rr]&&{\bullet}\ar@{-}[rr]&& {\bullet}
}$ & 44 & 68 & 5 & --  \\
\hline
$A_3+A_2+A_1$ & $2A_3+2A_1$ & $\begin{array}{c}
    62 \\ 22 \\ 13 \\ 40
   \end{array}
$ & $\begin{array}{c}
    97 \\ 7 \\ 3 \\ 2
   \end{array}
$ & 0 & 0 & -- & -- & 66 & 66 & 2 & --  \\
\hline
$D_4(a_1)+A_2$ & $D_4(a_1)+A_2$ & $\begin{array}{c}
    7 \\ 25 \\97
   \end{array}
$ & $\begin{array}{c}
    4 \\ 15 \\ 68
   \end{array}
$ & 2 & 0 & -- & -- & 62 & 64 & 6 & --  \\
\hline
$D_4+A_1$ & $D_4+4A_1$ & $\begin{array}{c}
    97\\61\\7\\32
   \end{array}
$ & $\begin{array}{c}
    8\\5\\2\\3
   \end{array}
$ & 0 & 0 & -- & --
& 64 & 64 & 1 & --  \\
\hline
$2A_3$ & $2D_4(a_1)$ & $\begin{array}{c}
    2 \\ 3 \\ 7 \\97
   \end{array}
$ & $\begin{array}{c}
    11 \\ 12 \\ 15 \\ 68
   \end{array}
$ & 0 & 0 & -- & -- & 60 & 60 & 1 & --  \\
\hline
$A_4+A_1$ & $A_4+A_1$ & $\begin{array}{c}
    26\\74\\93
   \end{array}
$ & $\begin{array}{c}
    1 \\ 8
   \end{array}
$ & 3 & 6 & $A_2$ & $\xymatrix@R=.1cm@C=.1cm{
{\scriptstyle 4}&&{\scriptstyle 5}\\
{\bullet}\ar@{-}[rr]&& {\bullet}
}$ & 51 & 60 & 5 & --  \\
\hline
$D_5(a_1)$ & $D_5(a_1)$ & $\begin{array}{c}
    8\\31
   \end{array}
$ & $\begin{array}{c}
    39\\61\\75
   \end{array}
$ & 3 & 12 & $A_3$ & $\xymatrix@R=.1cm@C=.1cm{
{\scriptstyle 3}&&{\scriptstyle 4}&&{\scriptstyle 5}\\
{\bullet}\ar@{-}[rr]&&{\bullet}\ar@{-}[rr]&& {\bullet}
}$ & 43 & 58 & 2 & --  \\
\hline
\end{tabular*}

\begin{tabular*}{1.03\textwidth}{|@{\extracolsep{\fill} }c|c|c|c|c|c|c|c|c|c|c|c|}
  \hline
 \rule[-.3cm]{0cm}{1cm}{} $\mathcal{O}$ & $\Psi^W(\mathcal{O})$ & $s^1$ & $s^2$ & ${\rm dim}~\h_0$ & $|{\overline{\Delta}}_0|$ & $\overline{\Delta}_0$ & $\Gamma_0$ & $l(s)$ & ${\rm dim}~\Sigma_s$ & d & q  \\ \hline \hline
$(D_4+A_2)_2$ & $D_4+A_2$ & $\begin{array}{c}
    49\\59\\63\\71
   \end{array}
$ & $\begin{array}{c}
    2\\8
   \end{array}
$ & 2 & 0 & -- & -- & 54 & 56 & 3 & --  \\
\hline
$A_4+2A_1$ & $A_4+2A_1$ & $\begin{array}{c}
    26\\74\\93
   \end{array}
$ & $\begin{array}{c}
    1 \\ 8 \\12
   \end{array}
$ & 2 & 0 & -- & -- & 54 & 56 & 5 & --  \\
\hline
$A_4+A_2$ & $A_4+A_2$ & $\begin{array}{c}
    18 \\ 74 \\ 93
   \end{array}
$ & $\begin{array}{c}
    1 \\ 6 \\ 8
   \end{array}
$ & 2 & 2 & $A_1$ & $\xymatrix@R=.1cm@C=.1cm{
{\scriptstyle 4}\\
{\bullet}
}$ & 50 & 54 & 15 & --  \\
\hline
$A_4+A_2+A_1$ & $A_4+A_2+A_1$ & $\begin{array}{c}
    18 \\ 74 \\ 93
   \end{array}
$ & $\begin{array}{c}
    1 \\ 4 \\ 6 \\ 8
   \end{array}
$ & 1 & 0 & -- & -- & 51 & 52 & 15 & --  \\
\hline
$A_5$ & $(A_5+A_1)'$ & $\begin{array}{c}
    22\\61\\62
   \end{array}
$ & $\begin{array}{c}
    7\\23 \\ 24
   \end{array}
$ & 2 & 6 & $A_2$ & $\xymatrix@R=.1cm@C=.1cm{
{\scriptstyle 3}&&{\scriptstyle 4}\\
{\bullet}\ar@{-}[rr]&& {\bullet}
}$ & 44 & 52 & 3 & --  \\
\hline
$D_5(a_1)+A_1$ & $D_5(a_1)+A_1$ & $\begin{array}{c}
    8\\19\\31
   \end{array}
$ & $\begin{array}{c}
    39\\61\\75
   \end{array}
$ & 2 & 2 & $A_1$ & $\xymatrix@R=.1cm@C=.1cm{
{\scriptstyle 4}\\
{\bullet}
}$ & 48 & 52 & 2 & --  \\
\hline
$D_4+A_2$ & $D_4+A_3$ & $\begin{array}{c}
    31\\46\\61\\70
   \end{array}
$ & $\begin{array}{c}
    8\\10\\25
   \end{array}
$ & 1 & 0 & -- & --
& 49 & 50 & 2 & --  \\
\hline
$(A_5+A_1)''$ & $E_6(a_2)$ & $\begin{array}{c}
   1 \\ 8\\44
   \end{array}
$ & $\begin{array}{c}
    34\\35\\68
   \end{array}
$ & 2 & 6 & $A_2$ & $\xymatrix@R=.1cm@C=.1cm{
{\scriptstyle 4}&&{\scriptstyle 5}\\
{\bullet}\ar@{-}[rr]&& {\bullet}
}$ & 42 & 50 & 3 & --  \\
\hline
$D_5$ & $D_5$ & $\begin{array}{c}
    1\\8\\44
   \end{array}
$ & $\begin{array}{c}
    7\\61
   \end{array}
$ & 3 & 13 & $A_3$ & $\xymatrix@R=.1cm@C=.1cm{
{\scriptstyle 2}&&{\scriptstyle 4}&&{\scriptstyle 5}\\
{\bullet}\ar@{-}[rr]&&{\bullet}\ar@{-}[rr]&& {\bullet}
}$
& 33 & 48 & 2 & --  \\
\hline
$A_4+A_3$ & $2A_4$ & $\begin{array}{c}
   4\\6\\74\\93
   \end{array}
$ & $\begin{array}{c}
    1\\2\\5\\8
   \end{array}
$ & 0 & 0 & -- & -- & 48 & 48 & 5 & 5  \\
\hline
$D_5(a_1)+A_2$ & $D_5(a_1)+A_3$ & $\begin{array}{c}
    4\\8\\17\\18
   \end{array}
$ & $\begin{array}{c}
   53 \\64\\61\\31
   \end{array}
$ & 0 & 0 & -- & -- & 46 & 46 & 2 & --  \\
\hline
$(A_5+A_1)'$ & $A_5+A_2+A_1$ & $\begin{array}{c}
    3\\22\\61\\62
   \end{array}
$ & $\begin{array}{c}
    4\\7\\23\\24
   \end{array}
$ & 0 & 0 & -- & -- & 46 & 46 & 3 & 3 \\
\hline
$D_6(a_2)$ & $2D_4$ & $\begin{array}{c}
    8\\4\\19\\31
   \end{array}
$ & $\begin{array}{c}
    59\\58\\61\\2
   \end{array}
$ & 0 & 0 & -- & --
& 44 & 44 & 1 & --  \\
\hline
$A_5+2A_1$ & $E_6(a_2)+A_2$ & $\begin{array}{c}
   1\\4\\8\\44
   \end{array}
$ & $\begin{array}{c}
    5\\34\\35\\68
   \end{array}
$ & 0 & 0 & -- & -- & 44 & 44 & 3 & 3  \\
\hline
$A_5+A_2$ & $E_7(a_4)+A_1$ & $\begin{array}{c}
   8\\37\\9\\5
   \end{array}
$ & $\begin{array}{c}
    59\\61\\58\\2
   \end{array}
$ & 0 & 0 & -- & -- & 42 & 42 & 1 & --  \\
\hline

\end{tabular*}

\newpage

\begin{tabular*}{1.03\textwidth}{|@{\extracolsep{\fill} }c|c|c|c|c|c|c|c|c|c|c|c|}
  \hline
 \rule[-.3cm]{0cm}{1cm}{} $\mathcal{O}$ & $\Psi^W(\mathcal{O})$ & $s^1$ & $s^2$ & ${\rm dim}~\h_0$ & $|{\overline{\Delta}}_0|$ & $\overline{\Delta}_0$ & $\Gamma_0$ & $l(s)$ & ${\rm dim}~\Sigma_s$ & d & q  \\ \hline \hline
$D_5+A_1$ & $D_5+2A_1$ & $\begin{array}{c}
   8\\1\\44\\18
   \end{array}
$ & $\begin{array}{c}
    4\\7\\61
   \end{array}
$ & 1 & 0 & -- & -- & 39 & 40 & 2 & --  \\
\hline
$2A_4$ & $E_8(a_8)$ & $\begin{array}{c}
   2\\3\\6\\8
   \end{array}
$ & $\begin{array}{c}
    19\\41\\59\\76
   \end{array}
$ & 0 & 0 & -- & -- & 40 & 40 & 1 & --  \\
\hline
$D_6(a_1)$ & $D_6(a_1)$ & $\begin{array}{c}
    3\\7\\61
   \end{array}
$ & $\begin{array}{c}
    8\\9\\37
   \end{array}
$ & 2 & 4 & $2A_1$ & $\xymatrix@R=.1cm@C=.1cm{
{\scriptstyle 2}&&{\scriptstyle 5}\\
{\bullet}&&{\bullet}
}$ & 32 & 38 & 2 & --  \\
\hline
$A_6$ & $A_6$ & $\begin{array}{c}
   40 \\ 44 \\ 62
   \end{array}
$ & $\begin{array}{c}
    1 \\ 13 \\ 14
   \end{array}
$ & 2 & 2 & $A_1$ & $\xymatrix@R=.1cm@C=.1cm{
{\scriptstyle 2}\\
{\bullet}
}$ & 34 & 38 & 7 & --  \\
\hline
$A_6+A_1$ & $A_6+A_1$ & $\begin{array}{c}
  2 \\ 40 \\ 44 \\ 62
   \end{array}
$ & $\begin{array}{c}
    1 \\ 13 \\ 14
   \end{array}
$ & 1 & 0 & -- & -- & 35 & 36 & 7 & --  \\
\hline
$D_6(a_1)+A_1$ & $A_7'$ & $\begin{array}{c}
  26 \\ 28 \\ 41 \\ 27
   \end{array}
$ & $\begin{array}{c}
    8 \\ 9 \\ 37
   \end{array}
$ & 1 & 2 & $A_1$ & $\xymatrix@R=.1cm@C=.1cm{
{\scriptstyle 5}\\
{\bullet}
}$ & 33 & 36 & 2 & --  \\
\hline
$(D_5+A_2)_2$ & $D_5+A_2$ & $\begin{array}{c}
    8\\23\\24\\25
   \end{array}
$ & $\begin{array}{c}
    4\\7\\61
   \end{array}
$ & 1 & 0 & -- & -- & 35 & 36 & 6 & --  \\
\hline
$D_5+A_2$ & $A_7+A_1$ & $\begin{array}{c}
  26 \\ 28 \\ 41 \\ 27
   \end{array}
$ & $\begin{array}{c}
   5 \\ 8 \\ 9 \\ 37
   \end{array}
$ & 0 & 0 & -- & -- & 34 & 34 & 2 & --  \\
\hline
$E_6(a_1)$ & $E_6(a_1)$ & $\begin{array}{c}
   15\\44\\55
   \end{array}
$ & $\begin{array}{c}
    1 \\ 6\\8
   \end{array}
$ & 2 & 6 & $A_2$ & $\xymatrix@R=.1cm@C=.1cm{
{\scriptstyle 2}&&{\scriptstyle 4}\\
{\bullet}\ar@{-}[rr]&& {\bullet}
}$ & 26 & 34 & 3 & --  \\
\hline
$D_6$ & $D_6+2A_1$ & $\begin{array}{c}
   3\\49\\21\\48
   \end{array}
$ & $\begin{array}{c}
    1\\2\\6\\8
   \end{array}
$ & 0 & 0 & -- & -- & 32 & 32 & 1 & --  \\
\hline
$D_7(a_2)$ & $D_7(a_2)$ & $\begin{array}{c}
   4\\7\\61
   \end{array}
$ & $\begin{array}{c}
    3\\8\\16\\30
   \end{array}
$ & 1 & 0 & -- & -- & 31 & 32 & 2 & --  \\
\hline
$E_6$ & $E_6$ & $\begin{array}{c}
   1\\7\\44
   \end{array}
$ & $\begin{array}{c}
    8\\26\\27
   \end{array}
$ & 2 & 6 & $A_2$ & $\xymatrix@R=.1cm@C=.1cm{
{\scriptstyle 4}&&{\scriptstyle 5}\\
{\bullet}\ar@{-}[rr]&& {\bullet}
}$ & 24 & 32 & 3 & --  \\
\hline
$(A_7)_3$ & $A_7''$ & $\begin{array}{c}
    21\\22\\26\\27
   \end{array}
$ & $\begin{array}{c}
    16\\30\\32
   \end{array}
$ & 1 & 0 & -- & -- & 31 & 32 & 4 & --  \\
\hline
$A_7$ & $D_8(a_3)$ & $\begin{array}{c}
   3\\5\\7\\61
   \end{array}
$ & $\begin{array}{c}
    4\\8\\23\\24
   \end{array}
$ & 0 & 0 & -- & -- & 30 & 30 & 1 & --  \\
\hline
\end{tabular*}

\newpage

\begin{tabular*}{1.03\textwidth}{|@{\extracolsep{\fill} }c|c|c|c|c|c|c|c|c|c|c|c|}
  \hline
 \rule[-.3cm]{0cm}{1cm}{} $\mathcal{O}$ & $\Psi^W(\mathcal{O})$ & $s^1$ & $s^2$ & ${\rm dim}~\h_0$ & $|{\overline{\Delta}}_0|$ & $\overline{\Delta}_0$ & $\Gamma_0$ & $l(s)$ & ${\rm dim}~\Sigma_s$ & d & q  \\ \hline \hline
$E_6(a_1)+A_1$ & $E_6(a_1)+A_1$ & $\begin{array}{c}
   15\\44\\55
   \end{array}
$ & $\begin{array}{c}
    1 \\ 6\\8\\10
   \end{array}
$ & 1 & 0 & -- & -- & 29 & 30 & 3 & --  \\
\hline
$D_8(a_3)$ & $A_8$ & $\begin{array}{c}
   22\\23\\26\\31
   \end{array}
$ & $\begin{array}{c}
    7\\11\\12\\25
   \end{array}
$ & 0 & 0 & -- & -- & 28 & 28 & 3 & 3  \\
\hline
$D_6+A_1$ & $E_7(a_3)$ & $\begin{array}{c}
   27\\28\\32\\41
   \end{array}
$ & $\begin{array}{c}
    1\\2\\8
   \end{array}
$ & 1 & 2 & $A_1$ & $\xymatrix@R=.1cm@C=.1cm{
{\scriptstyle 5}\\
{\bullet}
}$
& 25 & 28 & 1 & --  \\
\hline
$(D_7(a_1))_2$ & $D_7(a_1)$ & $\begin{array}{c}
   48\\3\\49\\21
   \end{array}
$ & $\begin{array}{c}
    1\\8\\26
   \end{array}
$ & 1 & 0 & -- & -- & 27 & 28 & 2 & --  \\
\hline
$D_7(a_1)$ & $D_8(a_2)$ & $\begin{array}{c}
   1\\6\\8
   \end{array}
$ & $\begin{array}{c}
    12\\21\\25\\27\\49
   \end{array}
$ & 0 & 0 & -- & -- & 26 & 26 & 1 & --  \\
\hline
$E_6+A_1$ & $E_6+A_2$ & $\begin{array}{c}
   1\\5\\7\\44
   \end{array}
$ & $\begin{array}{c}
    4\\8\\26\\27
   \end{array}
$ & 0 & 0 & -- & -- & 26 & 26 & 3 & 3  \\
\hline
$E_7(a_2)$ & $E_7(a_2)+A_1$ & $\begin{array}{c}
   7\\18\\23\\24
   \end{array}
$ & $\begin{array}{c}
    4\\8\\26\\27
   \end{array}
$ & 0 & 0 & -- & -- & 24 & 24 & 1 & --  \\
\hline
$A_8$ & $E_8(a_6)$ & $\begin{array}{c}
   1\\2\\5\\8
   \end{array}
$ & $\begin{array}{c}
    21\\25\\27\\49
   \end{array}
$ & 0 & 0 & -- & -- & 24 & 24 & 1 & --  \\
\hline
$D_7$ & $D_8(a_1)$ & $\begin{array}{c}
    7\\17\\18\\19
   \end{array}
$ & $\begin{array}{c}
    1\\20\\21\\22
   \end{array}
$ & 0 & 0 & -- & -- & 22 & 22 & 1 & --  \\
\hline
$E_7(a_2)+A_1$ & $E_8(a_7)$ & $\begin{array}{c}
   4\\7\\23\\24
   \end{array}
$ & $\begin{array}{c}
    5\\8\\20\\33
   \end{array}
$ & 0 & 0 & -- & -- & 22 & 22 & 1 & --  \\
\hline
$E_7(a_1)$ & $E_7(a_1)$ & $\begin{array}{c}
   1\\13\\14
   \end{array}
$ & $\begin{array}{c}
    3\\5\\8\\32
   \end{array}
$ & 1 & 2 & $A_1$ & $\xymatrix@R=.1cm@C=.1cm{
{\scriptstyle 2}\\
{\bullet}
}$ & 17 & 20 & 1 & --  \\
\hline
$D_8(a_1)$ & $E_8(a_3)$ & $\begin{array}{c}
   3\\7\\23\\24
   \end{array}
$ & $\begin{array}{c}
    4\\8\\26\\27
   \end{array}
$ & 0 & 0 & -- & -- & 20 & 20 & 1 & --  \\
\hline
$E_7(a_1)+A_1$ & $D_8$ & $\begin{array}{c}
   4\\8\\17\\26\\27
   \end{array}
$ & $\begin{array}{c}
    1\\5\\7
   \end{array}
$ & 0 & 0 & -- & -- & 18 & 18 & 1 & --  \\
\hline
\end{tabular*}

\newpage

\begin{tabular*}{1.03\textwidth}{|@{\extracolsep{\fill} }c|c|c|c|c|c|c|c|c|c|c|c|}
  \hline
\rule[-.3cm]{0cm}{1cm}{}  $\mathcal{O}$ & $\Psi^W(\mathcal{O})$ & $s^1$ & $s^2$ & ${\rm dim}~\h_0$ & $|{\overline{\Delta}}_0|$ & $\overline{\Delta}_0$ & $\Gamma_0$ & $l(s)$ & ${\rm dim}~\Sigma_s$ & d & q  \\ \hline \hline
$D_8$ & $E_8(a_5)$ & $\begin{array}{c}
   10\\16\\20\\22
   \end{array}
$ & $\begin{array}{c}
    2\\3\\5\\7
   \end{array}
$ & 0 & 0 & -- & -- & 16 & 16 & 1 & --  \\
\hline
$E_7$ & $E_7+A_1$ & $\begin{array}{c}
   1\\6\\8\\10
   \end{array}
$ & $\begin{array}{c}
    7\\11\\12\\25
   \end{array}
$ & 0 & 0 & -- & -- & 16 & 16 & 1 & --  \\
\hline
$E_7+A_1$ & $E_8(a_4)$ & $\begin{array}{c}
   7\\11\\12\\25
   \end{array}
$ & $\begin{array}{c}
    1\\2\\6\\8
   \end{array}
$ & 0 & 0 & -- & -- & 14 & 14 & 1 & --  \\
\hline
$E_8(a_2)$ & $E_8(a_2)$ & $\begin{array}{c}
   2\\3\\5\\7
   \end{array}
$ & $\begin{array}{c}
    1\\8\\10\\20
   \end{array}
$ & 0 & 0 & -- & -- & 12 & 12 & 1 & --  \\
\hline
$E_8(a_1)$ & $E_8(a_1)$ & $\begin{array}{c}
   6\\8\\10\\11
   \end{array}
$ & $\begin{array}{c}
    1\\2\\5\\7
   \end{array}
$ & 0 & 0 & -- & -- & 10 & 10 & 1 & --  \\
\hline
$E_8$ & $E_8$ & $\begin{array}{c}
   1\\4\\6\\8
   \end{array}
$ & $\begin{array}{c}
    2\\3\\5\\7
   \end{array}
$ & 0 & 0 & -- & -- & 8 & 8 & 1 & --  \\
\hline
\end{tabular*}


\setcounter{equation}{0}
\setcounter{theorem}{0}

\section*{Appendix 2. Irreducible root systems of exceptional types.}

In this Appendix we give the lists of positive roots in irreducible root systems of exceptional types. All simple roots are numbered as shown at the Dynkin diagrams. The other roots in each list are given in terms of their coordinates with respect to the basis of simple roots. The coordinates are indicated in the brackets ( ). Each set of coordinates is preceded by the number of the corresponding root. These numbers are used to indicate roots which appear in the columns $s^1$, $s^2$ and $\Gamma_0$ in the tables in Appendix 1.

\vskip 0.3cm

$\bf G_2.$

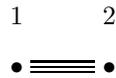
$$\xymatrix@R=.25cm{
1&2\\
{\bullet}\ar@3{-}[r] &{\bullet}
}$$

\begin{enumerate}[1~~]
\item
    (1 0)
\item
    (0 1)
\item
    (1 1)
\item
    (2 1)
\item
    (3 1)
\item
    (3 2)
\end{enumerate}

\vskip 1cm

$\bf F4.$

$$\xymatrix@R=.25cm{
1&2&3&4\\
{\bullet}\ar@{-}[r]&{\bullet}\ar@2{-}[r]
&{\bullet}\ar@{-}[r]&{\bullet} }$$

\begin{multicols}{3}
\begin{enumerate}[1~~]
\item
    (1 0 0 0)
\item
    (0 1 0 0)
\item
    (0 0 1 0)
\item
    (0 0 0 1)
\item
    (1 1 0 0)
\item
    (0 1 1 0)
\item
    (0 0 1 1)
\item
    (1 1 1 0)
\item
    (0 1 2 0)
\item
    (0 1 1 1)
\item
    (1 1 2 0)
\item
    (1 1 1 1)
\item
    (0 1 2 1)
\item
    (1 2 2 0)
\item
    (1 1 2 1)
\item
    (0 1 2 2)
\item
    (1 2 2 1)
\item
    (1 1 2 2)
\item
    (1 2 3 1)
\item
    (1 2 2 2)
\item
    (1 2 3 2)
\item
    (1 2 4 2)
\item
    (1 3 4 2)
\item
    (2 3 4 2)
\end{enumerate}
\end{multicols}

\vskip 1cm

$\bf E_6.$

$$\xymatrix@R=.25cm@C=.25cm{
1&&3&&4&&5&&6\\
{\bullet}\ar@{-}[rr]&&{\bullet}\ar@{-}[rr]&&{\bullet}
\ar@{-}[dd]\ar@{-}[rr]&& {\bullet}\ar@{-}[rr] &&{\bullet}
\\ &&&&&&&&\\
&&&&{\bullet}&&&& \\
&&&&2&&&&}$$

\begin{multicols}{3}
\begin{enumerate}[1~~]
\item
   (1 0 0 0 0 0)
\item
   (0 1 0 0 0 0)
\item
   (0 0 1 0 0 0)
\item
   (0 0 0 1 0 0)
\item
   (0 0 0 0 1 0)
\item
   (0 0 0 0 0 1)
\item
   (1 0 1 0 0 0)
\item
   (0 1 0 1 0 0)
\item
   (0 0 1 1 0 0)
\item
   (0 0 0 1 1 0)
\item
   (0 0 0 0 1 1)
\item
   (1 0 1 1 0 0)
\item
   (0 1 1 1 0 0)
\item
   (0 1 0 1 1 0)
\item
   (0 0 1 1 1 0)
\item
   (0 0 0 1 1 1)
\item
   (1 1 1 1 0 0)
\item
   (1 0 1 1 1 0)
\item
   (0 1 1 1 1 0)
\item
   (0 1 0 1 1 1)
\item
   (0 0 1 1 1 1)
\item
   (1 1 1 1 1 0)
\item
   (1 0 1 1 1 1)
\item
   (0 1 1 2 1 0)
\item
   (0 1 1 1 1 1)
\item
   (1 1 1 2 1 0)
\item
   (1 1 1 1 1 1)
\item
   (0 1 1 2 1 1)
\item
   (1 1 2 2 1 0)
\item
   (1 1 1 2 1 1)
\item
   (0 1 1 2 2 1)
\item
   (1 1 2 2 1 1)
\item
   (1 1 1 2 2 1)
\item
   (1 1 2 2 2 1)
\item
   (1 1 2 3 2 1)
\item
   (1 2 2 3 2 1)
\end{enumerate}
\end{multicols}

\vskip 1cm

$\bf E_7.$

$$\xymatrix@R=.25cm@C=.25cm{
1&&3&&4&&5&&6&&7\\
{\bullet}\ar@{-}[rr]&&{\bullet}\ar@{-}[rr]&&{\bullet}
\ar@{-}[dd]\ar@{-}[rr]&& {\bullet}\ar@{-}[rr] &&{\bullet}\ar@{-}[rr]&&{\bullet}
\\ &&&&&&&&&&\\
&&&&{\bullet}&&&&&& \\
&&&&2&&&&&&}$$

\begin{multicols}{3}
\begin{enumerate}[1~~]
\item
    (1 0 0 0 0 0 0)
\item
    (0 1 0 0 0 0 0)
\item
    (0 0 1 0 0 0 0)
\item
    (0 0 0 1 0 0 0)
\item
    (0 0 0 0 1 0 0)
\item
    (0 0 0 0 0 1 0)
\item
    (0 0 0 0 0 0 1)
\item
    (1 0 1 0 0 0 0)
\item
    (0 1 0 1 0 0 0)
\item
    (0 0 1 1 0 0 0)
\item
    (0 0 0 1 1 0 0)
\item
    (0 0 0 0 1 1 0)
\item
    (0 0 0 0 0 1 1)
\item
    (1 0 1 1 0 0 0)
\item
    (0 1 1 1 0 0 0)
\item
    (0 1 0 1 1 0 0)
\item
    (0 0 1 1 1 0 0)
\item
    (0 0 0 1 1 1 0)
\item
    (0 0 0 0 1 1 1)
\item
    (1 1 1 1 0 0 0)
\item
    (1 0 1 1 1 0 0)
\item
    (0 1 1 1 1 0 0)
\item
    (0 1 0 1 1 1 0)
\item
    (0 0 1 1 1 1 0)
\item
    (0 0 0 1 1 1 1)
\item
    (1 1 1 1 1 0 0)
\item
    (1 0 1 1 1 1 0)
\item
    (0 1 1 2 1 0 0)
\item
    (0 1 1 1 1 1 0)
\item
    (0 1 0 1 1 1 1)
\item
    (0 0 1 1 1 1 1)
\item
    (1 1 1 2 1 0 0)
\item
    (1 1 1 1 1 1 0)
\item
    (1 0 1 1 1 1 1)
\item
    (0 1 1 2 1 1 0)
\item
    (0 1 1 1 1 1 1)
\item
    (1 1 2 2 1 0 0)
\item
    (1 1 1 2 1 1 0)
\item
    (1 1 1 1 1 1 1)
\item
    (0 1 1 2 2 1 0)
\item
    (0 1 1 2 1 1 1)
\item
    (1 1 2 2 1 1 0)
\item
    (1 1 1 2 2 1 0)
\item
    (1 1 1 2 1 1 1)
\item
    (0 1 1 2 2 1 1)
\item
    (1 1 2 2 2 1 0)
\item
    (1 1 2 2 1 1 1)
\item
    (1 1 1 2 2 1 1)
\item
    (0 1 1 2 2 2 1)
\item
    (1 1 2 3 2 1 0)
\item
    (1 1 2 2 2 1 1)
\item
    (1 1 1 2 2 2 1)
\item
    (1 2 2 3 2 1 0)
\item
    (1 1 2 3 2 1 1)
\item
    (1 1 2 2 2 2 1)
\item
    (1 2 2 3 2 1 1)
\item
    (1 1 2 3 2 2 1)
\item
    (1 2 2 3 2 2 1)
\item
    (1 1 2 3 3 2 1)
\item
    (1 2 2 3 3 2 1)
\item
    (1 2 2 4 3 2 1)
\item
    (1 2 3 4 3 2 1)
\item
    (2 2 3 4 3 2 1)
\end{enumerate}
\end{multicols}

\vskip 1cm

$\bf E_8.$

$$\xymatrix@R=.25cm@C=.25cm{
1&&3&&4&&5&&6&&7&&8\\
{\bullet}\ar@{-}[rr]&&{\bullet}\ar@{-}[rr]&&{\bullet}
\ar@{-}[dd]\ar@{-}[rr]&& {\bullet}\ar@{-}[rr] &&{\bullet}\ar@{-}[rr]&&{\bullet}\ar@{-}[rr]&&{\bullet}
\\ &&&&&&&&&&&&\\
&&&&{\bullet}&&&&&&&& \\
&&&&2&&&&&&&&}$$

\begin{multicols}{3}
\begin{enumerate}[1~~]
\item
    (1 0 0 0 0 0 0 0)
\item
    (0 1 0 0 0 0 0 0)
\item
    (0 0 1 0 0 0 0 0)
\item
    (0 0 0 1 0 0 0 0)
\item
    (0 0 0 0 1 0 0 0)
\item
    (0 0 0 0 0 1 0 0)
\item
    (0 0 0 0 0 0 1 0)
\item
    (0 0 0 0 0 0 0 1)
\item
    (1 0 1 0 0 0 0 0)
\item
    (0 1 0 1 0 0 0 0)
\item
    (0 0 1 1 0 0 0 0)
\item
    (0 0 0 1 1 0 0 0)
\item
    (0 0 0 0 1 1 0 0)
\item
    (0 0 0 0 0 1 1 0)
\item
    (0 0 0 0 0 0 1 1)
\item
    (1 0 1 1 0 0 0 0)
\item
    (0 1 1 1 0 0 0 0)
\item
    (0 1 0 1 1 0 0 0)
\item
    (0 0 1 1 1 0 0 0)
\item
    (0 0 0 1 1 1 0 0)
\item
    (0 0 0 0 1 1 1 0)
\item
    (0 0 0 0 0 1 1 1)
\item
    (1 1 1 1 0 0 0 0)
\item
    (1 0 1 1 1 0 0 0)
\item
    (0 1 1 1 1 0 0 0)
\item
    (0 1 0 1 1 1 0 0)
\item
    (0 0 1 1 1 1 0 0)
\item
    (0 0 0 1 1 1 1 0)
\item
    (0 0 0 0 1 1 1 1)
\item
    (1 1 1 1 1 0 0 0)
\item
    (1 0 1 1 1 1 0 0)
\item
    (0 1 1 2 1 0 0 0)
\item
    (0 1 1 1 1 1 0 0)
\item
    (0 1 0 1 1 1 1 0)
\item
    (0 0 1 1 1 1 1 0)
\item
    (0 0 0 1 1 1 1 1)
\item
    (1 1 1 2 1 0 0 0)
\item
    (1 1 1 1 1 1 0 0)
\item
    (1 0 1 1 1 1 1 0)
\item
    (0 1 1 2 1 1 0 0)
\item
    (0 1 1 1 1 1 1 0)
\item
    (0 1 0 1 1 1 1 1)
\item
    (0 0 1 1 1 1 1 1)
\item
    (1 1 2 2 1 0 0 0)
\item
    (1 1 1 2 1 1 0 0)
\item
    (1 1 1 1 1 1 1 0)
\item
    (1 0 1 1 1 1 1 1)
\item
    (0 1 1 2 2 1 0 0)
\item
    (0 1 1 2 1 1 1 0)
\item
    (0 1 1 1 1 1 1 1)
\item
    (1 1 2 2 1 1 0 0)
\item
    (1 1 1 2 2 1 0 0)
\item
    (1 1 1 2 1 1 1 0)
\item
    (1 1 1 1 1 1 1 1)
\item
    (0 1 1 2 2 1 1 0)
\item
    (0 1 1 2 1 1 1 1)
\item
    (1 1 2 2 2 1 0 0)
\item
    (1 1 2 2 1 1 1 0)
\item
    (1 1 1 2 2 1 1 0)
\item
    (1 1 1 2 1 1 1 1)
\item
    (0 1 1 2 2 2 1 0)
\item
    (0 1 1 2 2 1 1 1)
\item
    (1 1 2 3 2 1 0 0)
\item
    (1 1 2 2 2 1 1 0)
\item
    (1 1 2 2 1 1 1 1)
\item
    (1 1 1 2 2 2 1 0)
\item
    (1 1 1 2 2 1 1 1)
\item
    (0 1 1 2 2 2 1 1)
\item
    (1 2 2 3 2 1 0 0)
\item
    (1 1 2 3 2 1 1 0)
\item
    (1 1 2 2 2 2 1 0)
\item
    (1 1 2 2 2 1 1 1)
\item
    (1 1 1 2 2 2 1 1)
\item
    (0 1 1 2 2 2 2 1)
\item
    (1 2 2 3 2 1 1 0)
\item
    (1 1 2 3 2 2 1 0)
\item
    (1 1 2 3 2 1 1 1)
\item
    (1 1 2 2 2 2 1 1)
\item
    (1 1 1 2 2 2 2 1)
\item
    (1 2 2 3 2 2 1 0)
\item
    (1 2 2 3 2 1 1 1)
\item
    (1 1 2 3 3 2 1 0)
\item
    (1 1 2 3 2 2 1 1)
\item
    (1 1 2 2 2 2 2 1)
\item
    (1 2 2 3 3 2 1 0)
\item
    (1 2 2 3 2 2 1 1)
\item
    (1 1 2 3 3 2 1 1)
\item
    (1 1 2 3 2 2 2 1)
\item
    (1 2 2 4 3 2 1 0)
\item
    (1 2 2 3 3 2 1 1)
\item
    (1 2 2 3 2 2 2 1)
\item
    (1 1 2 3 3 2 2 1)
\item
    (1 2 3 4 3 2 1 0)
\item
    (1 2 2 4 3 2 1 1)
\item
    (1 2 2 3 3 2 2 1)
\item
    (1 1 2 3 3 3 2 1)
\item
    (2 2 3 4 3 2 1 0)
\item
    (1 2 3 4 3 2 1 1)
\item
    (1 2 2 4 3 2 2 1)
\item
    (1 2 2 3 3 3 2 1)
\item
    (2 2 3 4 3 2 1 1)
\item
    (1 2 3 4 3 2 2 1)
\item
    (1 2 2 4 3 3 2 1)
\item
    (2 2 3 4 3 2 2 1)
\item
    (1 2 3 4 3 3 2 1)
\item
    (1 2 2 4 4 3 2 1)
\item
    (2 2 3 4 3 3 2 1)
\item
    (1 2 3 4 4 3 2 1)
\item
    (2 2 3 4 4 3 2 1)
\item
    (1 2 3 5 4 3 2 1)
\item
    (2 2 3 5 4 3 2 1)
\item
    (1 3 3 5 4 3 2 1)
\item
    (2 3 3 5 4 3 2 1)
\item
    (2 2 4 5 4 3 2 1)
\item
    (2 3 4 5 4 3 2 1)
\item
    (2 3 4 6 4 3 2 1)
\item
    (2 3 4 6 5 3 2 1)
\item
    (2 3 4 6 5 4 2 1)
\item
    (2 3 4 6 5 4 3 1)
\item
    (2 3 4 6 5 4 3 2)
\end{enumerate}
\end{multicols}

\end{document}